\documentclass[12pt]{amsart}
\usepackage[dvipsnames]{xcolor}
\usepackage{amssymb,amsthm,graphicx,mathrsfs,url,latexsym,enumerate}
\usepackage{a4wide}
\usepackage{bm}
\usepackage{ifthen}
\usepackage{verbatim}
\usepackage{fancyhdr}
\usepackage{url}
\usepackage{bbm}
\usepackage[sans]{dsfont}
\usepackage{transparent}
\usepackage{amsbsy}
\usepackage[colorlinks=true,linkcolor=Red,citecolor=Green]{hyperref}
\usepackage{wrapfig}
\usepackage{tikz}

\usetikzlibrary{decorations.pathreplacing}

\usepackage[top=2.5cm, bottom=2cm, hmargin=1.8cm]{geometry}

\definecolor{bleu}{RGB}{27,88,145}
\definecolor{mauve}{RGB}{138,20,79}

\renewcommand{\Im}{\operatorname{Im}}
\renewcommand{\Re}{\operatorname{Re}}

\newcommand{\diag}{\operatorname{diag}}

\newcommand{\E}{\mathbb E}
\renewcommand{\P}{\mathbb P}
\newcommand{\C}{\mathbb C}
\newcommand{\N}{\mathbb N}
\newcommand{\R}{\mathbb R}
\newcommand{\indic}{\operatorname{1\negthinspace l}}

\newcommand{\Ran}{\operatorname{Ran}}

\renewcommand{\div}{\operatorname{div}}
\def\<{\langle}
\def\>{\rangle}

\newcommand{\be}{\begin{equation}}
\newcommand{\ee}{\end{equation}}
\newcommand{\bes}{\begin{equation*}}
\newcommand{\ees}{\end{equation*}}

\numberwithin{equation}{section}
\numberwithin{figure}{section}

\newcommand{\vvert}[1]{\left\lVert#1\right\rVert}
\newcommand{\lb}{\operatorname{[\negthinspace [}}
\newcommand{\rb}{\operatorname{]\negthinspace ]}}

\newcommand{\Ke}{K^{\bm \varepsilon}}

\newtheorem{theorem}{Theorem}

\newtheorem{defin}{Definition}[section]
\newtheorem{corollary}[defin]{Corollary}
\newtheorem{lemma}[defin]{Lemma}

\newtheorem{proposition}[defin]{Proposition}
\newtheorem{remark}[defin]{Remark}

\def\ccc{{\mathcal C}}\def\ddd{{\mathcal D}}

\def\nnn{{\mathcal N}}

\def\vvv{{\mathcal V}}

\def\Gr{{\mathscr G}}\def\Hr{{\mathscr H}}
\def\Lr{{\mathscr L}}

\def\Tr{{\mathscr T}}

\begin{document}

\title[The Two-Dimensional Narrow Escape Problem]{The Two-Dimensional Narrow Escape Problem with Deterministic Initial Condition: Mean Escape Time and Escape location probabilities}
\author[L. Delande, T. Lelièvre, B. Nectoux, T. Normand]{Lois Delande, Tony Lelièvre, Boris Nectoux, Thomas Normand}
\maketitle

\begin{abstract}
This paper presents a rigorous asymptotic analysis of the two-dimensional narrow escape problem for a Brownian particle initiated at a deterministic position $x$. While we address the mean escape time, the main novelty lies in our derivation of the asymptotic escape distribution.  We derive explicit asymptotic formulas that characterize the dependence of the escape behavior on the relative configuration of the starting point and the window centers. These results remain valid as the starting position approaches the boundary windows and include refined error estimates. The formulas are obtained with  two independent methods, each of them having their own interest. 
A central feature of this work is the rigorous treatment of the underlying mixed boundary value problems in the absence of a contact angle at the escape window boundaries. By leveraging recent results on the maximal regularity of the gradient of the solutions of such mixed boundary value problems, we ensure the analytical validity of our approach despite the inherent loss of regularity of these solutions at the escape window boundaries.  Finally, we apply these general results to provide a full sharp asymptotic equivalents for regimes where window sizes scale at specific relative rates.

\medskip
\noindent \textbf{Keywords.} Narrow escape problem, reflected Brownian motion,  escape time location, equilibrium potential, mixed boundary value problems.\\
\textbf{AMS classification.} 35J25, 35C20, 60J65, 35R60, 92C37. 
   
\end{abstract}

\section{Introduction}

\subsection{The narrow escape problem}\label{sec.NEP}
The problem investigated in this paper is motivated by applications in biology, and more specifically by the study of cellular receptors. Receptors are specialized proteins, typically located either on the cell membrane or within neurons, whose role is to interact with signaling molecules such as neurotransmitters, hormones, or other chemical messengers. Through these interactions, receptors play a fundamental role in cellular communication and in the regulation of many physiological processes.

In realistic biological environments, however, receptors do not move freely at all times. Experimental observations indicate that they may remain confined or trapped within specific regions of the cell for relatively long periods before escaping and resuming their motion. Such trapping phenomena are particularly important in neuronal membranes, where the spatial organization of receptors strongly influences synaptic transmission and signal processing.

These observations naturally lead to several mathematical and modeling questions. First, one must determine how to describe the trapping region (also called corrals) and the stochastic dynamics governing the motion of the receptor inside and outside this region. Second, it becomes important to quantify the duration of confinement, namely the amount of time a receptor spends in the trapped state before escaping. Finally, one seeks to understand the statistical properties of the escape event itself, in particular the probability distribution of the locations through which the receptor escapes the trapping domain. Addressing these questions is essential for gaining a deeper understanding of receptor dynamics and their biological implications. 

To address these questions, Holcman and Schuss \cite{HoSc04, HoSc14, HoSc15} proposed the \textit{narrow escape problem} framework. In this model, the confinement region (the corral) is represented by a smooth, bounded subdomain $\Omega \subset \mathbb{R}^2$. While the framework is applicable to $\mathbb{R}^3$, this work focuses exclusively on the two-dimensional case. 
The escape locations, representing the narrow openings through which the receptor escapes, are modeled as $N$ small, disjoint subsets of the boundary, denoted by $\Sigma_i \subset \partial \Omega$ for $i \in \{1, \dots, N\}$. The trajectory of the receptor is described by a Brownian motion in  $\Omega$, which is:
\begin{itemize}
    \item[-]  Absorbed  on each escape window $\Sigma_i$;
    \item[-] Reflected on the remaining boundary $\partial \Omega \setminus \bigcup_{i=1}^N \Sigma_i$.
\end{itemize}
In the limit of asymptotically small openings $\Sigma_i$, the narrow escape problem is  then defined by two primary objectives:
\begin{enumerate}
    \item  (Mean escape time).  Compute the expected time the process spends in $\overline{\Omega}$ before escaping through one of the windows $\Sigma_i$.
       \item  (escape time probability).  Determine the probability that the process escapes $\Omega$ specifically through a given window $\Sigma_i$.
\end{enumerate}
Prior to introducing our specific objectives, we provide the mathematical formulation of the narrow escape problem.

\begin{remark}  The asymptotic analysis of such trapping phenomena is also closely tied to the concept of entropic metastability in molecular dynamics \cite{LeSt16}. While metastability is frequently associated with energy landscapes in high dimensions, the two-dimensional narrow escape problem provides a purely geometric analogue where the system is trapped in a subdomain due to an entropic bottleneck, the small size of the escape  windows, rather than an energetic barrier. 
 \end{remark}

\subsection{The mathematical formulation}
Throughout this work,  the corral $\Omega$ is a  smooth (i.e. $\mathcal C^\infty$) bounded  subdomain\footnote{Namely, open, nonempty, and connected.} of $\mathbb R^2$. Let $(X_t)_{t \geq 0}$  be Brownian motion in $\Omega$ which is reflected on the boundary $\partial \Omega$ of $\Omega$. Since $\Omega$ is a smooth subdomain of $\mathbb R^2$, such a process   $(X_t)_{t \geq 0}$ is well-defined and is solution of the following stochastic differential equation: 
\begin{equation}\label{eq:SDE}
dX_t = \sqrt{2} dB_t - \indic_{\partial \Omega}(X_t) n(X_t) d\ell_t,
\end{equation}
where $(B_t)_{t \geq 0}$ is a standard Brownian motion in $\mathbb{R}^2$, $\indic_A$ denotes the indicator function of a set $A \subset \mathbb{R}^2$, $n(x)$ is the outward unit normal vector at $x \in \partial \Omega$, and $(\ell_t)_{t \geq 0}$ denotes the boundary local time.
 For any closed set $A \subset \mathbb{R}^2$, we define  
\begin{equation}\label{eq:stopping}
\tau_A = \inf \{ t \geq 0 \mid X_t \in A \}.
\end{equation}
As the set $A$ is closed and the trajectories of $(X_t)_{t \geq 0}$ are a.s. continuous,  $\tau_A$ is a  stopping time w.r.t.  the natural filtration of $(X_t)_{t \geq 0}$.  
Let $x^{(1)}, \dots, x^{(N)}$ be $N \geq 1$ distinct points on the boundary $\partial \Omega$. For each $i \in \lb1,N\rb$ and a given radius $\varepsilon_i > 0$, we define the escape window $\Sigma_\ddd^{i, \varepsilon_i}$ as the intersection:
\be\label{eq.fene}
\Sigma_\ddd^{i, \varepsilon_i} = \partial\Omega \cap B(x^{(i)}, \varepsilon_i)
\ee
where $B(x, r)$ denotes the open ball in $\mathbb{R}^2$ centered at $x$ with radius $r$. Note that $\Sigma_\ddd^{i, \varepsilon_i}$ is an open subset  $\partial \Omega$.  The subscript $\ddd$ serves as a precursor to Section~\ref{sub:regu}, indicating that Dirichlet (i.e. absorbing) boundary conditions  will be imposed on the infinitesimal generator of the process $(X_t)_{t \geq 0}$ (namely on the Laplacian operator). Set 
\be\label{eq.epsil}
\bm \varepsilon = (\varepsilon_1, \dots, \varepsilon_N)\in (\mathbb R_+^*)^N.\ee
We assume   throughout this work  that $ |\bm \varepsilon|>0$ is sufficiently small (say $0< |\bm \varepsilon|\le \min_{i\neq j} |x^{(i)}-x^{(j)}|/4$)  so that the escape windows are disjoint, i.e. 
\bes
\forall i\neq j,\ \ \overline\Sigma_\ddd^{i, \varepsilon_i}\cap\overline\Sigma_\ddd^{j, \varepsilon_j} = \emptyset.
\ees
A schematic representation of  the geometric situation is given in Figure \ref{fig:narrow}. For notational convenience, we introduce
\bes
\Sigma_\ddd^{\bm \varepsilon} = \bigsqcup_{i=1}^N \Sigma_\ddd^{i, \varepsilon_i}\ \text{ and }\ \Sigma_\nnn^{\bm \varepsilon} = \partial\Omega\setminus\overline\Sigma_\ddd^{\bm \varepsilon}
\ees
Note that $\Sigma_\nnn^{\bm \varepsilon}$ is an open subset  of $\partial \Omega$.  
The subscript $\nnn$ indicates  that Neumann (i.e. reflecting) boundary conditions will be imposed on the infinitesimal generator of the process $(X_t)_{t \geq 0}$. 
We can now reformulate the two primary objectives of the narrow   problem which were introduced in Section \ref{sec.NEP} as follows.   
 Given a initial condition  $X_0$ supported within $\overline \Omega\setminus  \overline{\Sigma}_\ddd $, the problem consist in  studying in the limit $\varepsilon\to 0$ (i.e. the small openings regime) the asymptotic behavior of: 
 \begin{enumerate}
    \item   The mean escape time   $\E  [\tau_{\overline\Sigma_\ddd^{\bm \varepsilon}}  ]$. 
    
           \item  The escape location probabilities  $\P   [X_{\tau_{\overline\Sigma_\ddd^{\bm \varepsilon}}}\in\overline\Sigma_\ddd^{i, \varepsilon_i}  ]$, $i \in \lb1,N\rb$.
\end{enumerate}
\begin{figure}[h]
\centering
\tikzset{every picture/.style={line width=0.75pt}}
\begin{tikzpicture}[x=0.75pt,y=0.75pt,yscale=-1,xscale=1]
\draw (642,453) .. controls (683,459) and (620,422) .. (668,429) .. controls (716,436) and (636,380) .. (764,403) .. controls (892,426) and (843.55,499.16) .. (863,528) .. controls (882.45,556.84) and (855,571) .. (843,564) .. controls (831,557) and (732.5,556) .. (727,563) .. controls (721.5,570) and (737,586.5) .. (748,593) .. controls (759,599.5) and (763,611) .. (705,609) .. controls (647,607) and (700,557) .. (651,535) .. controls (602,513) and (601,447) .. (642,453) -- cycle ;
\draw [color = black, line width = 2.5pt] (642,453) .. controls (683,459) and (620,422) .. (668,429) ;
\draw [color = black, line width = 2.5pt] (863,528) .. controls (882.45,556.84) and (855,571) .. (843,564) ;
\draw [color = black, line width = 2.5pt] (748,593) .. controls (759,599.5) and (763,611) .. (705,609) ;
\draw (750,480) node{$\Omega$};
\draw (650,400) node{$\Sigma_\ddd^{1, \varepsilon_1}$};\draw (652,440) node{$\bullet$} node[left]{$x^{(1)}$};
\draw (900,550) node{$\Sigma_\ddd^{2, \varepsilon_2}$};\draw (868,555) node{$\bullet$} node[below]{$x^{(2)}$};
\draw (780,600) node{$\Sigma_\ddd^{3, \varepsilon_3}$};\draw (735,609) node{$\bullet$} node[below]{$x^{(3)}$};
\draw (825,515) node{$X_t$};
\draw [densely dashed] (672,524) -- (687,518) -- (659,505) -- (687,530) -- (676,541) -- (672,555) -- (683,564) -- (686,554) -- (696,571) -- (714,557) -- (727,563) -- (719,573) -- (700,557) -- (675,578) -- (685,605) -- (700,595) -- (716,593) -- (705,581) -- (724,552) -- (707,540) -- (704,523) -- (719,533) -- (717,517) -- (725,502) -- (741,506) -- (732,519) -- (752,514) -- (751,531) -- (768,518) -- (763,537) -- (741,542) -- (752,557) -- (763,548) -- (776,552) -- (778,543) -- (791,548) -- (794,534) -- (787,519) -- (799,523) -- (803,536) -- (815,531) -- (807,545) -- (824,542) -- (832,548) -- (849,545) -- (849,555) -- (862,561) ;
\end{tikzpicture}
\caption{The corral $\Omega$ with three escape windows on its boundary ($N=3$). }
\label{fig:narrow}
\end{figure}
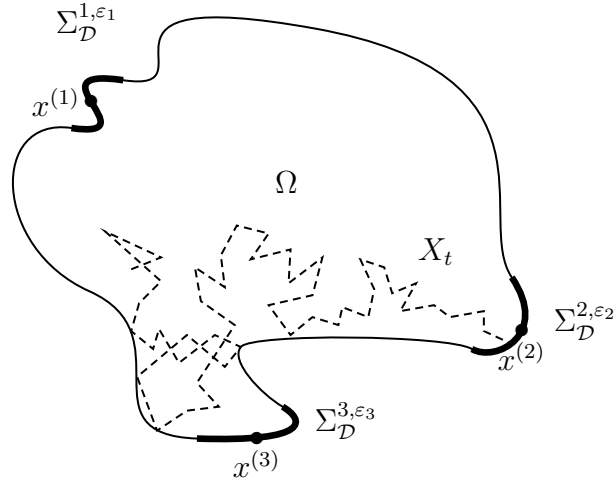

\subsection{Existing results and purpose of this work}

\subsubsection{Related literature on the narrow escape problem}
Several results concerning the mean escape time have been established in the literature. Initial derivations for the mean escape time of a Brownian particle as well as the escape location distribution, in the asymptotic regime of small escape windows, were performed using formal methods in \cite{HoSc04,SiScHo06,SiScHoEi06,HoSc15}, providing a foundational understanding of the narrow escape problem (see also \cite{PiWaPeKo10}). In these works, one of the central contributions lies in framing the narrow escape problem as a singular perturbation of some Neumann problems. This conceptual framework is taken up in our study, where we employ two different methodologies to address the narrow escape problem, see Section \ref{sec.cont} for more details. We mention that in these works, the three-dimensional case is also investigated.

More recently, the full exit-time distribution in the narrow escape regime was obtained, still at a formal level, in \cite{GrMeOs19}.
Building on the approach introduced in \cite{HoSc04} and using layer potential techniques developed in \cite{AmKaLe09},  asymptotic expansions of the mean exit time and of the eigenvalues of the associated Laplacian with mixed boundary conditions, in the small openings  limit, were established in \cite{AmKaKaLe12} with more rigorous mathematical considerations (see also \cite{AmGaKaLeSo11}). Let us also mention that the case of Brownian particle perturbed by a gradient drift field is also treated in \cite{AmKaKaLe12}.

These results were subsequently extended to three-dimensional manifolds with boundary in \cite{SiScHo08,ChWaSt10} and more recently in \cite{NuTzTz21} using microlocal analysis techniques. We also refer to \cite{Ji15} for the asymptotic behavior of the eigenvalues of the Laplacian in a domain with a thin tubular hole, with Robin or Neumann boundary conditions imposed on the boundary of the hole. More generally, complete asymptotic expansions for eigenvalue perturbations due to small inclusions or small deformations of their interfaces can be found in \cite[Chapter 3]{AmKaLe09}.

In contrast, the study of escape location probabilities remains much more limited and has mainly been restricted to specific geometric configurations.  More precisely, in the case where the the corral $\Omega$ is a  slightly  modified  $n$-dimensional sphere ($n=2,3$) and when    $(X_t)_{t \geq 0}$ is initially distributed according to its quasi-stationary distribution, the asymptotic distribution  of the escape location has recently been investigated in \cite{LeRaSt24} using semiclassical analysis. 
 Assuming  $(X_t)_{t \geq 0}$ is initially distributed according to its quasi-stationary distribution, the results of \cite{LeRaSt24} were recently extended   in \cite{CaLeNo26} for general subdomain $\Omega$ of $\mathbb R^2$.
 The quasi-stationary distribution describes a long-term forgotten memory regime where the process has already spent enough time mixing in the interior to lose all information about its original location.
    The approximation framework developed in \cite{CaLeNo26} serves as one of the two independent methodologies employed in the present work.  
    However, while their analysis is tailored to a quasi-stationary initial state, our focus extends to deriving sharp asymptotic equivalents for a process starting from an arbitrary deterministic point $X_0 = x \in \overline{\Omega}$ (see Section \ref{sec.cont} for further details).  We also refer to \cite{Fa24} for a complex-analysis-based approach  to the hitting probability of a single small target under mixed boundary conditions, where the Brownian motion is reflected along the domain and absorbed elsewhere. While their results hold in arbitrary dimensions, the geometric configuration in  \cite{Fa24}  is highly specific, the derived equivalents and bounds remain implicit, and no explicit error terms are provided. 

\subsubsection{Purpose of this work and strategy}
\label{sec.cont}
The primary objective of this work is to provide a rigorous and comprehensive asymptotic analysis of the two-dimensional narrow escape problem: i.e. of the mean escape time and of the escape location probabilities. 
  Our focus is structured around three key pillars:
    
    \begin{enumerate}
    \item[(A)] Objectives of the study.  We investigate the mean escape time (Theorem \ref{thm:1}) and, most importantly, the escape location probabilities   (Theorem \ref{thm:2}) for a process starting from a deterministic point $x$ within the closure of $\Omega$ (more precisely within $\overline{\Omega}_{\bm \varepsilon}$, see \eqref{eq.Oe}). While the mean escape time has been extensively studied, a general asymptotic formula for the escape probabilities starting from an arbitrary interior point, rather than a quasi-stationary distribution, addresses a relatively less explored aspect in the existing literature, which this work seeks to explore.  By avoiding the forgotten memory assumption inherent to the quasi-stationary distribution, a deterministic starting point preserves the localized history of the process' path, allowing us to explicitly quantify the immediate geometric influence of the initial position. 
Indeed, when the process starts from a specific point $x$, the resulting formula \eqref{eq:thmP} explicitly illustrates how the asymptotic behavior of the escape location probabilities depends on the relative positioning of $x$ and the window centers $x^{(k)}$. It is worth noting that the initial conditions $x \in \overline{\Omega}_{\bm{\varepsilon}}$ we consider are permitted to approach the escape windows as $|\bm{\varepsilon}| \to 0$, and that the results of Theorems \ref{thm:1} and \ref{thm:2} hold uniformly for initial conditions $x \in \overline{\Omega}_{\bm{\varepsilon}}$. We then use the general frameworks of Theorems \ref{thm:1} and \ref{thm:2} to derive sharp asymptotic equivalents, complete with explicit error terms, for specific geometric configurations where the relative scaling of the windows is prescribed. This is the purpose of Corollaries \ref{cor:3}, \ref{cor:1}, and \ref{cor:2}. \\
    
    \item[(B)] Strategy and methodology. Our approach relies on the Feynman-Kac formula to transform the probabilistic questions into mixed elliptic boundary value problems. A major strength of this work is the use of two fully independent  methods that yield identical results, thereby reinforcing the robustness of our formulas.\\
          \begin{enumerate} 
          \item[(M1)] \textit{The approximation method}: This approach consists of constructing explicit functions  that approximate the quantities of interest in a suitable topology. We build upon the approximating functions recently developed in \cite{CaLeNo26} for the study of mean escape times starting from a quasi-stationary distribution. A key insight of the present work is the observation that these approximating functions, originally designed for the mean escape time, can be adapted to construct high-order approximations of the escape location probabilities.  \\
          
          \item[(M2)] \textit{The Green method}. This method  utilizes representations of the solution to the underlying mixed boundary value problems (Zaremba problems) in terms of Neumann Green functions (which are introduced below in Section \ref{sec.RkR}). This representation is obtained via  a  careful integration by parts (see Corollary \ref{cor:IPP}). Using the boundary constrains, the problem is then reduced to a linear system involving the unknown traces of the normal derivatives of the functions on the escape windows. The resolution of this system requires the inversion of singular integral operators (with logarithmic kernels) acting on small subintervals of $\mathbb R$. To perform this inversion rigorously, we leverage classical results on singular operators from \cite{Jo82}. We refer to Section \ref{sec.star2} for more details on this strategy. \\
          \end{enumerate}

     \item[(C)] Novelty of the results.  The main novelty lies mainly in the derivation of the explicit asymptotic formula \eqref{eq:thmP} for the hitting probabilities. This result characterizes how the escape location is influenced by the precise geometry of the corral $\Omega$ and the relative configuration of the starting point and the escape window centers. Furthermore, we provide a rigorous justification for these expansions. Unlike the recent spectral approaches  \cite{LeRaSt24}, we do not assume the existence of a contact angle between the absorbing and reflecting boundary parts. By considering a smooth boundary where $\Sigma_\ddd$ and $\Sigma_\nnn^{\bm \varepsilon}$ meet at an angle of $\pi$, we operate in a regime where the solutions to the underlying mixed boundary value problems lose regularity. We address this challenge by employing modern results on the maximal regularity of the gradient for mixed problems, drawing on the works of R. Brown \textit{et al.}, K. Gröger, E. Shamir, and G. Savaré.
        In this low-regularity framework where $\Sigma_\ddd^{\bm \varepsilon}$ and $\Sigma_\nnn^{\bm \varepsilon}$ meets at an angle $\pi$, a particular care is given to the rigorous derivation of the Feynman-Kac, integration-by-parts formulas, and to the maximal regularity of the approximating functions to ensure the validity of our analytical framework.  
                  \end{enumerate} 
            

            \subsection{Notation and main results}
            
            \subsubsection{Some  basic notation}

  For any     open subset $D$ of $\Omega$ with sufficiently smooth boundary and  every $z\in\partial D$,   $n_D(z)$ denotes the outward-pointing unit normal to $D$ at $z$. Moreover,    $\partial_{n_D(z)}f(z)$ denotes the normal derivative  of $f$ at $z\in\partial D$, which we recall is defined by $\nabla f(z)\cdot n_D(z)$. If  no  confusion is possible,  we write $n(z)$ or simply $n$  for $n_D(z)$.  In addition, for a function  $(x,y)\in \mathbb R^2\times \mathbb R^2\mapsto f(x,y)$, 
 the Laplacian of $f$ in the variable $x=(x_1,x_2)\in \mathbb R^2$ is classically defined by 
$$\Delta_x f (x,y)
= \partial_{x_1}^2f (x,y)+ \partial_{x_2}^2f (x,y). 
$$
For a nonempty    open   subset $U$ of $\mathbb R^2$ or a  nonempty    open   subset $U$ of the one-dimensional smooth manifold $\partial \Omega$ $p\in[1,+\infty]$, we denote $W^{s,p}(U)$ the  Sobolev space of functions with regularity $s\in \R$ in $L^p(U)$. We also use  the standard notation $H^s(U)$ for the set $W^{s,2}(U)$. For  $\alpha\in(0,1)$, the set   $\ccc^{0,\alpha}(\overline U)$ is the Banach space of H\"older continuous functions on $\overline U$ of order $\alpha$. For  $x\in\overline\Omega$, we denote by $\delta_x$ the Dirac mass at $x$. 
The   surface measure on $\partial \Omega$ is denoted by $d\sigma(\zeta)$ and  $|A|$ denotes the Lebesgue measure of a Borel set $A \subset \mathbb R^2$.  
We also use the Kronecker delta $\delta_{i,j}$ for $i,j\in\lb1,N\rb$, that is $\delta_{i,j} = 1$ if $i=j$ and $0$ otherwise.  
 In what follows, $\E_x$ (resp. $\P_x$)  stands for the expectation (resp. the probability) when $X_0=x\in \overline \Omega$.

Define
\be\label{eq:defK}
K_k^{\varepsilon_k} = \frac{1}{\big|\ln\frac{\varepsilon_k}{2}\big|}, \ k\in\lb1,N\rb, \  \text{ and }   \mathbf{K}^{\bm \varepsilon} = (K_1^{\varepsilon_1}, \dots, K_N^{\varepsilon_N})\in (\mathbb R_+^*)^N.
\ee 
Note that   $|\mathbf{K}^{\bm \varepsilon}|\to 0$ as $|\bm \varepsilon|\to 0$. 
We finally define the following open subset of $\Omega$:
\be\label{eq.Oe}
\Omega_{\bm \varepsilon} = \Omega\setminus\Big(\bigcup_{k=1}^N\overline{B}(x^{(k)},K_k^{\varepsilon_k})\Big).
\ee 

  \subsubsection{The Neumann Green functions}
\label{sec.RkR}
In this section, we introduce the several auxiliary functions ($R_0^k$, $R_{\Omega}$, $R_{\partial\Omega}$) along with the Neumann Green functions   $G_{\Omega}$ and $G_{\partial\Omega}$ associated with the Neumann problem on $\Omega$. These functions play a pivotal role in establishing the asymptotic formulas for both the mean escape time and the escape location probabilities. More specifically, the function $R_0^k$ is utilized in our first approach (the approximation method) to derive the main results. Conversely, the functions $R_{\Omega}$ and $R_{\partial\Omega}$ serve to construct the Neumann Green functions $G_{\Omega}$ and $G_{\partial\Omega}$, which are the core ingredients of our second approach (the Green function method). As shown later, this second method provides a fully independent proof of our main theorems. The considerations that follow are well-known, and we recall them here for the reader's convenience and to fix notation. 

Define for $z\in \partial \Omega$, the function 
$$\zeta \in \partial \Omega\mapsto w_{\partial\Omega}(\zeta,z) = \frac{1}{|\partial\Omega|} - \frac{1}{\pi}\frac{\zeta-z}{|\zeta-z|^2}\cdot n(\zeta).$$
This function satisfies:
\be\label{eq.CondC2}
\forall z \in \partial \Omega, \ 
\int_{\partial\Omega} w_{\partial \Omega}(\zeta,y)d\sigma(\zeta)=0. 
\ee
 To prove \eqref{eq.CondC2},   consider for $z \in \partial\Omega$ and $r>0$, $\Omega_r(z) := \Omega \setminus \overline{B}(z,r)$. Then
 \bes
\begin{split}
    0 = \int_{\Omega_r(z)}\Delta_x\ln|x-z|dx &= \int_{\partial\Omega\setminus\overline{B}(z,r)}\frac{\zeta-z}{|\zeta-z|^2}\cdot n_{\Omega}(\zeta)\, d\sigma(\zeta) - \int_{\Omega\cap \partial B(z,r)}\frac{d\sigma(\zeta)}{|\zeta-z|}\\
    &\xrightarrow[r\to 0]{}\int_{\partial\Omega}\frac{\zeta-z}{|\zeta-z|^2}\cdot n_{\Omega}(\zeta)d\sigma(\zeta) - \pi,
\end{split}
\ees
 where we used  the fact that, asymptotically as $r \to 0$, only half of the ball $B(z,r)$ lies inside $\Omega$. This achieves the proof of  \eqref{eq.CondC2}. 
  Moreover, note that for every $z\in \partial \Omega$,  the mapping 
\begin{equation}\label{eq.Nde}
\zeta \in \partial \Omega \mapsto \partial_{n(\zeta)} \ln|\zeta - z| = \frac{\zeta - z}{|\zeta - z|^2}\cdot n(\zeta)
\end{equation} 
  is actually smooth. Indeed,   this  mapping is smooth over  $\partial \Omega \setminus \{z\}$.  At $\zeta = z$, a Taylor expansion removes the apparent singularity of this function. 
  Introduce also for $y \in \Omega$,  the    function $w_{\Omega}(\cdot,y)$ over $\partial \Omega$ defined by
\bes 
\zeta \in \partial \Omega\mapsto w_{\Omega}(\zeta,y) = \frac{1}{|\partial\Omega|} - \frac{1}{2\pi}\frac{\zeta-y}{|\zeta-y|^2} \cdot n(\zeta).
\ees
For every $y \in \Omega$,    $\zeta \in \partial \Omega \mapsto w_{\Omega}(\zeta,y) $ is $\mathcal C^\infty$ and  satisfies:
\be\label{eq.CondC1}
\forall y \in \Omega, 
\int_{\partial\Omega} w_{\Omega}(\zeta,y)d\sigma(\zeta)=0. 
\ee
To prove \eqref{eq.CondC1} consider  for $r>0$,   $\Omega_r(y) := \Omega \setminus \overline{B}(y,r)$. For $r$ small enough, the boundary of $\Omega_r(y)$ is the disjoint union of $\partial\Omega$ and $\partial B(y,r)$. Since  $\Delta_x \ln|x-y| = 0$ in $\Omega_r$, it holds that:
\bes
\begin{split}
    0 &= \int_{\partial\Omega}\partial_{n_{\Omega_r}(\zeta)} \ln|y-\zeta|\, d\sigma(\zeta) - \int_{\partial B(y,r)} \partial_{n_{B(y,r)}(z)} \ln|y-\zeta|\, d\sigma(\zeta)\\
    &= \int_{\partial\Omega}\frac{\zeta-y}{|\zeta-y|^2} \cdot n_{\Omega}(\zeta)\, d\sigma(\zeta) - \int_{\partial B(y,r)} \frac{ d\sigma(\zeta) }{|\zeta-y|} = \int_{\partial\Omega}\frac{\zeta-y}{|\zeta-y|^2} \cdot n_{\Omega}(\zeta)d\sigma(\zeta) - 2\pi.
\end{split}
\ees
This achieves the proof of  \eqref{eq.CondC1}.
 With the previous considerations in hand,  we can consider, for every  $k\in\lb1,N\rb$, $y\in \Omega$, and $z \in \partial \Omega$,     the functions $x\in \overline{\Omega}\mapsto R_k^0(x)$, $x\in \overline{\Omega}\mapsto R_{ \Omega}(x, y)$, and $x\in \overline{\Omega}\mapsto R_{\partial \Omega}(x, z)$  defined as the unique (weak) solutions in $H^1(\Omega)$ to the following  the following Neumann problems:
\be\label{eq:Rok}
\left\{
\begin{aligned}
    \Delta R_k^0 &= -\frac{\pi}{|\Omega|}&&\text{in }\Omega,\\
    \partial_{n_\Omega(\zeta)} R_k^0(\zeta) &= -\partial_{n(\zeta)} \ln|\zeta-x^{(k)}| &&\text{on }\partial\Omega,\\
    R_k^0(x^{(k)}) &= 0,
\end{aligned}
\right.
\ee
\be\label{eq:RGreens}
\left\{
    \begin{aligned}
        \Delta_x R_\Omega(x,y) &= 0 &&\text{in }\Omega,\\ 
        \partial_{n_\Omega(\zeta)} R_\Omega(\zeta,y) &= w_\Omega(\zeta,y) &&\text{on }\partial\Omega,\\
        \int_{\partial\Omega} R_\Omega(\zeta,y) d\sigma(\zeta) &= -\frac1{2\pi}\int_{\partial\Omega}  \ln|\zeta-y| d\sigma(\zeta),
    \end{aligned}
\right. 
\ee
\be\label{eq.RpaOmega}
\left\{
\begin{aligned}
    \Delta_x R_{\partial\Omega}(x,z) &= 0&&\text{in }\Omega,\\
    \partial_{n_\Omega(\zeta)} R_{\partial\Omega}(\zeta,z) &= w_{\partial\Omega}(\zeta,z)  &&\text{on }\partial\Omega,\\
    \int_{\partial\Omega}R_{\partial\Omega}(\zeta,z)d\sigma(\zeta) &= -\frac1\pi\int_{\partial\Omega}\ln|\zeta-z|d\sigma(\zeta).
\end{aligned}
\right.
\ee 
In addition, by the elliptic regularity theory for Neumann problems, the  functions $x\in \overline{\Omega}\mapsto R_k^0(x)$, $x\in \overline{\Omega}\mapsto R_{ \Omega}(x, y)$, and $x\in \overline{\Omega}\mapsto R_{\partial \Omega}(x, z)$  are  $\mathcal C^\infty$.  
We note that these  functions are generally not explicit\footnote{Except in some  specific geometric settings such as the two-dimensional disk (see Section \ref{sub:disk}).}.
We then finally set, for  $y \in \Omega$ and $z \in \partial\Omega$,   for all $x\in \overline{\Omega}$\footnote{We set $\ln 0^+=+\infty$.}:
\be\label{eq:Greens2}
G_\Omega(x,y)=\frac1{2\pi}\ln|x-y| + R_\Omega(x,y),\ \ \ G_{\partial\Omega}(x,z)=\frac1{\pi}\ln|x-z| + R_{\partial\Omega}(x,z).
\ee   
 The results of Section \ref{sec.IPPs} below will in particular prove that these    functions  are fundamental solutions  of:
\be\label{eq:Greens1}
\left\{
    \begin{aligned}
        \Delta_x G_\Omega(x,y) &= \delta_y  &&\text{in }\Omega,\\
        \partial_{n_\Omega(\zeta)} G_\Omega(\zeta,y) &= \frac{1}{|\partial\Omega|} &&\text{on }\partial\Omega,\\
        \int_{\partial\Omega} G_\Omega(\zeta,y) d\sigma(\zeta) &= 0,
    \end{aligned}
\right.
\
\text{ and }  \ 
\left\{
    \begin{aligned}
        \Delta_x G_{\partial\Omega}(x,z) &= \delta_z  &&\text{in }\Omega,\\
        \partial_{n_\Omega(\zeta)} G_{\partial\Omega}(\zeta,z) &= \frac{1}{|\partial\Omega|} &&\text{on }\partial\Omega,\\
        \int_{\partial\Omega} G_{\partial\Omega}(\zeta,z) d\sigma(\zeta) &= 0.
    \end{aligned}
\right.
\ee        
 The functions    $R_{\Omega}$ and  $G_{\Omega}$ (resp. $R_{\partial\Omega}$ and $G_{\partial\Omega}$) are symmetric over $\Omega \times \Omega$  (resp. over $\partial\Omega \times \partial\Omega$), see Lemma \ref{le.sym} and \eqref{eq:Greens2}.   
Lemma \ref{lem:symmetry} below  shows that for every $x\in\Omega$ and $z\in\partial\Omega$, 
    $G_\Omega(z,x) = G_{\partial\Omega}(x,z)$.

 \subsubsection{Main results} 
Our  main result regarding the mean escape time is the following.

\begin{theorem}\label{thm:1}
{For all  $\delta\in (0,1)$,  the following asymptotic expansion holds as  $|\bm \varepsilon|\to 0$, uniformly in   $x\in\overline\Omega_{\bm \varepsilon}$:}
    \be\label{eq:thmE} \E_x\big[\tau_{\overline\Sigma_\ddd^{\bm \varepsilon}}\big ] = \frac{|\Omega|}{\pi\sum_{k=1}^{N}K_k^{\varepsilon_k}}\Big[1 + \sum_{k=1}^{N}K_k^{\varepsilon_k}\big (\ln|x-x^{(k)}| + g_k^{\bm \varepsilon}(x)\big )\Big] + O (|\mathbf K^{\bm \varepsilon}|^{1-\delta}),
    \ee
    where  for the function $(g_k^{\bm \varepsilon} )_{k\in\lb1,N\rb}$ is given by, for all  {and $x\in \overline{\Omega}$} (see \eqref{eq:Rok}):
    \be\label{eq:espThomas}
    g_k^{\bm \varepsilon}(x) = R_k^0(x) - \frac{\sum_{j=1,j\neq k}^N K_j^{\varepsilon_j} \big[\ln|x^{(j)}-x^{(k)}| + R_j^0(x^{(k)}) \big]}{\sum_{j=1}^NK_j^{\varepsilon_j}}.
    \ee 
\end{theorem}

\noindent
A few remarks and further details concerning Theorem \ref{thm:1} are in order.

 \begin{enumerate}
\item[\textbf{(r1)}] We adopt the convention that $\sum_{\emptyset} = 0$, which  arises in Theorem \ref{thm:1}   in the specific case $N=1$. 
\vspace{0.2cm}

\item[\textbf{(r2)}] As it will be clear from our analysis, the remainder term $O (|\mathbf K^{\bm \varepsilon}|^{1-\delta})$ in \eqref{eq:thmE} can be refined as $O(|\mathbf K^{\bm \varepsilon}|)$ if one restricts  the initial conditions $x$ in a   fixed (i.e. independent of $\bm \varepsilon$) open set $D$ such that $\overline{D} \subset \overline{\Omega} \setminus \{x^{(1)}, \ldots, x^{(N)}\}$.   
\vspace{0.2cm}

\item[\textbf{(r3)}] One can verify that every function $g_k^{\bm \varepsilon} $ is  bounded over $\overline\Omega$, uniformly in  $|\bm \varepsilon|>0$ small enough. Indeed,  using \eqref{eq:espThomas}, we deduce that for all $x\in \overline \Omega$, 
 $$|g_k^{\bm \varepsilon}(x)| \le |R_k^0(x)|+   \sum_{j=1,j\neq k}^N \big (\, \big | \ln|x^{(j)}-x^{(k)}| \big | +  |R_j^0(x^{(k)}) | \, \big ),$$
 where we recall that $R_k^0 \in \mathcal{C}^\infty(\overline{\Omega})$.  
On the other hand,  for any $x\in \overline\Omega_{\bm \varepsilon}$, we have $ |x-x^{(k)}|\ge K_k^{\varepsilon_k}$ and therefore, using also that $ |x-x^{(k)}|\le C$ for some $C>0$ independent of $x\in \overline{\Omega}$, we get that:
$$\lim_{|\bm \varepsilon| \to 0} K_k^{\varepsilon_k}\big [\ln|x-x^{(k)}| + g_k^{\bm \varepsilon}(x)\big ]=0, \ \text{         uniformly in $x\in \overline\Omega_{\bm \varepsilon}$. }$$
Hence, it follows from \eqref{eq:thmE} that for all  $x\in\overline\Omega_{\bm \varepsilon}$, we have as $|\bm{\varepsilon}|\to 0$ and uniformly in $x\in \overline\Omega_{\bm \varepsilon}$: 
    \be\label{eq:thmE2}
    \E_x\big[\tau_{\overline\Sigma_\ddd^{\bm \varepsilon}}\big ] = \frac{|\Omega|}{\pi\sum_{k=1}^{N}K_k^{\varepsilon_k}} \Big[1 +  \sum_{k=1}^{N} O ( K_k^{\varepsilon_k}|\ln K_k^{\varepsilon_k}| )+ O (|\mathbf K^{\bm \varepsilon}|^{2-\delta})\Big].
        \ee 
        The two error terms in \eqref{eq:thmE2} cannot be compared without imposing additional assumptions on the relative rates at which the window sizes $\varepsilon_k$ vanish. Nonetheless, within quite general frameworks (such as the linear and polynomial scaling regimes considered in Corollaries \ref{cor:1} and \ref{cor:2}), we have $O(K_k^{\varepsilon_k}|\ln K_k^{\varepsilon_k}| ) + O(|\mathbf{K}^{\bm{\varepsilon}}|^{2-\delta} ) = O(K_k^{\varepsilon_k}|\ln K_k^{\varepsilon_k}| )$. 
            Note that at   fixed window sizes, the leading-order term in  \eqref{eq:thmE2}  decreases as the number of escape windows increases, as physically expected. 
Let us also emphasize that  if  $x$ is in a   fixed open set $D$ such that $\overline{D} \subset \overline{\Omega} \setminus \{x^{(1)}, \ldots, x^{(N)}\}$,
the error term  in \eqref{eq:thmE2} is $O(|\mathbf{K}^{\bm{\varepsilon}}|)$. 
\vspace{0.2cm}

\item[\textbf{(r4)}]  The asymptotic expansion \eqref{eq:thmE} reveals a so-called \emph{leveling result} for the mean escape time, namely, as the escape window sizes approach zero, the leading-order mean escape time becomes independent of the initial position. More precisely, 
we have    as  $|\bm \varepsilon|\to 0$ and  uniformly in   $x,y\in\overline\Omega_{\bm \varepsilon}$, 
$$\frac{\E_x\big[\tau_{\overline\Sigma_\ddd^{\bm \varepsilon}}\big ] }{\E_y\big[\tau_{\overline\Sigma_\ddd^{\bm \varepsilon}}\big ] } = 1 +  \sum_{k=1}^{N} O ( K_k^{\varepsilon_k}|\ln K_k^{\varepsilon_k}| ) + O (|\mathbf K^{\bm \varepsilon}|^{2-\delta}).
$$
 The same phenomena is observed for the escape location probabilities, see the discussion after Corollary \ref{cor:2}. 
 \vspace{0.2cm}

\item[\textbf{(r5)}]  We finally mention that in the case where all the escape windows have the same size,  the  formula \eqref{eq.Cor1} was first derived in \cite{SiScHoEi06,SiScHo06,PiWaPeKo10} and in~\cite[Theorem 3.6]{AmKaLe09}.   
     
     \end{enumerate}
          
 Our second main results concern the escape location probabilities when $X_0=x\in \overline\Omega_{\bm \varepsilon}$ in the limit where all the escape windows $\Sigma^j_\ddd$ shrink to $0$  (i.e. when $|\bm \varepsilon|\to 0$).

\begin{theorem}\label{thm:2}
    Let $i\in\lb1,N\rb$. {Then,   for all  $\delta\in (0,2)$, the following asymptotic expansion holds as  $|\bm \varepsilon|\to 0$, uniformly in   $x\in\overline\Omega_{\bm \varepsilon}$:}
    \be\label{eq:thmP}
    \begin{split} 
        \P_x\big[X_{\tau_{\overline\Sigma_\ddd^{\bm \varepsilon}}}\in\overline\Sigma_\ddd^{i, \varepsilon_i}\big] &= \frac{K^{\varepsilon_i}_i}{\sum_{k=1}^{N}K_k^{\varepsilon_k}} + \frac{K^{\varepsilon_i}_i}{(\sum_{k=1}^{N}K_k^{\varepsilon_k})^2}\sum_{k=1}^N\sum_{j\notin\{i,k\}}K_k^{\varepsilon_k}K_j^{\varepsilon_j}\ln\frac{|x^{(k)}-x^{(i)}|}{|x^{(k)}-x^{(j)}|}\\
        &\quad+\frac{K^{\varepsilon_i}_i}{(\sum_{k=1}^{N}K_k^{\varepsilon_k})^2}\sum_{k=1}^NK_k^{\varepsilon_k}(K_k^{\varepsilon_k}-K_i^{\varepsilon_i})\ln|x^{(i)}-x^{(k)}|\\
        &\quad+ \frac{K^{\varepsilon_i}_i}{\sum_{k=1}^{N}K_k^{\varepsilon_k}}\sum_{k=1}^NK_k^{\varepsilon_k}\ln\frac{|x-x^{(k)}|}{|x-x^{(i)}|}\\
        &\quad+ \frac{K^{\varepsilon_i}_i}{(\sum_{k=1}^{N} K_k^{\varepsilon_k})^2}\sum_{j,k=1}^N K_j^{\varepsilon_j}K_k^{\varepsilon_k}g_{i,j,k}(x) + O(|\mathbf K^{\bm \varepsilon}|^{2-\delta}),
    \end{split}
    \ee
   where for all every $i,j,k\in\lb1,N\rb$,  $g_{i,j,k}(x)$ is given by either of the two formulas for $x\in \overline \Omega$:
    \be\label{eq:probaThomas}
    g_{i,j,k}(x) = R_i^0(x^{(k)}) - R_j^0(x^{(k)}) + R_j^0(x) - R_i^0(x)
    \ee 
    or
    \be\label{eq:probaBoris}
    g_{i,j,k}(x) = \pi R_{\partial\Omega}(x^{(i)},x^{(k)}) - \pi R_{\partial\Omega}(x^{(j)},x^{(k)}) + \pi R_{\partial\Omega}(x,x^{(j)}) - \pi R_{\partial\Omega}(x,x^{(i)}).
    \ee
\end{theorem}

 We now comment the results of Theorem \ref{thm:2}.
\begin{enumerate}
 \item[\textbf{(R1)}] Formula \eqref{eq:thmP} provides a general asymptotic framework showing that the escape location for $X_0=x$ depends on the relative positioning of the starting point $x$ and the various window centers $x^{(k)}$.
 \vspace{0.2cm}
 
 \item[\textbf{(R2)}]  
The result of Theorem \ref{thm:2}  can be pushed forward by giving a so-called   expansions in powers of $|\mathbf K^{\bm \varepsilon}|$ up to an error of order $|\mathbf K^{\bm \varepsilon}|^{q-\delta}$ for any $q \ge 2$ of the espace location probabilities. Indeed, within the approximation method, this is achieved by constructing higher-order approximating functions, following the approach of \cite{CaLeNo26}, see Remark \ref{re.am}. In the Green method, such expansions are obtained by developing the asymptotic expansions of key quantities appearing in Section \ref{sec:thmBoris}. A similar remark also holds for Theorem  \ref{thm:1}. 
 \vspace{0.2cm}

 \item[\textbf{(R3)}]   The remainder term $O(|\mathbf K^{\bm \varepsilon}|^{2-\delta})$ in \eqref{eq:thmP} can be refined as $O(|\mathbf K^{\bm \varepsilon}|^2)$  if one restricts  the initial conditions $x$ in a   fixed open set $D$ such that $\overline{D} \subset \overline{\Omega} \setminus \{x^{(1)}, \ldots, x^{(N)}\}$. 
  \vspace{0.2cm}
 
 \item[\textbf{(R4)}]  
Unlike the expression for the expected exit time in \eqref{eq:thmE}, the formula   \eqref{eq:thmP} does not yield a direct asymptotic equivalent for the escape location  probabilities   without further constraints on  the relative rates at which the window sizes $\varepsilon_k$ vanish.  
We provide sufficiently general  assumptions  to  extract from \eqref{eq:thmE} both a leveling result and precise   asymptotic  equivalents  of  the escape location probabilities. This is the purposes of  Corollaries  \ref{cor:3}, \ref{cor:1}, and \ref{cor:2} below.   
\end{enumerate}


\begin{corollary}\label{cor:3}
Let $i\in\lb1,N\rb$.   Assume that there is $\delta\in (0,2)$ such that  when  $|\bm{\varepsilon}| \to 0$:
\begin{equation}\label{eq:assH}
\gamma_i^{\bm \varepsilon}(\delta) :=  \frac{ |\mathbf{K}^{\bm{\varepsilon}}|^{2-\delta}\sum_{k=1}^N K_k^{\varepsilon_k} }{ K_i^{\varepsilon_i}}  \to 0.  
\end{equation} 
Then,  
we have as $|\bm{\varepsilon}| \to 0$ and uniformly in  $x  \in \overline{\Omega}_{\bm{\varepsilon}}$:  
  $$\P_x\big[X_{\tau_{\overline\Sigma_\ddd^{\bm \varepsilon}}}\in\overline\Sigma_\ddd^{i, \varepsilon_i}\big]= \frac{K_i^{\varepsilon_i} }{ \sum_{k=1}^N K_k^{\varepsilon_k}} \Big [ 1  + \sum_{k=1}^N O\big (\mu_{k,i}^{\bm{\varepsilon}}\big )   + O(\gamma_i^{\bm \varepsilon}(\delta)) \Big],$$
where  $\mu_{k,i}^{\bm{\varepsilon}}:=K_k^{\varepsilon_k}| \ln (K_k^{\varepsilon_k}  K_i^{\varepsilon_i})| \to 0$. In addition,   as $|\bm{\varepsilon}| \to 0$ and uniformly in  $x, y \in \overline{\Omega}_{\bm{\varepsilon}}$: 
\be\label{eq.LR}
  \frac{\P_x\big[X_{\tau_{\overline\Sigma_\ddd^{\bm \varepsilon}}}\in\overline\Sigma_\ddd^{i, \varepsilon_i}\big]}{\P_y\big[X_{\tau_{\overline\Sigma_\ddd^{\bm \varepsilon}}}\in\overline\Sigma_\ddd^{i, \varepsilon_i}\big]} =1 + \sum_{k=1}^N O\big (\mu_{k,i}^{\bm{\varepsilon}}\big )   + O(\gamma_i^{\bm \varepsilon}(\delta)), 
\ee 
  \end{corollary}


The proof of the Corollary \ref{cor:3} is performed in Section \ref{sec:cor}. 
Notice that \eqref{eq:assH} indeed  implies that   $\mu_{k,i}^{\bm{\varepsilon}}$   goes to $0$ as $|\bm{\varepsilon}|\to 0$.
In full generality, even if $\gamma_i^{\bm{\varepsilon}}(\delta) \to 0$, the error terms $O (\mu_{k,i}^{\bm{\varepsilon}} )$ and $O (\gamma_i^{\bm{\varepsilon}}(\delta) )$ cannot be compared  without further constraints on  the relative rates at which the window sizes $\varepsilon_k$ vanish. 
Assumption  \eqref{eq:assH} formalizes the requirement that the $i$-th escape window must not be excessively small compared to the others.
In connection with Remark \textbf{(R2)} above, it is worth noting that assumption \eqref{eq:assH} can be weakened   if one extends the asymptotic expansion in powers of $|\mathbf{K}^{\bm{\varepsilon}}|$ to higher orders than $2$  in Theorem \ref{thm:2}.  
  In addition, note that \eqref{eq.LR} provides a leveling result for the escape location probabilities:   the relative impact of the starting positions  vanishes at leading order.     
It is worth noting that, even under the assumptions of Corollary \ref{cor:3}, the   ratio $K_i^{\varepsilon_i}/\sum_{k=1}^N K_k^{\varepsilon_k}$ is not guaranteed to converge as $|\bm{\varepsilon}| \to 0$.
%

 Finally, we include a remark that points toward some interesting questions for future work.  In the regime where a specific escape window size $\varepsilon_i$ is negligible compared to the remaining sizes, more precisely, when $K_i^{\bm{\varepsilon}} = O(|\mathbf K^{\bm{\varepsilon}}|^\infty)$, an  equivalent  for  the escape probability through 
$\Sigma_\ddd^{i, \varepsilon_i}$ 
  cannot be established by means of the formula \eqref{eq:thmP}. Under these conditions, the Taylor expansion of the hitting probability in powers of $|K^{\bm{\varepsilon}}|$ vanishes identically. Consequently,  it would be interesting    to capture the hidden, smaller corrections that fall outside the reach of the techniques used in this work.

\subsection{Linear and polynomial regimes}
\label{sec.LP}

In this section, we apply Theorem \ref{thm:2} to derive precise asymptotic equivalents for the escape location probabilities under two regimes. Specifically, we examine configurations where the window sizes $\varepsilon_k$ exhibit either linear or polynomial scaling relative to one another. In these cases, assumption \eqref{eq:assH} is satisfied but  we can be much more precise in the order remaining  terms than those  given in Corollary \ref{cor:3}. 
Let us also mention that while we focus here on the escape location probabilities, we note that the mean escape time given by Theorem \ref{thm:1} can be computed in an analogous fashion under these exact same regimes. As the proofs demonstrate, these scenarios serve as illustrative examples; our framework is sufficiently flexible to accommodate other regimes, such as e.g. a combination of these two scaling regimes. The proofs of the following corollaries are performed in Section~\ref{sec:cor}. 

\begin{corollary}\label{cor:1} (Linear regime)
    Let $i\in\lb1,N\rb$. Assume there exists $(\alpha_k)_{2\leq k\leq N}\subset\R_+^*$ such that for all $k\in\lb2,N\rb$, $\varepsilon_k = \alpha_k\varepsilon_1$. Then, {for all  $\delta\in (0,2)$, the following asymptotic expansion holds as  $|\bm \varepsilon|\to 0$, uniformly in   $x\in\overline\Omega_{\bm \varepsilon}$:}
    \be\label{eq.Cor1}
    \begin{split}
        \P_x\big[X_{\tau_{\overline\Sigma_\ddd^{\bm \varepsilon}}}\in\overline\Sigma_\ddd^{i, \varepsilon_i}\big] &= \frac1N + \frac{1}{N^2\ln\varepsilon_1}\sum_{k=1}^N\ln\frac{\alpha_k\prod_{j\neq k}|x^{(j)}-x^{(k)}|}{\alpha_i\prod_{j\neq i}|x^{(j)}-x^{(i)}|}\\
        &\quad- \frac{1}{N\ln\varepsilon_1}\sum_{k=1}^N\ln\frac{|x-x^{(k)}|}{|x-x^{(i)}|} - \frac{1}{N^2\ln\varepsilon_1}\sum_{j,k=1}^Ng_{i,j,k}(x) + \frac{O_{L^\infty(\Omega_{\bm \varepsilon})}(1)}{|\ln\varepsilon_1|^{2-\delta}},
    \end{split}
    \ee
    where we recall that $g_{i,j,k}$ is given by \eqref{eq:probaThomas} or \eqref{eq:probaBoris}, and where we set $\alpha_1=1$ for consistency.
\end{corollary}

 Under this asymptotic regime, the asymptotic escape distribution is independent of the coefficients $\alpha_k$ and converges to a uniform distribution over the escape windows. 
 More precisely if we define the discrete random variable $Y(x)\in \lb1,N\rb$ by 
\begin{equation}\label{eq.Y=law}
Y(x)=i \ \text{ when } \ X_{\tau_{\overline\Sigma_\ddd^{\bm \varepsilon}}}(x)\in\overline\Sigma_\ddd^{i, \varepsilon_i},
\end{equation}
  then uniformly in  $x\in\overline\Omega_{\bm \varepsilon}$, it holds that 
 $$ \lim_{ |\bm \varepsilon|\to 0} Y(x) \overset{\text{in law}}{=} Y_{\text{uni}} ,$$
 where $Y_{\text{uni}}$ is the uniform distribution over  $  \lb1,N\rb$.  
 Moreover, as it is expected,  the formula  \eqref{eq.Cor1} is fully consistent with the fact that it always true that (this could also be directly observed from the general formula \eqref{eq:thmP}):
\begin{equation}\label{eq.sum1}
\sum_{i=1}^N \P_x\big[X_{\tau_{\overline\Sigma_\ddd^{\bm \varepsilon}}}\in\overline\Sigma_\ddd^{i, \varepsilon_i}\big] = 1, \ \forall x\in\overline\Omega .
\end{equation}
Indeed, direct computation shows that the sum of all error terms in \eqref{eq.Cor1} cancels out exactly. Namely, for all $x\in\overline\Omega_{\bm \varepsilon}
$, it holds that:
$$\sum_{i,k=1}^N \ln \frac{\alpha_k \prod_{j \neq k} |x^{(j)} - x^{(k)}|}{\alpha_i \prod_{j \neq i} |x^{(j)} - x^{(i)}|}   = 0, \  \sum_{i,k=1}^N  \ln \frac{|x - x^{(k)}|}{|x - x^{(i)}|}  =0, \text{ and }  \sum_{i,j,k=1}^N  g_{i,j,k}(x) = 0. $$
Note also that the formula  \eqref{eq.Cor1} is expressed in terms of the size of the first escape window, that is,  in terms of $\varepsilon_1$ and with the relative scaling of the other windows w.r.t. $\varepsilon_1$ (namely $\varepsilon_k = \alpha_k\varepsilon_1$). This choice is entirely arbitrary; consequently, formula \eqref{eq.Cor1} must remain invariant regardless of which window size, say the $j_0$-th, is selected as the reference.    This is indeed the case since  we have that for all  $i\in\lb1,N\rb$:
     \bes
  \varepsilon_i= (\alpha_i / \alpha_{j_0}) \varepsilon_{j_0}, \  \frac{1}{\ln\varepsilon_1} = \frac{1}{\ln\varepsilon_{j_0} - \ln \alpha_{j_0}} = \frac{1}{\ln\varepsilon_{j_0}} + \frac{O(1)}{|\ln\varepsilon_{j_0}|^2},\ \text{ and }\ \frac{\alpha_i}{\alpha_k} = \frac{\alpha_i/\alpha_{j_0}}{\alpha_k/\alpha_{j_0}}.
    \ees
    Hence for all $i\in\lb1,N\rb$, as expected, we also deduce that  {for all  $\delta\in (0,2)$,  the following asymptotic expansion holds as  $|\bm \varepsilon|\to 0$, uniformly in   $x\in\overline\Omega_{\bm \varepsilon}$:}
\bes 
    \begin{split}
        \P_x\big[X_{\tau_{\overline\Sigma_\ddd^{\bm \varepsilon}}}\in\overline\Sigma_\ddd^{i, \varepsilon_i}\big] &= \frac1N + \frac{1}{N^2\ln\varepsilon_{j_0}}\sum_{k=1}^N\ln\frac{\frac{\alpha_k}{\alpha_{j_0}}\prod_{j\neq k}|x^{(j)}-x^{(k)}|}{\frac{\alpha_i}{\alpha_{j_0}}\prod_{j\neq i}|x^{(j)}-x^{(i)}|} - \frac{1}{N\ln\varepsilon_{j_0}}\sum_{k=1}^N\ln\frac{|x-x^{(k)}|}{|x-x^{(i)}|}\\
        &\quad- \frac{1}{N^2\ln\varepsilon_{j_0}}\sum_{j,k=1}^Ng_{i,j,k}(x) + \frac{O(1)}{|\ln\varepsilon_{j_0}|^{2-\delta}}.
    \end{split}
    \ees 
    Note that we have used that    in the limit $|\bm \varepsilon|\to 0$  and uniformly in $x\in\overline\Omega_{\bm \varepsilon}$: 
    $$\frac{1}{|\ln\varepsilon_{j_0}|^2} \sum_{k=1}^N\ln\frac{|x-x^{(k)}|}{|x-x^{(i)}|}=  \frac{O(1)}{|\ln\varepsilon_{j_0}|^{2-\delta}}.$$ 
 

\begin{corollary}\label{cor:2}  (Polynomial regime)
    Let $i\in\lb1,N\rb$. Assume there exists $(\beta_k)_{2\leq k\leq N}\subset\R_+^*$ such that for all $k\in\lb2,N\rb$, $\varepsilon_k = \varepsilon_1^{1/\beta_k}$. Then, { for all  $\delta\in (0,2)$, the following asymptotic expansion holds as  $|\bm \varepsilon|\to 0$, uniformly in  $x\in\overline\Omega_{\bm \varepsilon}$:}
    \be \label{eq.Cor2}
    \begin{split}
        \P_x\big[X_{\tau_{\overline\Sigma_\ddd^{\bm \varepsilon}}}\in\overline\Sigma_\ddd^{i, \varepsilon_i}\big] &= \frac{\beta_i}{\sum_{k=1}^N\beta_k} + \frac{1}{\ln\varepsilon_1}\frac{\beta_i}{(\sum_{k=1}^N\beta_k)^2}\sum_{k=1}^N\beta_k(\beta_i - \beta_k)\ln(2|x^{(k)}-x^{(i)}|)\\
        &\quad- \frac{1}{\ln\varepsilon_1}\frac{\beta_i}{(\sum_{k=1}^N\beta_k)^2}\sum_{k=1}^N\sum_{j\notin\{i,k\}}\beta_k\beta_j\ln\frac{|x^{(j)}-x^{(i)}|}{|x^{(j)}-x^{(k)}|}\\
        &\quad- \frac{1}{\ln\varepsilon_1}\frac{\beta_i}{\sum_{k=1}^N\beta_k}\sum_{k=1}^N\beta_k\ln\frac{|x-x^{(k)}|}{|x-x^{(i)}|}\\
        &\quad- \frac{1}{\ln\varepsilon_1}\frac{\beta_i}{(\sum_{k=1}^N\beta_k)^2}\sum_{j,k=1}^N\beta_j\beta_kg_{i,j,k}(x) + \frac{O(1)}{|\ln\varepsilon_1|^{2-\delta}},
    \end{split}
    \ee
    where  we recall that $g_{i,j,k}$ is given by \eqref{eq:probaThomas} or \eqref{eq:probaBoris}, and where we set  $\beta_1=1$ for consistency.
\end{corollary}

In contrast with the linear case, the asymptotic exit distribution is no longer uniform because  the relative sizes of the windows now enter through the weights $p_i:= {\beta_i}/{\sum_{j=1}^N \beta_j}$. More precisely,  uniformly in  $x\in\overline\Omega_{\bm \varepsilon}$, it holds that (recall the definition of $Y(x)$ in \eqref{eq.Y=law}):
 $$ \lim_{ |\bm \varepsilon|\to 0} Y(x) \overset{\text{in law}}{=} Y ,$$
 where the law of $Y\in   \lb1,N\rb$ is given by $\mathbb P[Y=i]=p_i$. 
 
In line with the observations made for the linear regime, we highlight two analogous properties of the expansion \eqref{eq.Cor2}. 

First,  the formula  \eqref{eq.Cor2} is again     perfectly consistent with the normalization condition \eqref{eq.sum1}.  
Indeed,   
direct computations shows that for all $x\in\overline\Omega_{\bm \varepsilon}
$,  the sum of each error terms over $i\in\lb1,N\rb$ in \eqref{eq.Cor2} vanishes. 

Secondly, the formula  \eqref{eq.Cor2} is expressed in terms of $\varepsilon_1$  and with the relative scaling of the other windows w.r.t. $\varepsilon_1$. 
Again, since the choice of reference escape window is arbitrary,    formula \eqref{eq.Cor2} must be   unchanged even if the $j_0$-th window size is used as the basis for the scaling.    This is indeed the case since  for every $i\in\lb1,N\rb$, it holds that:
     \bes
   \varepsilon_i = \varepsilon_{j_0}^{\beta_{j_0}/\beta_i}, \   \frac{\beta_i}{\sum_{k=1}^N\beta_k} = \frac{\beta_i/\beta_{j_0}}{\sum_{k=1}^N\beta_k/\beta_{j_0}}, \ \text{ and } \ \frac{\beta_i}{\ln\varepsilon_1} = \frac{\beta_k/\beta_{j_0}}{\ln\varepsilon_{j_0}}.
   \ees
 This demonstrates the consistency of the formula \eqref{eq.Cor2}.
   


\subsection{The case of the two-dimensional unit disk}\label{sub:disk}
In the case when  $\Omega = B(0,1)$ is the two-dimensional  unit disk, we are able to explicit all the functions appearing in the previous asymptotic formulas, and especially the function $R_k^0$, $R_{\partial \Omega}$ (see Section \ref{sec.RkR}), and $G_\Omega$, $G_{\partial \Omega}$   (see \eqref{eq:Greens1} and \eqref{eq:Greens2} below).   Note that these functions indeed appear in the definitions of $g_k^{\bm \varepsilon}$ (see  \eqref{eq:espThomas}) and $g_{i,j,k}$ (see  \eqref{eq:probaThomas} or \eqref{eq:probaBoris}).   In this geometric setting, 
by a direct computations,  one has (see \eqref{eq.Nde}):
$${\forall z,\zeta \in \partial \Omega, \ \ \partial_{n(\zeta)} \ln|\zeta-z| =\frac 12} \text{ and } \int_{\partial\Omega}\ln|\zeta-z|d\sigma(\zeta)=0.$$ 
It is thus straightforward to obtain that for every $k \in \lb 1, N \rb$, $x \in \overline{\Omega}$, and $z \in \partial \Omega$, it holds that:
\bes
R_k^0(x) = \frac{1-|x|^2}{4}\  \ \text{and} \ \ R_{\partial\Omega}(x,z) = 0.
\ees
Therefore, from \eqref{eq:probaThomas}, one deduces that for all $i,j,k\in\lb1,N\rb$ and every $x\in \overline \Omega$,
$$g_{i,j,k}(x)=0,$$
an identity that one also recovers from  \eqref{eq:probaBoris}. 
In addition, using \eqref{eq:espThomas}, we deduce that for all $\forall k\in\lb1,N\rb$ and every $x\in \overline \Omega$:
\begin{equation}\label{eq.gkDISK}
 g_k^{\bm \varepsilon}(x) = \frac{1-|x|^2}{4} - \frac{\sum_{j\neq k}K_j^{\varepsilon_j} \ln|x^{(j)}-x^{(k)}|}{\sum_{j=1}^NK_j^{\varepsilon_j}},
\end{equation}


\subsection{Organization of the paper}
The remainder of this manuscript is organized as follows. 

In Section~\ref{sec:feynmankac}, we establish the rigorous mathematical link between the stochastic process \eqref{eq:SDE} and the underlying elliptic partial differential equations. We state and prove the Feynman-Kac formula under mixed boundary value conditions, while paying close attention to the local regularity issues at the junctions where the boundary condition changes type. 

Section~\ref{sec:approx} is devoted to our first approach, the approximation method. We construct explicit high-order approximating functions of the mean escape time and the escape location probabilities over the domain, study their regularity and their behavior as $|\bm \varepsilon |\to 0$. We then  employ them to provide a first complete proof of both Theorems~\ref{thm:1} and~\ref{thm:2}.

 In Section~\ref{sec:layer}, we develop our second approach, the Green method, which, we recall,  serves as a fully independent method  to derive Theorems~\ref{thm:1} and~\ref{thm:2}.  
 
  Finally, Section~\ref{sec:cor} deals with the applications of our general formulas. We show how our uniform expansions can be simplified under specific algebraic scaling regimes, providing the explicit proofs for Corollary~\ref{cor:3}   as well as Corollaries~\ref{cor:1} and \ref{cor:2} (the linear and polynomial window-scaling regimes).
%
%
%
%

\section{The Feynman-Kac formula}\label{sec:feynmankac}

 In this section, we show that the mean escape time and the espace location probabilities   satisfy elliptic mixed boundary value problems. To establish this connection, we rigorously prove a Feynman-Kac formula (see Proposition \ref{prop:FeynmanKac}). By doing so, we shift from a probabilistic framework to a deterministic study of the asymptotic behavior of solutions of mixed boundary value problems.  
 The  derivation of the Feynman-Kac formula  requires careful treatment of the solution's regularity at the junctions between Dirichlet and Neumann boundary conditions, where a well-known loss of regularity occurs, especially without so-called angle conditions.  In particular, we rely on the global continuity of the solutions over $\overline{\Omega}$ to   derive the Feynman-Kac formulas. 

\subsection{Regularity of solutions to elliptic mixed boundary value problems}\label{sub:regu}
Let $i \in \lb 1, N \rb$. For   $\bm \varepsilon\in (\mathbb R_+^*)^N$ (see \eqref{eq.epsil}), we consider the (weak) solution $\mathbf p_i^{\bm \varepsilon} \in H^1(\Omega)$ of the mixed boundary value problem:
\be\label{eq:E1}
\left\{
    \begin{aligned}
        \Delta \mathbf p_i^{\bm \varepsilon} & = 0 &&\text{ in } \Omega,\\
        \partial_n \mathbf p_i^{\bm \varepsilon} & = 0 &&\text{ on } \Sigma_\nnn^{\bm \varepsilon},\\
        \mathbf p_i^{\bm \varepsilon} &=\indic_{\Sigma_\ddd^{i, \varepsilon_i}} &&\text{ on } \Sigma_\ddd^{\bm \varepsilon}
    \end{aligned}
\right.
\ee
Note that $\mathbf p_i^{\bm \varepsilon}$ depends on the parameter $\bm \varepsilon = (\varepsilon_1, \ldots, \varepsilon_N)$ through the decomposition of $\partial \Omega$ into $\Sigma_\ddd$ and $\Sigma_\nnn^{\bm \varepsilon}$. We also introduce, for   $\bm \varepsilon\in (\mathbb R_+^*)^N$, the (weak) solution $\mathbf e^{\bm \varepsilon} \in H^1(\Omega)$ of
\be\label{eq:E2}
\left\{
    \begin{aligned}
        \Delta \mathbf e^{\bm \varepsilon} & = -1 &&\text{ in } \Omega,\\
        \partial_n \mathbf e^{\bm \varepsilon} & = 0 &&\text{ on } \Sigma_\nnn^{\bm \varepsilon},\\
        \mathbf e^{\bm \varepsilon} &=0 &&\text{ on } \Sigma_\ddd^{\bm \varepsilon}
    \end{aligned}
\right.
\ee
In general, at fixed   $\bm \varepsilon\in (\mathbb R_+^*)^N$, 
  the gradients of $\mathbf p_i^{\bm \varepsilon}$ and $\mathbf e^{\bm \varepsilon}$ exhibit singular behavior near $\partial \Sigma_\ddd$, which may lead to a loss of integrability in certain $L^p$ spaces. A classical illustration is given by the function $F : z \mapsto \Re(\sqrt{z})$, which is harmonic in the upper half-plane $\C^+ = \{z \in \C \mid \Im z > 0\}$, satisfies homogeneous Neumann boundary conditions on the positive real axis and homogeneous Dirichlet conditions on the negative real axis, yet satisfies $\nabla F \notin L^2_{\mathrm{loc}}(\partial \C^+)$. This highlights the need for careful analysis near $\partial \Sigma_\ddd$, in particular when applying integration by parts or handling traces of derivatives on $\partial \Omega$ (see, e.g. Proposition~\ref{prop:IPPhard}).

The regularity of solutions to mixed boundary value problems has been extensively studied over the past decades, especially in the absence of an angle condition between the Dirichlet and Neumann parts of the boundary ~\cite{Sa97_1,Sh68,Gr89,Br94,OtBr13,TaOtBr13}. 
Let us also mention the recent works  \cite{AlCh25,AlChMa22,Ch25} where $L^{2+\delta}$ gradient estimates are derived for  $H^1$ solutions to elliptic mixed boundary value problems  associated with operators of the form $\div( a(x)\nabla ) + b(x)\cdot \nabla $ in local Lipschitz graph domains.
 We collect regularity properties   to ensure sufficient regularity of $\nabla \mathbf p_i^{\bm \varepsilon}$ and $\nabla \mathbf e^{\bm \varepsilon}$ near $\partial \Sigma_\ddd$ that we need to perform our analysis.

 By elliptic regularity, one has (see, e.g. \cite[Example 2.2]{LaPeSh17}):
\be\label{eq:regularity}
\mathbf e^{\bm \varepsilon}, \mathbf p_i^{\bm \varepsilon} \in \ccc^\infty(\overline{\Omega} \setminus \partial \Sigma_D).
\ee
Moreover, by \cite[Theorem 1]{Gr89} (see also \cite{OtBr13,TaOtBr13}), there exists $q > 2$ such that
\be\label{eq:W1p}
\mathbf e^{\bm \varepsilon}, \mathbf p_i^{\bm \varepsilon} \in W^{1,q}(\Omega).
\ee
In addition, by \cite{Sh68} (see also \cite[Equation (0.3)]{Sa97_1}), for any $\alpha \in (0, 1/2)$, one has
\be\label{eq:holder}
\mathbf e^{\bm \varepsilon}, \mathbf p_i^{\bm \varepsilon} \in \ccc^{0,1/2-\alpha}(\overline{\Omega}).
\ee

Inversion of the boundary value systems and the application of Green's identities that we will perform require specific regularity estimates for the gradients of $\mathbf e^{\bm \varepsilon}$ and $\mathbf p_i^{\bm \varepsilon}$ on  the boundary. By \cite[Theorem 1.1]{TaOtBr13} (see also \cite[Theorem 1.2]{OtBr13}), there exists $q > 1$ such that
\be\label{eq:nontangential}
x \mapsto (\nabla \mathbf p_i^{\bm \varepsilon})^\star(x) \in L^q(\partial \Omega),
\ee
where, for a function $v : \Omega \to \mathbb{R}^2$, the non-tangential maximal function is defined by
\bes
v^\star(x)= \sup_{y\in \Gamma_\alpha(x)} | v(y)|, \ x\in \partial\Omega,
\ees
where $
\Gamma_\alpha(x)= \{y\in\Omega\ |\ |x-y|\le (1+\alpha) d(y,\partial\Omega)\}$. 
For different values of $\alpha$, the corresponding non-tangential maximal functions have equivalent $L^p$ norms \cite{To86,Br94}; hence the dependence on $\alpha$ is typically omitted.

In our setting, since $\partial \Sigma_\ddd$ has zero surface measure in $\partial \Omega$, it follows from \eqref{eq:nontangential} and \eqref{eq:regularity} that there exists $q > 1$ such that
\be\label{eq:tangential}
{\nabla \mathbf e^{\bm \varepsilon}|_{\partial \Omega} }, \nabla \mathbf p_i^{\bm \varepsilon}|_{\partial \Omega} \in L^q(\partial \Omega).
\ee
We emphasize that, in general, one cannot expect $\nabla u \in L^2(\partial \Omega)$ for solutions of mixed boundary value problems without an angle condition between the Dirichlet and Neumann parts of the boundary, as illustrated by \cite{Br94,OtBr13} and the preceding example $F$.

\subsection{The Feynman-Kac formula}
 The following proposition connects the solutions of \eqref{eq:E1} and \eqref{eq:E2} with the reflected Brownian motion  $(X_t)_{t \geq 0}$ (see \eqref{eq:SDE}). 
 
\begin{proposition}\label{prop:FeynmanKac}
For all $\bm \varepsilon = (\varepsilon_1, \ldots, \varepsilon_N) \in (\mathbb{R}_+^*)^N$, $i \in \lb 1, N \rb$, and all $x \in \overline{\Omega}$, we have:
\begin{itemize}
\item[$i)$] $\P_x[\tau_{\overline\Sigma_\ddd^{\bm \varepsilon}}<+\infty]=1$;
\item[$ii)$] $\mathbf e^{\bm \varepsilon}(x)= \E_x[\tau_{\overline\Sigma_\ddd^{\bm \varepsilon}}]$;
\item[$iii)$] $\mathbf p_i^{\bm \varepsilon}(x)= \P_x\big[X_{\tau_{\overline\Sigma_\ddd^{\bm \varepsilon}}}\in\overline\Sigma_\ddd^{i, \varepsilon_i}\big]$.
\end{itemize}
\end{proposition}
While this formula is often treated as intuitive through formal computations, providing a rigorous proof is a more delicate matter. Our derivation relies fundamentally on the non-trivial regularity properties established in \eqref{eq:regularity} and \eqref{eq:holder}. Since these two properties hold for any dimension $d \geq 2$, the assertions in Proposition \ref{prop:FeynmanKac} remain valid in an arbitrary dimensional setting

\begin{proof}
We first prove $i)$. Since $\tau_{\overline\Sigma_\ddd^{\bm \varepsilon}}(x)=0$ a.s. for all $x\in \overline\Sigma_\ddd^{\bm \varepsilon}$,  we  only have to consider initial conditions  $x \in \overline{\Omega} \setminus \overline\Sigma_\ddd^{\bm \varepsilon}$.
For $r > 0$, we set 
\bes
\vvv_r = \{x\in\overline\Omega\ |\ d(x,\overline \Sigma_\ddd^{\bm \varepsilon})\leq r\},
\ees
where $d(x,A)$ stands for the Euclidean distance between $x$ and $A\subset \mathbb R^2$. 
In what follows,  $x \in \overline{\Omega} \setminus \overline\Sigma_\ddd^{\bm \varepsilon}$ and $r > 0$ is chosen small enough such that $x \notin \vvv_r$.  
Since $\mathbf e^{\bm \varepsilon}$ is smooth in  $\overline{\Omega}\setminus\overline\Sigma_\ddd^{\bm \varepsilon}$ (see indeed \eqref{eq:regularity}),  using  It\=o’s formula together with the fact that   $\int_0^{t\wedge\tau_{\vvv_r}} \partial_n \mathbf e^{\bm \varepsilon}(X_s) d\ell_s = 0$ (see \eqref{eq:E2}),  we deduce that   for all $t \geq 0$: 
\bes
\E_x[\mathbf e^{\bm \varepsilon}(X_{t\wedge\tau_{\vvv_r}})] - \mathbf e^{\bm \varepsilon}(x) = \E_x\Big[\int_0^{t\wedge\tau_{\vvv_r}}\Delta\mathbf e^{\bm \varepsilon}(X_s)ds\Big].
\ees
Using \eqref{eq:E2} and \eqref{eq:regularity}, we have $\Delta \mathbf e^{\bm \varepsilon} = -1$ on $\overline{\Omega} \setminus \partial \Sigma_\ddd$, and therefore
\be\label{eq:e0}
\mathbf e^{\bm \varepsilon}(x) - \E_x[\mathbf e^{\bm \varepsilon}(X_{t\wedge\tau_{\vvv_r}})] = \E_x[t\wedge\tau_{\vvv_r}].
\ee
Since $\mathbf e^{\bm \varepsilon}$ is bounded on $\overline{\Omega}$ by \eqref{eq:holder}, it follows from \eqref{eq:e0} that there exists $C > 0$ such that, for all sufficiently small $r > 0$ and all $t \geq 0$,
\bes
\E_x[t \wedge \tau_{\vvv_r}] \leq C.
\ees
Letting $t\to +\infty$, we deduce that $\E_x[ \tau_{\vvv_r}] \leq C$. 
Moreover, $\tau_{\vvv_r} \to \tau_{\overline\Sigma_\ddd^{\bm \varepsilon}}$ from below almost surely as $r \to 0$. By the monotone convergence theorem, it thus follows that  $\E_x[\tau_{\overline\Sigma_\ddd^{\bm \varepsilon}}] < +\infty$. Hence, for all $x \in \overline{\Omega} \setminus \overline\Sigma_\ddd^{\bm \varepsilon}$, it holds that:
\be\label{eq:Pfini}
\P_x[\tau_{\overline\Sigma_\ddd^{\bm \varepsilon}}<+\infty] = 1.
\ee
We now prove $ii)$. Returning to \eqref{eq:e0}, we apply the dominated convergence theorem to the left-hand side (using \eqref{eq:holder}) and the monotone convergence theorem to the right-hand side, which yields by continuity of the trajectories of the process $(X_t)_{t\ge 0}$:
\bes
\mathbf e^{\bm \varepsilon}(x) - \E_x[\mathbf e^{\bm \varepsilon}(X_{t\wedge\tau_{\overline\Sigma_\ddd^{\bm \varepsilon}}})] = \E_x[t\wedge\tau_{\overline\Sigma_\ddd^{\bm \varepsilon}}].
\ees
As $t \to +\infty$ and using \eqref{eq:Pfini}, we it holds a.s. that $X_{t \wedge \tau_{\overline\Sigma_\ddd^{\bm \varepsilon}}} \to X_{\tau_{\overline\Sigma_\ddd^{\bm \varepsilon}}} \in \overline\Sigma_\ddd^{\bm \varepsilon}$. Since $\mathbf e^{\bm \varepsilon} = 0$ on $\overline\Sigma_\ddd^{\bm \varepsilon}$, we obtain
\bes
\mathbf e^{\bm \varepsilon}(x) = \E_x[\tau_{\overline\Sigma_\ddd^{\bm \varepsilon}}],
\ees
for all $x \in \overline{\Omega} \setminus \overline\Sigma_\ddd^{\bm \varepsilon}$. This equality  extends to every $x \in \overline\Sigma_\ddd^{\bm \varepsilon}$. Indeed, $\mathbf e^{\bm \varepsilon}=0$ on $\overline\Sigma_\ddd^{\bm \varepsilon}$  (by \eqref{eq:E2} and by the continuity of $\mathbf e^{\bm \varepsilon}$ over $\overline \Omega$).

We now prove $iii)$.  
By continuity of  $\mathbf p_i^{\bm \varepsilon}$ over $\overline \Omega$ (see \eqref{eq:holder}) and since $\mathbf p_i^{\bm \varepsilon}$ is smooth  in  $\overline{\Omega}\setminus \overline \Sigma_\ddd^{\bm \varepsilon}$ (see indeed \eqref{eq:regularity}), we obtain that for every  $x \in \overline{\Omega} \setminus \overline\Sigma_\ddd^{\bm \varepsilon}$ and sufficiently small $r > 0$,  
\bes
\mathbf p_i^{\bm \varepsilon}(x) = \E_x\big[\mathbf p_i^{\bm \varepsilon}(X_{\tau_{\vvv_r}})\big] = \P_x\big[X_{\tau_{\overline\Sigma_\ddd^{\bm \varepsilon}}}\in\overline\Sigma_\ddd^{i, \varepsilon_i}\big].
\ees  
Note that we have used that 
$$X_{\tau_{\overline\Sigma_\ddd^{\bm \varepsilon}}} \in \overline\Sigma_\ddd^{\bm \varepsilon} \text{ and }  \mathbf p_i^{\bm \varepsilon}=0 \text{ over } \overline\Sigma_\ddd^{\bm \varepsilon}$$
This equality extends to all $x \in \overline{\Omega}$ by continuity.
The proof of the proposition is complete. 
\end{proof}


\section{Proofs of Theorems \ref{thm:1} and \ref{thm:2} via the approximation method}
\label{sec:approx}

The goal of this section is to  to prove Theorems \ref{thm:1} and \ref{thm:2} by constructing sufficiently accurate approximations of $\mathbf e^{\bm \varepsilon}$ and $\mathbf p_i^{\bm \varepsilon}$ (see Proposition \ref{prop:FeynmanKac}). This provides the first approach we use to derive sharp asymptotics for both the mean escape time and the escape location probabilities. The method relies on constructions introduced in \cite{CaLeNo26} which  is in particular  based on a careful study of the spectral properties as $|\bm \varepsilon|\to 0$ of the generator of $(X_t)_{t\ge 0}$ (namely the Laplacian)   with  mixed boundary conditions.

\subsection{Construction of approximating functions for  the mean escape time and the escape location probabilities}

\subsubsection{Approximation of the mean escape time} 
The main ingredient in the approximation of $\mathbf e^{\bm \varepsilon}$ is the existence,  which is the objective of this section,  of a sufficiently smooth function $\phi^{\bm \varepsilon}:\overline \Omega \to \mathbb R$ (which depends on  $\bm \varepsilon$) satisfying
\be\label{eq:phi}
\left\{
    \begin{aligned}
        \Delta \phi^{\bm \varepsilon} &= -1 &&\text{in } \Omega,\\
        \partial_{n_\Omega} \phi^{\bm \varepsilon} &= 0 &&\text{on } \Sigma_\nnn^{\bm \varepsilon},\\
        \phi^{\bm \varepsilon} &= O_{L^\infty(\Sigma_\ddd^{\bm \varepsilon})}(|\mathbf K^{\bm \varepsilon}|) &&\text{on } \Sigma_\ddd^{\bm \varepsilon}
    \end{aligned}
\right.
\ee
As will be shown in Corollary~\ref{cor:phie}, such a function $\phi^{\bm \varepsilon}$ yields an accurate approximation of $\mathbf e^{\bm \varepsilon}$. It also constitutes the main building block in the approximation of the exit location probabilities $\mathbf p_i^{\bm \varepsilon}$. The construction of $\phi^{\bm \varepsilon}$ relies on the framework introduced in \cite{CaLeNo26}, which we now briefly recall.

To construct $\phi^{\bm \varepsilon}$, we introduce a family of functions $(f_k^{\varepsilon_k})_{k \in \lb 1, N \rb} \subset L^2(\Omega)$, which is the purpose of what follows.  The function $f_k^{\varepsilon_k}$ is defined on $\overline{\Omega}$ using two auxiliary functions:  a function $R_k^{\varepsilon_k}$ and $w_k^{\varepsilon_k}$.  The function $w_k^{\varepsilon_k}$ is defined with the help of a  coordinates system $\psi_k$  in a neighborhood of $x^{(k)}$  in $\partial \Omega$.  

\subsubsection{Coordinate systems $\psi_k$ and the functions $w_k^{\varepsilon_k}$}
\label{sec.psik}
To defined the functions $w_k^{\varepsilon_k}$,  we use a  coordinate systems $\Psi_k$ defined as follows. 
Let $k   \in \lb 1, N \rb$.  Since $\partial \Omega$ is $\mathcal C^\infty$ , we can consider  two  open neighborhoods,   $U_0$ and $U_k$,  of respectively   $0$ in $\mathbb{R}^2$ and  $x^{(k)}$ in $\mathbb{R}^2$, and  a $\mathcal C^\infty$ diffeomorphism  map $\Psi_k : \overline U_0 \to \overline  U_k$  such that $\Psi_k(0) = x^{(k)}$, $d_0\Psi_k$ is unitary, and 
$$
\Omega \cap U_k = \Psi_k\left( \left(\mathbb{R}^*_- \times \mathbb{R}\right) \cap U_0 \right) \quad \text{and} \quad \Sigma_\ddd^{k, \varepsilon_k} = \Psi_k\left( \{0\} \times \left(T _{k,\bm \varepsilon}^-, T _{k,\bm \varepsilon}^+\right) \right),
$$
  for some real numbers $T _{k,\bm \varepsilon}^-$, $T _{k,\bm \varepsilon}^+$ satisfying 
$|T_{k,\varepsilon_k}^-| \le T_{k,\varepsilon_k}^+$, and 
$$
T_{k,\varepsilon_k}^- = -\varepsilon_k + O(|\bm \varepsilon|^2) \quad \text{and} \quad T_{k,\varepsilon_k}^+ = \varepsilon_k + O(|\bm \varepsilon|^2).
 $$
 Above, $\mathbb{R}^*_- = \{s\in \mathbb R, s<0\}$.
 For $t \in \left(T_{k,\varepsilon_k}^-, T_{k,\varepsilon_k}^+\right)$, we set 
 $$
 \psi_k(t) = \Psi_k(0,t),$$
  which is nothing but a $\mathcal C^\infty$ parametrization of $ \Sigma_\ddd^{k, \varepsilon_k}$.  
Let $\chi_k^{\varepsilon_k} :\mathbb R\to [0, 1]$ be  a $\mathcal C^\infty$  function  such that:
$$
\chi_k^{\varepsilon_k} \equiv 1 \text{ on } \left[T_{k,\varepsilon_k}^- + \varepsilon_k^2, \, |T_{k,\varepsilon_k}^-| - \varepsilon_k^2\right] \ \text{and } \  \text{supp } \chi_k^{\varepsilon_k} \subset \left[T_{k,\varepsilon_k}^- + \varepsilon_k^2/2, \, |T_{k,\varepsilon_k}^-| - \varepsilon_k^2/2\right].
$$
For $t \in [T_{k,\varepsilon_k}^-, T_{k,\varepsilon_k}^+]$,  we set 
\be\label{eq.wke}
w_k^{\varepsilon_k}(t) = \frac{ \chi_k^{\varepsilon_k}(t)  }{\sqrt{ 1-(t/ T_{k,\varepsilon_k}^+)^2} }.    
\ee
The function   $w_k^{\varepsilon_k}$  is extended by $0$ outside of the support of $\chi_k^{\varepsilon_k}$. Note that $w_k^{\varepsilon_k}$ is  bounded over $\mathbb R$ {as it is a $\mathcal C^\infty$ function}. 
Moreover, direct computations show that (see \cite[Lemma 2.7]{CaLeNo26}):
$$
\|w_k^{\varepsilon_k}\|_{L^1(\mathbb R)} = \pi\varepsilon_k \left(1 + O \left(|{\bm \varepsilon}|^{1/2}\right)\right).
$$

\subsubsection{The functions $R_k^{\varepsilon_k}$}
The construction of $R_k^{\varepsilon_k}$ is more involved than the one for $w_k^{\varepsilon_k}$. According to the results of Section 2.2 in \cite{CaLeNo26}, the following hold true. There exists a function $R_k^{\varepsilon_k}\in W^{1,3}(\Omega) \hookrightarrow \mathcal{C}^{0,1/3}(\overline{\Omega})$ solving the following Neumann problem:
\be\label{eq:Rkepd}
\left\{
    \begin{aligned}
        \Delta R_k^{\varepsilon_k} &= -\frac{\pi}{|\Omega|} &&\text{in }\Omega,\\
        \partial_{n(\zeta)}R_k^{\varepsilon_k}(\zeta) &= \frac{-1}{\|w_k^{\varepsilon_k}\|_{L^1(\mathbb R)}} \int_{T_{k,\varepsilon_k}^-}^{T_{k,\bm  \varepsilon}^+} w_k^\varepsilon(t)\partial_{n_\Omega(\zeta)} \ln|\zeta-\psi_k(t)|   \,dt  &&\text{on }\partial\Omega\setminus \overline \Sigma_\ddd^{k, \varepsilon_k},\\
        R_k^{\varepsilon_k}(x^{(k)}) &= 0.
    \end{aligned}
\right.
\ee
Moreover, it holds $R_k^{\varepsilon_k}$ is uniformly bounded in $W^{1,3}(\Omega)$ (and thus also in $\mathcal{C}^{0,1/3}(\overline{\Omega})$) uniformly in $|\bm \varepsilon|>0$ small enough, namely  there exists $C>0$ such that for all $|\bm \varepsilon|>0$ small enough:
$$
\|R_k^{\varepsilon_k}\|_{W^{1,3}(\Omega)} \le C.
$$
Note also that the mapping (recall $\psi_k$ is a parametrization of $\Sigma_\ddd^{k, \varepsilon_k}$):
$$ \zeta \mapsto \int_{T_{k,\varepsilon_k}^-}^{T_{k,\bm  \varepsilon}^+} w_k^\varepsilon(t)\partial_{n_\Omega(\zeta)} \ln|\zeta-\psi_k(t)|   \,dt \text{ is $\mathcal C^\infty$ on $\partial \Omega\setminus \overline \Sigma_\ddd^{k, \varepsilon_k}$}. 
$$
 Hence, since  $R_k^{\varepsilon_k}\in H^1(\Omega)$, by the elliptic regularity theory for Neumann problems, it holds that:
\be\label{re.regRke}
{R_k^{\varepsilon_k}\in \ccc^\infty(\overline{\Omega} \setminus \overline\Sigma_\ddd^{k, \varepsilon_k})}.
\ee
Finally, it  is shown in \cite[Lemma 2.10]{CaLeNo26} that  (see \eqref{eq:Rok}):
\be\label{eq.Rke-0}
R_k^{\varepsilon_k}= R_k^0+ O(\varepsilon_k^{1/3}) \text{ in } W^{1,3}(\Omega).
\ee
 
\subsubsection{The functions $f_k^{\varepsilon_k}$} 
For $k   \in \lb 1, N \rb$ and $x\in \overline \Omega$,  we set 
\be\label{eq:deffk}
 f_k^{\varepsilon_k}(x) = \frac{1}{\Vert w_k^{\varepsilon_k}\Vert_{L^1(  \mathbb R)}} \int_{T_{k,\varepsilon_k}^-}^{T_{k,\bm  \varepsilon}^+}  w_k^{\varepsilon_k}(t) \ln |x-\psi_k(t)|d t + R_k^{\varepsilon_k}(x).
 \ee
 We have the following result from \cite{CaLeNo26}. 
 \begin{lemma}\label{lem:fk}
 For all $k   \in \lb 1, N \rb$ and for any {$\alpha \in (0,1/2)$}:
 $$f_k^{\varepsilon_k} \in \, { \ccc^\infty(\overline{\Omega} \setminus \overline\Sigma_\ddd^{k, \varepsilon_k})}  \,  \cap\,   H^1(\Omega)\,  \cap\,  \ccc^{0, \alpha}(\overline{\Omega}).$$
  In addition,  for all $k   \in \lb 1, N \rb$,  
 $f_k^{\varepsilon_k}= f_k^0+  O(\varepsilon_k^{1/3}) \text{ in } L^2(\Omega)$,  
 where: 
\be\label{eq:fk0}
 f_k^0(x) = \ln|x-x^{(k)}|+R_k^0(x), \ x\in \overline \Omega.
\ee  
  Moreover, we have:
\be\label{eq:fk}
\left \{
    \begin{aligned}
        \Delta f_k^{\varepsilon_k} &= -\frac{\pi}{|\Omega|} &&\text{in }  \Omega,\\
        \partial_{n_\Omega(\zeta)}f_k^{\varepsilon_k}(\zeta) &= 0 &&\text{on } \Sigma_\nnn^{\bm \varepsilon},\\
        f_k^{\varepsilon_k} &= \ln\frac{\varepsilon_k}{2}\indic_{\Sigma_\ddd^{k, \varepsilon_k}} + \sum_{k'\in\lb1,N\rb \setminus \{k\}}f_k^0(x^{(k')})\indic_{\Sigma^{k'}_D} + O_{L^\infty(\Sigma_\ddd^{\bm \varepsilon})}(|\bm \varepsilon|^{1/6}) &&\text{on } \Sigma_\ddd^{\bm \varepsilon},
    \end{aligned}
\right. 
\ee
\end{lemma}

\begin{proof} 
The fact that $f_k^{\varepsilon_k} \in \, { \ccc^\infty(\overline{\Omega} \setminus \overline\Sigma_\ddd^{k, \varepsilon_k})}  \,  \cap\,   H^1(\Omega)$ 
follows from \eqref{re.regRke},  \eqref{eq:deffk}, and the fact that $R_k^{\varepsilon_k}\in H^1(\Omega)$ and (see~\cite[Lemma 2.9]{CaLeNo26}):
\be\label{eq.Ske}
S_k^{\varepsilon_k}: x\mapsto \frac{1}{\Vert w_k^{\varepsilon_k}\Vert_{L^1(  \mathbb R)}}   \int_{T_{k,\varepsilon_k}^-}^{T_{k,\bm  \varepsilon}^+}  w_k^{\varepsilon_k}(t) \ln |x-\psi_k(t)|d t \in \ccc^\infty(\overline{\Omega} \setminus \overline\Sigma_\ddd^{k, \varepsilon_k})\cap H^1(\Omega).
\ee
It is also proved in \cite{CaLeNo26}  that 
$$S_k^{\varepsilon_k}= \ln|\cdot-x^{(k)}|+  O(\varepsilon_k^{1/3}) \text{ in } L^2(\Omega), \text{ and } R_k^{\varepsilon_k}= R_0^{\bm \varepsilon}+  O(\varepsilon_k^{1/3})\text{ in } W^{1,3}(\Omega).$$
This implies that $f_k^{\varepsilon_k}= f_k^0+  O(\varepsilon_k^{1/3}) $ in $L^2(\Omega)$. 

Let us now show that for every {$\alpha \in (0,1/2)$},
$f_k^{\varepsilon_k} \in    \ccc^{0,\alpha}(\overline{\Omega})$. 
To this end, we use \cite[Theorem 5.3]{Sh68}. 
We note that by construction 
\begin{equation}\label{eq.bdfk}
\Delta f_k^{\varepsilon_k} = -\frac{\pi}{|\Omega|}  \text{ in } \Omega, \  \text{ and } \ \partial_{n_\Omega(\zeta)} f_k^{\varepsilon_k}(\zeta) =0  \ \text{ on }\partial\Omega\setminus \overline \Sigma_\ddd^{k, \varepsilon_k}.
\end{equation} 
According to \eqref{eq.bdfk} and since $f_k^{\varepsilon_k}\in H^1(\Omega)$, it thus only remains    to study the regularity of the   function 
$$ \zeta\in \Sigma_\ddd^{k, \varepsilon_k} \mapsto  f_k^{\varepsilon_k}( \zeta)=S_k^{\varepsilon_k}( \zeta)+ R_k^{\varepsilon_k}( \zeta).$$ 
On the one hand, $R_k^{\varepsilon_k}|_{\partial \Omega}$ is continuous. On the other hand, according to the analysis led in the proof of~\cite[Lemma 2.8]{CaLeNo26}, $S_k^{\varepsilon_k}\in L^\infty(\Sigma_\ddd^{k, \varepsilon_k})$. 
Note  actually   that   $S_k^{\varepsilon_k}$ enjoys more regularity over $\Sigma_\ddd^{k, \varepsilon_k}$. Indeed, 
consider the mapping $$ u \in (T_{k,\varepsilon_k}^-, T_{k,\bm  \varepsilon}^+)   \mapsto S_k^{\varepsilon_k}(\psi_k(u))=\frac{1}{\Vert w_k^{\varepsilon_k}\Vert_{L^1(  \mathbb R)}} \int_{T_{k,\varepsilon_k}^-}^{T_{k,\bm  \varepsilon}^+}  w_k^{\varepsilon_k}(t) \ln |\psi_k(u)-\psi_k(t)|d t.$$
Note that 
\bes
\begin{split}
S_k^{\varepsilon_k}(\psi_k(s)) &= \frac{1}{2\Vert w_k^{\varepsilon_k}\Vert_{L^1(  \mathbb R)}} \int_{T_{k,\varepsilon_k}^-}^{T_{k,\bm  \varepsilon}^+}  w_k^{\varepsilon_k}(t) \ln\frac{ |\psi_k(s)-\psi_k(t)|^2}{|s-t|^2}\, d t\\
&\quad +\frac{1}{\Vert w_k^{\varepsilon_k}\Vert_{L^1(  \mathbb R)}} \int_{T_{k,\varepsilon_k}^-}^{T_{k,\bm  \varepsilon}^+}  w_k^{\varepsilon_k}(t) \ln |s-t| \, d t. 
\end{split}
\ees
According  to the analysis led in \cite[Section 13.6]{Jo82} and since $w_k^{\varepsilon_k} \in L^p((T_{k,\varepsilon_k}^-, T_{k,\bm  \varepsilon}^+))$ for every $p\in (1,+\infty)$,  
 the mapping 
$$ u \in (T_{k,\varepsilon_k}^-, T_{k,\bm  \varepsilon}^+)   \mapsto \int_{T_{k,\varepsilon_k}^-}^{T_{k,\bm  \varepsilon}^+}  w_k^{\varepsilon_k}(t) \ln |s-t|\, d t,$$
is   absolutely continuous function and its  derivative lies  in $L^p((T_{k,\varepsilon_k}^-, T_{k,\bm  \varepsilon}^+))$ for every $1<p<2$.  
In addition,  
 the mapping 
$$ s \in [T_{k,\varepsilon_k}^-, T_{k,\bm  \varepsilon}^+]   \mapsto \int_{T_{k,\varepsilon_k}^-}^{T_{k,\bm  \varepsilon}^+}  w_k^{\varepsilon_k}(t) \ln \frac{ |\psi_k(s)-\psi_k(t)|^2}{|s-t|^2} \, d t \text{ is $\mathcal C^\infty$.}$$ 
Indeed, 
$$\Phi_t(s):s\in [T_{k,\varepsilon_k}^-, T_{k,\bm  \varepsilon}^+]\mapsto \frac{ |\psi_k(s)-\psi_k(t)|^2}{|s-t|^2}  \text{ is } \mathcal C^\infty \text{ and } \exists c>0,  \inf_{s,t\in     [T_{k,\varepsilon_k}^-, T_{k,\bm  \varepsilon}^+]} 
 |\Phi_t(s)|\ge c.$$  
In conclusion, $S_k^{\varepsilon_k}:\Sigma_\ddd^{k, \varepsilon_k}\to \mathbb R$  is   absolutely continuous and  $S_k^{\varepsilon_k}\in W^{1,p}(\Sigma_\ddd^{k, \varepsilon_k})$ for every $1<p<2$.

We  now   conclude the proof concerning the Hölder continuity of $f_k^{\varepsilon_k}$. 
 Because $S_k^{\varepsilon_k}\in L^\infty(\Sigma_\ddd^{k, \varepsilon_k})$, we also get that $f_k^{\varepsilon_k}\in L^\infty(\Sigma_\ddd^{k, \varepsilon_k})$.  Since  for every $p\in (1,+\infty)$, $ L^\infty(\Sigma_\ddd^{k, \varepsilon_k})\hookrightarrow W^{-1/p, p}(\Sigma_\ddd^{k, \varepsilon_k})$,  we deduce that $f_k^{\varepsilon_k}\in  W^{-1/p, p}(\Sigma_\ddd^{k, \varepsilon_k})$. 
It thus follows from  \cite[Theorem 5.3]{Sh68} that 
$$f_k^{\varepsilon_k}\in W^{s,p}(\Omega), \ \forall p\in (4,+\infty) \text{ and } s< \frac 12 +\frac 2p.$$
Recall that $W^{s,p}(\Omega) \hookrightarrow \mathcal{C}^{0,\alpha}(\overline{\Omega})$ for all $0< \alpha \le  s -  {2}/{p}$. 
Hence, 
for every {$\alpha \in (0,1/2)$},
$f_k^{\varepsilon_k} \in    \ccc^{0,\alpha}(\overline{\Omega})$. 

Finally, the last boundary condition in \eqref{eq:fk} is proved in~\cite[Corollary 2.11]{CaLeNo26}. 
The proof of the lemma is complete.  
\end{proof}

 \subsubsection{The function $\phi^{\bm \varepsilon}$}

Using the family of functions $(f_k^{\varepsilon_k})_{k \in \lb 1, N \rb}$, we can now construct the function $\phi^{\bm \varepsilon}$ so that it satisfies  \eqref{eq:phi}. Recall that we want $\phi^{\bm \varepsilon}$ to approximate the mean escape time $\mathbf e^{\bm \varepsilon}$ over $\Omega_{\bm \varepsilon}$. We look for $\phi^{\bm \varepsilon}$ over $\overline \Omega$ in the form:
\be\label{eq:defphi}
\phi^{\bm \varepsilon} = c_{\bm \varepsilon} \Big[1 + \sum_{k=1}^N K_k^{\varepsilon_k} f_k^{\varepsilon_k} \Big(1 + \sum_{j=1}^N b_{j,k} K_j^{\varepsilon_j} \Big)\Big],
\ee
where $c_{\bm \varepsilon}$ is a constant to be determined, and where the constants $b_{j,k}$ are   independent of $\bm \varepsilon$  and are also to be determined. 

\begin{remark}\label{re.am}
Following the insights from this section's analysis, extending the asymptotic expansions of $\mathbf p_i^{\bm \varepsilon}$ and $\mathbf e^{\bm \varepsilon}$ in Theorems \ref{thm:1} and \ref{thm:2} involves generalizing the definition of $\phi^{\bm \varepsilon}$. Specifically, the term $1 + \sum_{j=1}^N b_{j,k} K_j^{\varepsilon_j}$ is replaced by a higher-degree polynomial function of the variable $\mathbf K^{\bm \varepsilon}=(K_1^{\varepsilon_1},\ldots, K_N^{\varepsilon_N})$, see the analysis  carried out in \cite{CaLeNo26}. 
\end{remark}

To find $c_{\bm \varepsilon}$ and the constants $b_{j,k}$, we use Lemma \ref{lem:fk} and we proceed as follows. 
First we have that  $\partial_{n_\Omega} \phi^{\bm \varepsilon} = 0$ on $\Sigma_\nnn^{\bm \varepsilon}$, whatever the choices of the contants $c_{\bm \varepsilon}$ and $b_{j,k}$. On the other hand,  we have: 
\bes
\Delta\phi^{\bm \varepsilon} = -c_{\bm \varepsilon}\frac{\pi}{|\Omega|}\sum_{k=1}^NK_k^{\varepsilon_k}\Big(1+\sum_{j=1}^Nb_{j,k}K_j^{\varepsilon_j}\Big) \text{ in } \Omega. 
\ees
We thus choose $c_{\bm \varepsilon}$ such that:
\be\label{eq:defceps}
c_{\bm \varepsilon}^{-1} = \frac{\pi}{|\Omega|}\sum_{k=1}^NK_k^{\varepsilon_k}\Big(1+\sum_{j=1}^Nb_{j,k}K_j^{\varepsilon_j}\Big)
\ee
so that   $\Delta \phi^{\bm \varepsilon} = -1$ in $\Omega$. Note that $c_{\bm \varepsilon} > 0$ for every sufficiently small  $|\bm \varepsilon|>0$ and that 
\be\label{eq.est0c}
c_{\bm \varepsilon}^{-1}=  \frac{\pi}{|\Omega|} \sum_{k=1}^NK_k^{\varepsilon_k} (1+ O(|\mathbf K^{\bm \varepsilon}|)),
\ee 
which  shows that $c_{\bm \varepsilon}^{-1}$ is of order $|\mathbf K^{\bm \varepsilon}|$ as $|\bm \varepsilon\to 0$.  
We now choose the constants  $b_{j,k}$   so that
\bes
\phi^{\bm \varepsilon}|_{\Sigma_\ddd^{\bm \varepsilon}} = O_{L^\infty(\Sigma_\ddd^{\bm \varepsilon})}(|\mathbf K^{\bm \varepsilon}|).
\ees
Note that thanks to \eqref{eq.est0c},  the previous condition is equivalent to 
\bes
c_{\bm \varepsilon}^{-1}\phi^{\bm \varepsilon}|_{\Sigma_\ddd^{\bm \varepsilon}} = O_{L^\infty(\Sigma_\ddd^{\bm \varepsilon})}(|\mathbf K^{\bm \varepsilon}|^2).
\ees
Let $k \in \lb 1, N \rb$. Using \eqref{eq:defphi} and \eqref{eq:fk}, we deduce that  on $\Sigma_\ddd^{k, \varepsilon_k}$:
\be\label{eq:phibord1}
\begin{split}
    c_{\bm \varepsilon}^{-1}\phi^{\bm \varepsilon}|_{\Sigma_\ddd^{k, \varepsilon_k}} &= 1 + \ln\frac{\varepsilon_k}{2} \, K_k^{\varepsilon_k}\Big(1+\sum_{j=1}^Nb_{j,k}K_j^{\varepsilon_j}\Big)\\
    &\quad+ \sum_{\ell\neq k}K^{ \varepsilon_\ell}_\ell f_\ell ^0(x^{(k)})\Big(1+\sum_{j=1}^Nb_{j,\ell}K_j^{\varepsilon_j}\Big) + O_{L^\infty(\Sigma_\ddd^{\bm \varepsilon})}(|\bm \varepsilon|^{1/6}).
\end{split}
\ee
Since by definition  $K_k^{\varepsilon_k} = \big|\ln(\varepsilon_k/2)\big|^{-1}$ (see \eqref{eq:defK}), we obtain that on $\Sigma_\ddd^{k, \varepsilon_k}$:
\be\label{eq:phibord2}
\begin{split}
    c_{\bm \varepsilon}^{-1}\phi^{\bm \varepsilon}|_{\Sigma_\ddd^{k, \varepsilon_k}} &= -\sum_{j=1}^Nb_{j,k}K_j^{\varepsilon_j} + \sum_{j\neq k} K_j^{\varepsilon_j}f_j^0(x^{(k)}) + O_{L^\infty(\Sigma_\ddd^{\bm \varepsilon})}(|\mathbf K^{\bm \varepsilon}|^{2})\\
    &= - b_{k,k}K_k^{\varepsilon_k} + \sum_{j\neq k} K_j^{\varepsilon_j}(f_j^0(x^{(k)}) - b_{j,k}) + O_{L^\infty(\Sigma_\ddd^{\bm \varepsilon})}(|\mathbf K^{\bm \varepsilon}|^{2}).
\end{split}
\ee
We then finally choose for every  $j,k\in\lb1,N\rb$:
\bes
b_{j,k} = f_j^0(x^{(k)})  \text{ if }  j\neq k,  \text{ and } b_{j,j} = 0.
\ees
We thus have the following result. 

\begin{proposition}\label{prop:phi}
Let $(f_k^{\varepsilon_k})_{k \in \lb 1, N \rb}$ be defined in \eqref{eq:deffk}. The function
$$
\phi^{\bm \varepsilon} = c_{\bm \varepsilon}\Big[1+\sum_{k=1}^{N}K_k^{\varepsilon_k}f_k^{\varepsilon_k}\Big(1+\!\!\!\!\sum_{j=1,j\neq k}^N K_j^{\varepsilon_j}f_j^0(x^{(k)})\Big)\Big],  \text{ where }
c_{\bm \varepsilon}^{-1} = \frac{\pi}{|\Omega|}\sum_{k=1}^{N}K_k^{\varepsilon_k}\Big(1+\sum_{j=1,j\neq k}^NK_j^{\varepsilon_j}f_j^0(x^{(k)})\Big),
$$
satisfies \eqref{eq:phi}. Moreover, $ \phi^{\bm \varepsilon}\in\ccc^\infty(\overline\Omega\setminus\overline\Sigma_\ddd^{\bm \varepsilon}) \cap H^1(\Omega)$ and, {for all  $\alpha\in (0,1/2)$},  $\phi^{\bm \varepsilon} \in  \ccc^{0,\frac12-\alpha}(\overline\Omega)$. 
\end{proposition}

Proposition~\ref{prop:phi} implies the following corollary which shows that  $\phi^{\bm \varepsilon}$ is a sufficiently good approximation of $\mathbf e^{\bm \varepsilon}$ (see also Remark \ref{re.am}). 

\begin{corollary}\label{cor:phie}
Let $\mathbf e^{\bm \varepsilon}$ be  the solution to \eqref{eq:E2} in $H^1(\Omega)$.
Then, we have:
\bes
\phi^{\bm \varepsilon} = \mathbf e^{\bm \varepsilon} + O_{L^\infty(\Omega)}(|\mathbf K^{\bm \varepsilon}|).
\ees 
\end{corollary}

\begin{proof}
Following the proof of Proposition~\ref{prop:FeynmanKac} $ii)$,  and since $\Delta \phi^{\bm \varepsilon} = -1$ in $\Omega$, $\phi^{\bm \varepsilon}\in\ccc^\infty(\overline\Omega\setminus\overline\Sigma_\ddd^{\bm \varepsilon})$,  and $\phi^{\bm \varepsilon}$ is continuous on $\overline{\Omega}$ (see Proposition \ref{prop:phi}), we deduce that: 
\bes
\forall x\in\overline\Omega\setminus\overline\Sigma_\ddd^{\bm \varepsilon},\ \phi^{\bm \varepsilon}(x) = \E_x[\tau_{\overline\Sigma_\ddd^{\bm \varepsilon}}] + O_{L^\infty(\Omega)}(|\mathbf K^{\bm \varepsilon}|).
\ees
The estimate trivially extends to $x \in \overline\Sigma_\ddd^{\bm \varepsilon}$, which completes the proof of the corollary.
\end{proof}


\subsubsection{Approximation of the escape location probabilities $\mathbf p_i^{\bm \varepsilon}$} 
This section introduces a construction of a sufficiently accurate approximation for the escape location probabilities $\mathbf p_i^{\bm \varepsilon}$ (see Proposition~\ref{prop:FeynmanKac} and  \eqref{eq:E1}), taking advantage of the properties of $\phi^{\bm \varepsilon}$. 
To do so, we proceed as follows. 

Since $\mathbf{p}_1 \equiv 1$ for $N=1$, we restrict our analysis to the case $N \ge 2$. For each $i \in \lb 1, N \rb$, let $\phi^{\bm \varepsilon}_i$ denote the function constructed to approximate the mean escape time from $\Omega$ when the $i$-th window, $\Sigma_\ddd^{i,\varepsilon_i}$, is removed. In this configuration, the remaining $N-1$ escape windows are $\{\Sigma_\ddd^{j, \varepsilon_j}\}_{j \neq i}$. By the previous construct, we thus consider 
\bes
\phi^{\bm \varepsilon}_i = c_{\bm \varepsilon}^i\Big[1+\sum_{k=1,k\neq i}^NK_k^{\varepsilon_k}f_k^{\varepsilon_k}\Big(1+\sum_{j=1,j\notin\{i,k\}}^NK_j^{\varepsilon_j}f_j^0(x^{(k)})\Big)\Big],
\ees
where
\bes
(c_{\bm \varepsilon}^i)^{-1} = \frac{\pi}{|\Omega|}\sum_{k=1,k\neq i}^NK_k^{\varepsilon_k}\Big(1+\sum_{j=1,j\notin\{i,k\}}^NK_j^{\varepsilon_j}f_j^0(x^{(k)})\Big),
\ees
where we recall that by convention $\sum_{\emptyset}=0$.
Note that  $(c_{\bm \varepsilon}^i)^{-1} = O (|\mathbf K^{\bm \varepsilon}|)$. 
 By Proposition~\ref{prop:phi}, we have:
\bes
\phi^{\bm \varepsilon}_i = O_{L^\infty(\Sigma_\ddd^{\bm \varepsilon}\setminus\Sigma_\ddd^{i, \varepsilon_i})}(|\mathbf K^{\bm \varepsilon}|)\qquad \text{ on } \Sigma_\ddd^{\bm \varepsilon}\setminus   \Sigma_\ddd^{i, \varepsilon_i} = \bigsqcup_{j=1,j\neq i}^N \Sigma_\ddd^{j, \varepsilon_j},
\ees
and $\partial_{n_\Omega} \phi^{\bm \varepsilon}_i= 0 \text{ on } \Sigma_\nnn^{\bm \varepsilon}$. 
Moreover, using \eqref{eq:fk}, we obtain that on $\Sigma_\ddd^{i, \varepsilon_i}$:
\be\label{eq:phiIbord}
(c_{\bm \varepsilon}^i)^{-1}{\phi^{\bm \varepsilon}_i}_{|\Sigma_\ddd^{i, \varepsilon_i}} = 1 + \sum_{k=1,k\neq i}^N K_k^{\varepsilon_k}f_k^0(x^{(i)}) + O_{L^\infty(\Sigma_\ddd^{i, \varepsilon_i})}(|\mathbf K^{\bm \varepsilon}|^2).
\ee 
We consider the function $\psi^{\bm \varepsilon}_i$ defined by:
\bes
\psi^{\bm \varepsilon}_i = \frac{\phi^{\bm \varepsilon}_i - \phi^{\bm \varepsilon}}{c_{\bm \varepsilon}^i\Big(1 + \sum_{k=1,k\neq i}^NK_k^{\varepsilon_k}f_k^0(x^{(i)})\Big)}.
\ees
Combining the previous estimates yields
\be\label{eq:psi}
\left\{
    \begin{aligned}
        \Delta \psi^{\bm \varepsilon}_i &= 0 &&\text{in } \Omega,\\
        \partial_n \psi^{\bm \varepsilon}_i &= 0 &&\text{on } \Sigma_\nnn^{\bm \varepsilon},\\
        \psi^{\bm \varepsilon}_i &= \indic_{\Sigma_\ddd^{i, \varepsilon_i}} +O_{L^\infty(\Sigma_\ddd^{\bm \varepsilon})}(|\mathbf K^{\bm \varepsilon}|^2) &&\text{on } \Sigma_\ddd^{\bm \varepsilon}
    \end{aligned}
\right.
\ee

\begin{proposition}\label{prop:psi}
For all $i \in \lb 1, N \rb$, as $\varepsilon \to 0$, we have
\bes
\mathbf p_i^{\bm \varepsilon} = \psi^{\bm \varepsilon}_i + O_{L^\infty(\Omega)}(|\mathbf K^{\bm \varepsilon}|^2).
\ees
\end{proposition}
\begin{proof}
Note that by Proposition~\ref{prop:phi} and the definition of $\psi^{\bm \varepsilon}_i$ just above,  $\psi^{\bm \varepsilon}\in\ccc^\infty(\overline\Omega\setminus\overline\Sigma_\ddd^{\bm \varepsilon})$  and $\psi^{\bm \varepsilon}$ is continuous on $\overline{\Omega}$. Using  \eqref{eq:psi}, we deduce the assertion of   Proposition \ref{prop:psi}   with the same arguments as those employed in the proof of Corollary~\ref{cor:phie}. 
\end{proof}


 
\subsection{End of the proofs of  Theorems \ref{thm:1} and \ref{thm:2} via the approximating method}\label{sec:thmThomas}

\subsubsection{A preliminary estimate}
Recall from Lemma \ref{lem:fk} that for all $k \in \lb 1, N \rb$, 
$$f_k^{\varepsilon_k} = f_k^0 + O(\varepsilon_k^{1/3})\text{ in } L^2(\Omega),$$
 where $f_k^0$ is given by \eqref{eq:fk0}. To prove our main theorems, however, we require estimates on the difference $f_k^{\varepsilon_k} - f_k^0$ in a stronger topology. This is the subject of the following lemma.

\begin{lemma}\label{le.Pes}
For all $k \in \lb 1, N \rb$ {and $\alpha \in (0,1/3]$}, it holds  as $\varepsilon_k\to 0$:
$$f_k^{\varepsilon_k} = f_k^0 + O_{L^\infty(\Omega\setminus \overline B(x^{(k)},\varepsilon_k^{1-\alpha}))}(\varepsilon_k^\alpha).$$
\end{lemma}

\begin{proof}
Let $k \in \lb 1, N \rb$. Recall the expression of $f_k^{\varepsilon_k}$ from \eqref{eq:deffk} (see also \eqref{eq.Ske} and \eqref{eq.wke}):
\bes
f_k^{\varepsilon_k}  = S_k^{\varepsilon_k}  + R_k^{\varepsilon_k}, \ \text{ where } S_k^{\varepsilon_k}(\cdot) = \frac{1}{\Vert w_k^{\varepsilon_k}\Vert_{L^1(  \mathbb R)}}   \int_{T_{k,\varepsilon_k}^-}^{T_{k,\bm  \varepsilon}^+}  w_k^{\varepsilon_k}(t) \ln |\cdot -\psi_k(t)|d t.
\ees
 Using \eqref{eq.Rke-0} and  by  Sobolev embedding $W^{1,3}(\Omega) \hookrightarrow C^{0,1/3}(\overline{\Omega})$, 
\be\label{eq.Repm}
R_k^{\varepsilon_k} = R_k^0 + O_{L^\infty(\Omega)}(\varepsilon_k^{1/3}).
\ee
We now study the convergence of  $S_k^{\varepsilon_k}(\cdot) $ towards  $  \ln|\cdot-x^{(k)}|$ in $L^\infty(\Omega\setminus B(x^{(k)},\varepsilon_k^{1-\alpha}))$ for  $\alpha \in (0,1/3)$. 
More precisely, we aim to prove  that for  every $\alpha \in (0,1)$, there exists $C>0$ such that for all $\varepsilon_k>0$ small enough, we have:
\be\label{eq.s=ln}
\forall x\in \Omega\setminus \overline B(x^{(k)},\varepsilon_k^{1-\alpha}), \ \big | S_k^{\varepsilon_k}(x)-\ln |x- x^{(k)}| \big | \le C \varepsilon_k^\alpha.
\ee 
We claim that for every $\alpha \in (0,1)$, it holds for all $\varepsilon_k>0$ small enough and uniformly   in $
x\in \Omega\setminus B(x^{(k)},\varepsilon_k^{1-\alpha})$ and  in  $t\in[T^-_{k,\bm\varepsilon},T^+_{k,\bm\varepsilon}]$:
\be \label{eq.cla1}
\ln|x-\psi_k(t)|
=
\ln|x-x^{(k)}|
+
O(\varepsilon_k^\alpha).
\ee
We prove this claim below.
Admitting \eqref{eq.cla1}, we deduce that:
\begin{equation*}
\begin{aligned}
S_k^{\bm\varepsilon}(x)
&=
\frac{\ln|x-x^{(k)}|}{ \Vert w_k^{\varepsilon_k}\Vert_{L^1(  \mathbb R)}  }
 \int_{T_{k,\varepsilon_k}^-  }^{T_{k,\varepsilon_k}^+}        
  w_k^{\varepsilon_k} (t) \, dt +
\frac{O(\varepsilon_k^\alpha) }{\Vert w_k^{\varepsilon_k}\Vert_{L^1(  \mathbb R)}} 
 \int_{T_{k,\varepsilon_k}^-  }^{T_{k,\varepsilon_k}^+}        
  w_k^{\varepsilon_k} (t) 
\,dt =\ln|x-x^{(k)}| + O(\varepsilon_k^\alpha).
\end{aligned}
\end{equation*}
 This achieves the proof of \eqref{eq.s=ln}. Together with \eqref{eq.Repm}, this ends the proof of the lemma.   It   only remains to prove the claim \eqref{eq.cla1}.
 Let $\alpha\in (0,1)$. 
We have:  
$$
|x-\psi_k(t)|
=
|x-x^{(k)}|
\left|
\frac{x-x^{(k)}}{|x-x^{(k)}|}
-
\frac{\psi_k(t)-x^{(k)}}{|x-x^{(k)}|}
\right|.
$$
Set $
e(x)
:=
 (x-x^{(k)})/ {|x-x^{(k)}|}$ 
and note that  $|e(x)|=1$. Then 
$$
\ln|x-\psi_k(t)|
=
\ln|x-x^{(k)}|
+
\ln\left|
e(x)-
v(t,x)
\right|,  \text{ where } \ 
v(t,x)
:=
\frac{\psi_k(t)-x^{(k)}}{|x-x^{(k)}|}.
$$
Since $
|\psi_k(t)-x^{(k)}|
\le C_k\varepsilon_k$ (for some $C_k>0$ independent of $t\in [T^-_{k,\bm\varepsilon},T^+_{k,\bm\varepsilon}]$) and 
$ 
|x-x^{(k)}|
\ge
\varepsilon_k^{1-\alpha}
$, 
we deduce that
$ 
|v(t,x)|
\le
C_k\varepsilon_k^\alpha$. 
As a result, 
$$
\big||e(x)-v(t,x)|-1\big|
\le |v(t,x)|
\le C_k\varepsilon_k^\alpha.
$$
Consequently, it holds that:
 $
|e(x)-v(t,x)|
=
1+O(\varepsilon_k^\alpha)$, 
uniformly with respect to $
x\in\Omega\setminus B(x^{(k)},\varepsilon_k^{1-\alpha})$ and $
t\in[T^-_{k,\bm\varepsilon},T^+_{k,\bm\varepsilon}]$. 
Therefore, we have:
$$
\ln\left|
e(x)-
 v(t,x)
\right|
=
\ln\big [1+O(\varepsilon_k^\alpha)\big ].
$$
Using the estimate $|\ln (1+z)|\le 2|z|$ for $|z|<1/2$, we get that 
$\ln|
e(x)-
 v(t,x)| 
= O(\varepsilon_k^\alpha) $.  This ends the proof of    \eqref{eq.cla1}.  
\end{proof}

\subsubsection{Proof of Theorem \ref{thm:1} using the approximation method}
We are in position to end the proof of Theorem \ref{thm:1}. 
 Recall the definition of $\Omega_{\bm \varepsilon}$  in   \eqref{eq.Oe}. 
  Using Proposition~\ref{prop:phi} together with Lemma \ref{le.Pes}, we obtain as $|\bm \varepsilon|\to 0$:
$$
    c_{\bm \varepsilon} = \frac{|\Omega|}{\pi\sum_{k=1}^NK_k^{\varepsilon_k}}\Big[1 - \frac{\sum_{{j,k=1, \, j\neq k}}^N K_k^{\varepsilon_k} K_j^{\varepsilon_j} f_j^0(x^{(k)})}{\sum_{k=1}^NK_k^{\varepsilon_k}} + O(|\mathbf K^{\bm \varepsilon}|^2)\Big]
    $$
    and
    $$
    \phi^{\bm \varepsilon} = c_{\bm \varepsilon}\Big[1+\sum_{k=1}^{N}K_k^{\varepsilon_k}f_k^0\Big(1+\sum_{{j=1,j\neq k}}^NK_j^{\varepsilon_j}f_j^0(x^{(k)})\Big)\Big] + O_{L^\infty(\Omega_{\bm \varepsilon})}(|\mathbf K^{\bm \varepsilon}|^\infty).
$$
Hence, when $|\bm{\varepsilon}|\to 0$, it holds that:
\bes
\begin{split}
    \phi^{\bm \varepsilon} &= \frac{|\Omega|}{\pi\sum_{k=1}^NK_k^{\varepsilon_k}}\Big[1 - \frac{\sum_{{j,k=1, \, j\neq k}}^N K_k^{\varepsilon_k} K_j^{\varepsilon_j} f_j^0(x^{(k)})}{\sum_{k=1}^NK_k^{\varepsilon_k}} + O(|\mathbf K^{\bm \varepsilon}|^2)\Big]\\
    &\quad \times  \left ( \Big[1+\sum_{k=1}^{N}K_k^{\varepsilon_k}f_k^0\Big(1+\sum_{{j=1,j\neq k}}^NK_j^{\varepsilon_j}f_j^0(x^{(k)})\Big)\Big] + O_{L^\infty(\Omega_{\bm \varepsilon})}(|\mathbf K^{\bm \varepsilon}|^\infty) \right ).
\end{split}
\ees
Recall that $ f_k^0(x) = \ln|x-x^{(k)}|+R_k^0(x) $, see \eqref{eq:fk0}. 
Moreover, we have   $ K_k^{\varepsilon_k}\le |x - x^{(k)}|\le C$ for all $x\in \overline{\Omega}_\varepsilon$. Hence $|\ln |x - x^{(k)}||\le \max (|\ln K_k^{\varepsilon_k}|, |\ln C|)\le |\ln K_k^{\varepsilon_k}| +  |\ln C|$. Thus, for any $\delta > 0$,    $\ln|x - x^{(k)}| = O(|\mathbf K^{\bm \varepsilon}|^{-\delta})$ uniformly on $\overline{\Omega}_\varepsilon$. Consequently,  for any $\delta > 0$, 
  \be\label{eq.app}
 f_k^0(x) = O(|\mathbf K^{\bm \varepsilon}|^{-\delta}) \text{ uniformly in } \overline{\Omega}_\varepsilon.
  \ee 
We therefore deduce  that in the limit $|\bm{\varepsilon}|\to 0$:
\bes
\begin{split}
    \phi^{\bm \varepsilon} &= \frac{|\Omega|}{\pi\sum_{k=1}^NK_k^{\varepsilon_k}}\Big[1 - \frac{\sum_{j\neq k}K_k^{\varepsilon_k} K_j^{\varepsilon_j} f_j^0(x^{(k)})}{\sum_{k=1}^NK_k^{\varepsilon_k}}\Big]\Big[1+\sum_{k=1}^{N}K_k^{\varepsilon_k}f_k^0\Big(1+\sum_{j\neq k}K_j^{\varepsilon_j}f_j^0(x^{(k)})\Big)\Big] +  {O_{L^\infty(\Omega_{\bm \varepsilon})}(|\mathbf K^{\bm \varepsilon}|)}\\
    &=
  \frac{|\Omega|}{\pi\sum_{k=1}^NK_k^{\varepsilon_k}}\Big[1 - \frac{\sum_{j\neq k}K_k^{\varepsilon_k} K_j^{\varepsilon_j} f_j^0(x^{(k)})}{\sum_{k=1}^NK_k^{\varepsilon_k}} + \sum_{k=1}^{N}K_k^{\varepsilon_k}f_k^0\Big] + O_{L^\infty(\Omega_{\bm \varepsilon})}(|\mathbf K^{\bm \varepsilon}|^{1-\delta}),
  \end{split}
\ees
for any $\delta > 0$.  Together with Corollary \ref{cor:phie}, this implies  \eqref{eq:thmE}. The proof of Theorem \ref{thm:1} is complete.

\subsubsection{Proof of Theorem \ref{thm:2} using the approximation method}
Let $i\in\lb1,N\rb$. For the exit location, we first notice that for any $x\in\overline\Omega_{\bm \varepsilon}$,
\bes
\psi^{\bm \varepsilon}_i(x) = \Big[ (c_{\bm \varepsilon}^i)^{-1}\phi^{\bm \varepsilon}_i(x) - \frac{(c_{\bm \varepsilon}^i)^{-1}}{c_{\bm \varepsilon}^{-1}}c_{\bm \varepsilon}^{-1}\phi^{\bm \varepsilon}(x)\Big]\, \frac{1}{1 + \sum_{k=1,k\neq i}^NK_k^{\varepsilon_k}f_k^0(x^{(i)})}.
\ees
For ease of notation, we set for $i,k\in\lb1,N\rb$:
\bes
F_k^{\bm \varepsilon} = \sum_{j=1,j\neq k}^NK_j^{\varepsilon_j}f_j^0(x^{(k)})\text{ and } F_{i,k}^{\bm \varepsilon} = \sum_{j=1, j\notin\{i,k\}}^NK_j^{\varepsilon_j}f_j^0(x^{(k)}).
\ees
With this notation, it follows that:
\bes
c_{\bm \varepsilon}^{-1} = \frac{\pi}{|\Omega|}\sum_{k=1}^NK_k^{\varepsilon_k}\big (1 + F_k^{\bm \varepsilon}\big )\ \text{ and }\ (c_{\bm \varepsilon}^i)^{-1} = \frac{\pi}{|\Omega|}\sum_{k=1,k\neq i}^NK_k^{\varepsilon_k}\big (1 + F_{i,k}^{\bm \varepsilon}\big ).
\ees
Note that $F_k^{\bm \varepsilon} = O(|\mathbf K^{\bm \varepsilon}|)$ and $F_{i,k}^{\bm \varepsilon} = O(|\mathbf K^{\bm \varepsilon}|)$. Hence, when $|\bm{\varepsilon}|\to 0$, it holds that:
\bes
\begin{split}
    \frac{(c_{\bm \varepsilon}^i)^{-1}}{c_{\bm \varepsilon}^{-1}} &= \frac{\sum_{k=1,k\neq i}^NK_k^{\varepsilon_k}(1 + F_{i,k}^{\bm \varepsilon})}{\sum_{k=1}^NK_k^{\varepsilon_k}(1 + F_k^{\bm \varepsilon})}\\
    &= \frac{\sum_{k=1,k\neq i}^NK_k^{\varepsilon_k}}{\sum_{k=1}^NK_k^{\varepsilon_k}}\Big[1 + \frac{\sum_{k=1,k\neq i}^NK_k^{\varepsilon_k} F_{i,k}^{\bm \varepsilon}}{\sum_{k=1,k\neq i}^NK_k^{\varepsilon_k}}\Big]\Big[1 + \frac{\sum_{k=1}^NK_k^{\varepsilon_k} F_k^{\bm \varepsilon}}{\sum_{k=1}^NK_k^{\varepsilon_k}}\Big]^{-1}\\
    &= \underbrace{\frac{\sum_{k=1,k\neq i}^NK_k^{\varepsilon_k}}{\sum_{k=1}^NK_k^{\varepsilon_k}}}_{A_i^{\bm \varepsilon}} \Big [1 + \underbrace{ \frac{\sum_{k=1,k\neq i}^NK_k^{\varepsilon_k} F_{i,k}^{\bm \varepsilon}}{\sum_{k=1,k\neq i}^NK_k^{\varepsilon_k}} }_{B_i^{\bm \varepsilon}}- \underbrace{\frac{\sum_{k=1}^NK_k^{\varepsilon_k} F_k^{\bm \varepsilon}}{\sum_{k=1}^NK_k^{\varepsilon_k}}}_{ C^{\bm \varepsilon}}\Big] + O(|\mathbf K^{\bm \varepsilon}|^2).
\end{split}
\ees
 Regarding the asymptotic behavior as $|\bm \varepsilon| \to 0$, we note that the precision for $A_i^{\bm \varepsilon}$ is limited to order $O(1)$. In contrast, $B_i^{\bm \varepsilon}$ and $C^{\bm \varepsilon}$  are  $O(|\mathbf K^{\bm \varepsilon}|)$. 
  We also set for  $x\in\overline\Omega_{\bm \varepsilon}$ (see \eqref{eq:fk0}):
\bes
J^{\bm \varepsilon}(x) = \sum_{j=1}^NK_j^{\varepsilon_j}f_j^0(x)\text{ and } J_i^{\bm \varepsilon}(x) = \sum_{j=1,j\neq i}^NK_j^{\varepsilon_j}f_j^0(x).
\ees
Note that by \eqref{eq.app}, we have for any $\delta>0$:
 \be\label{eq.app2}
|J_i^{\bm \varepsilon}(x)|+  |J^{\bm \varepsilon}(x)|= O(|\mathbf K^{\bm \varepsilon}|^{1-\delta}) \text{ uniformly in } \overline{\Omega}_{\bm \varepsilon}.
  \ee 
Using  Lemma \ref{le.Pes}, Proposition \ref{prop:phi}, and \eqref{eq.app},  we have for all $x\in\overline\Omega_{\bm \varepsilon}$:
\bes
\begin{split}
    c_{\bm \varepsilon}^{-1}\phi^{\bm \varepsilon}(x) &= 1 + \sum_{k=1}^NK_k^{\varepsilon_k}f_k^0(x)(1+F_k^{\bm \varepsilon}) + O_{L^\infty(\Omega_{\bm \varepsilon})}(|\mathbf K^{\bm \varepsilon}|^\infty)\\
    &= 1 + \sum_{k=1}^NK_k^{\varepsilon_k}f_k^0(x) + O_{L^\infty(\Omega_{\bm \varepsilon})}(|\mathbf K^{\bm \varepsilon}|^{2-\delta}) = 1 + J^{\bm \varepsilon}(x) + O_{L^\infty(\Omega_{\bm \varepsilon})}(|\mathbf  K^{\bm \varepsilon}|^{2-\delta}).
\end{split}
\ees
For the same raisons, we also deduce that for all $i\in\lb1,N\rb$ and 
$x\in\overline\Omega_{\bm \varepsilon}$, it holds that:
\bes
\begin{split}
    (c_{\bm \varepsilon}^i)^{-1}\phi^{\bm \varepsilon}_i(x)   
    &= 1+J^{\bm \varepsilon}_i(x) + O_{L^\infty(\Omega_{\bm \varepsilon})}(|\mathbf K^{\bm \varepsilon}|^{2-\delta}).
\end{split}
\ees
Combining all the previous estimates and using \eqref{eq.app2},  we obtain that for $|\bm{\varepsilon}|$ small enough,  for all $x\in\overline\Omega_{\bm \varepsilon}$: 
\bes
\begin{split}
 [1 + F_i^{\bm \varepsilon} ]  \psi_i^{\bm \varepsilon}(x)&= 1+J^{\bm \varepsilon}_i(x) + O_{L^\infty(\Omega_{\bm \varepsilon})}(|\mathbf K^{\bm \varepsilon}|^{2-\delta}) - \Big[1 + J^{\bm \varepsilon}(x) + O_{L^\infty(\Omega_{\bm \varepsilon})}(|\mathbf  K^{\bm \varepsilon}|^{2-\delta}) \Big]   \\
  &\quad \times \left [  A_i^{\bm \varepsilon}\left [1 +   B_i^{\bm \varepsilon}  -  C^{\bm \varepsilon} \right] + O(|\mathbf K^{\bm \varepsilon}|^2) \right]\\
  &=  1+J^{\bm \varepsilon}_i(x)   -  A_i^{\bm \varepsilon}[1+ B_i^{\bm \varepsilon}  -  C^{\bm \varepsilon} ]- A_i^{\bm \varepsilon} J^{\bm \varepsilon}(x)  +  O_{L^\infty(\Omega_{\bm \varepsilon})}(|\mathbf K^{\bm \varepsilon}|^{2-\delta})  . 
\end{split}
\ees
Hence, using again \eqref{eq.app2},  for every $x\in\overline\Omega_{\bm \varepsilon}$, we have:
\bes
\begin{split}
    \psi_i^{\bm \varepsilon}(x) &=   \Big(  1+J^{\bm \varepsilon}_i(x)   -  A_i^{\bm \varepsilon}[1+ B_i^{\bm \varepsilon}  -  C^{\bm \varepsilon} ]- A_i^{\bm \varepsilon} J^{\bm \varepsilon}(x)    \Big)\\
    &\quad  \times\big [ 1 - F_i^{\bm \varepsilon} + O(|\mathbf K^{\bm \varepsilon}|^2 )\big] + O_{L^\infty(\Omega_{\bm \varepsilon})}(|\mathbf K^{\bm \varepsilon}|^{2-\delta})\\ 
    &= 1- A_i^{\bm \varepsilon} -  F_i^{\bm \varepsilon} + J_i^{\bm \varepsilon}(x)  +  F_i^{\bm \varepsilon} A_i^{\bm \varepsilon}  -  A_i^{\bm \varepsilon}  J^{\bm \varepsilon}(x)-   A_i^{\bm \varepsilon}  [B_i^{\bm \varepsilon}  -  C^{\bm \varepsilon} ] + O_{L^\infty(\Omega_{\bm \varepsilon})}(|\mathbf K^{\bm \varepsilon}|^{2-\delta})
       \end{split}
\ees
Note that $1-  A_i^{\bm \varepsilon}=K_i^{\varepsilon_i}/\sum_{k=1}^NK_k^{\varepsilon_k}$. Thus, it holds that:
\be\label{eq:pieceThomas1}
\psi_i(x) = \frac{K_i^{\varepsilon_i}}{\sum_{k=1}^NK_k^{\varepsilon_k}} +\mathscr  P _i(\mathbf K^{\bm \varepsilon})(x) + \mathscr Q_i(\mathbf K^{\bm \varepsilon}) + O_{L^\infty(\Omega_{\bm \varepsilon})}(|\mathbf K^{\bm \varepsilon}|^{2-\delta})
\ee
where,  
$$
\forall x\in\overline\Omega_{\bm \varepsilon}, \ \mathscr  P_i(\mathbf K^{\bm \varepsilon})(x) = J^{\bm \varepsilon}_i(x) - A_i^{\bm \varepsilon} J^{\bm \varepsilon}(x) \, \,\,\text{ and } \,\,\,
 \mathscr Q_i(\mathbf K^{\bm \varepsilon}) = -   A_i^{\bm \varepsilon}  [B_i^{\bm \varepsilon}  -  C^{\bm \varepsilon} ]    - \frac{K_i^{\varepsilon_i}}{\sum_{k=1}^NK_k^{\varepsilon_k}}\, F_i^{\bm \varepsilon}.
$$
Using that $J^{\bm \varepsilon}(x) = J^{\bm \varepsilon}_i(x) + K_i^{\varepsilon_i} f_i^0(x)$, we have by definition of $A_i^{\bm \varepsilon}$ and $J^{\bm \varepsilon}_i$:
\be\label{eq:pieceThomas2}
\begin{split}
 \mathscr P_i(\mathbf K^{\bm \varepsilon})(x) &= \frac{K_i^{\varepsilon_i}}{\sum_{k=1}^NK_k^{\varepsilon_k}}J^{\bm \varepsilon}_i(x) - \frac{\sum_{k=1,k\neq i}^NK_k^{\varepsilon_k}}{\sum_{k=1}^NK_k^{\varepsilon_k}}\, K_i^{\varepsilon_i} f_i^0(x)\\
    &= \frac{K_i^{\varepsilon_i}}{\sum_{k=1}^NK_k^{\varepsilon_k}}\sum_{k=1,k\neq i}^NK_k^{\varepsilon_k}[f_k^0(x) - f_i^0(x)]\\
    &= \frac{K_i^{\varepsilon_i}}{\sum_{k=1}^NK_k^{\varepsilon_k}}\sum_{k=1,k\neq i}^NK_k^{\varepsilon_k}\Big [R_k^0(x) - R_i^0(x) + \ln \frac{|x-x^{(k)}|}{|x-x^{(i)}|}\Big ]\\
    &=\frac{K_i^{\varepsilon_i}}{\big(\sum_{k=1}^NK_k^{\varepsilon_k}\big)^2}\sum_{j,k=1}^NK_j^{\varepsilon_k} K_k^{\varepsilon_k} [R_k^0(x) - R_i^0(x)]   
+ \frac{K_i^{\varepsilon_i}}{\sum_{k=1}^NK_k^{\varepsilon_k}}\sum_{k=1,k\neq i}^NK_k^{\varepsilon_k}  \ln \frac{|x-x^{(k)}|}{|x-x^{(i)}|} 
\end{split}
\ee
We now expand $\mathscr Q_i(\mathbf K^{\bm \varepsilon})$ to recover the asymptotic expression for $\mathbf p_i^{\bm \varepsilon}$ provided in Theorem \ref{thm:2}. First, we note that:
\bes
\begin{split}
  \Big[\sum_{k=1}^NK_k^{\varepsilon_k} \Big]A_i^{\bm \varepsilon}  [B_i^{\bm \varepsilon}  -  C^{\bm \varepsilon} ] 
   & =\sum_{\ell=1,\ell\neq i}^NK_\ell^{\varepsilon_\ell}\, \Big[\, \frac{\sum_{k=1,k\neq i}^NK_k^{\varepsilon_k} F_{i,k}^{\bm \varepsilon}}{\sum_{k=1,k\neq i}^NK_k^{\varepsilon_k}} - \frac{\sum_{k=1}^NK_k^{\varepsilon_k} F_k^{\bm \varepsilon}}{\sum_{k=1}^NK_k^{\varepsilon_k}}\,  \Big] \\
   & = \sum_{k=1,k\neq i}^NK_k^{\varepsilon_k} F_{i,k}^{\bm \varepsilon}  -  \Big[\sum_{\ell=1,\ell\neq i}^NK_\ell^{\varepsilon_\ell}  \Big]\frac{\sum_{k=1}^NK_k^{\varepsilon_k} F_k^{\bm \varepsilon}}{\sum_{k=1}^NK_k^{\varepsilon_k}} \\
    &= \frac{1}{\sum_{k=1}^NK_k^{\varepsilon_k}}\Big[\sum_{j=1}^N\sum_{k=1,k\neq i}^NK_j^{\varepsilon_j} K_k^{\varepsilon_k} F_{i,k}^{\bm \varepsilon} - \sum_{k=1}^N\sum_{j=1,j\neq i}^NK_j^{\varepsilon_j} K_k^{\varepsilon_k} F_k^{\bm \varepsilon}\Big].
\end{split}
\ees
Note that because for  $k\neq i$, $F_{i,k}^{\bm \varepsilon} - F_k^{\bm \varepsilon} = -K_i^{\varepsilon_i} f_i^0(x^{(k)})$, it holds that  
\bes
\sum_{k=1,k\neq i}^NK_k^{\varepsilon_k}(F_{i,k}^{\bm \varepsilon} - F_k^{\bm \varepsilon}) = -\sum_{k=1,k\neq i}^NK_k^{\varepsilon_k} K_i^{\varepsilon_i} f_i^0(x^{(k)}).
\ees
Writing  $F_{i,k}^{\bm\varepsilon}= F_{i,k}^{\bm\varepsilon}
-
F_k^{\bm\varepsilon}+ F_k^{\bm\varepsilon}$, we therefore deduce that:
\bes
\begin{split}
&\sum_{j=1}^{N}\sum_{k=1,k\neq i}^{N}
K_j^{\varepsilon_j}K_k^{\varepsilon_k}
F_{i,k}^{\bm\varepsilon}
-
\sum_{k=1}^{N}\sum_{j=1,j\neq i}^{N}
K_j^{\varepsilon_j}K_k^{\varepsilon_k}
F_k^{\bm\varepsilon}
\\
&=
\sum_{j=1}^{N}\sum_{k=1,k\neq i}^{N}
K_j^{\varepsilon_j}K_k^{\varepsilon_k}
\big(
F_{i,k}^{\bm\varepsilon}
-
F_k^{\bm\varepsilon}
\big)
\\
&\quad
+
\sum_{j=1}^{N}\sum_{k=1,k\neq i}^{N}
K_j^{\varepsilon_j}K_k^{\varepsilon_k}
F_k^{\bm\varepsilon}
-
\sum_{k=1}^{N}\sum_{j=1,j\neq i}^{N}
K_j^{\varepsilon_j}K_k^{\varepsilon_k}
F_k^{\bm\varepsilon}
\\
&=
-\sum_{j=1}^{N}\sum_{k=1,k\neq i}^{N}
K_j^{\varepsilon_j}
K_k^{\varepsilon_k}
K_i^{\varepsilon_i}
f_i^0(x^{(k)}) 
-\sum_{j=1}^{N}
K_j^{\varepsilon_j}
K_i^{\varepsilon_i}
F_i^{\bm\varepsilon}
+
\sum_{k=1}^{N}
K_i^{\varepsilon_i}
K_k^{\varepsilon_k}
F_k^{\bm\varepsilon}.
\end{split}
\ees
Consequently,  factorizing by $K_i^{\varepsilon_i}$, we  thus get that the term $\Big[\sum_{k=1}^{N}K_k^{\varepsilon_k}\Big]
A_i^{\bm\varepsilon}
\big(B_i^{\bm\varepsilon}-C^{\bm\varepsilon}\big)
$ is equal to:
\bes
\begin{split}
&
\frac{K_i^{\varepsilon_i}}
{\sum_{k=1}^{N}K_k^{\varepsilon_k}}
\Bigg[
-
\sum_{j=1}^{N}
\sum_{k=1,k\neq i}^{N}
K_j^{\varepsilon_j}
K_k^{\varepsilon_k}
f_i^0(x^{(k)}) 
-
\sum_{j=1}^{N}
K_j^{\varepsilon_j}
F_i^{\bm\varepsilon}
+
\sum_{k=1}^{N}
K_k^{\varepsilon_k}
F_k^{\bm\varepsilon}
\Bigg]
\\
&=
\frac{K_i^{\varepsilon_i}}
{\sum_{k=1}^{N}K_k^{\varepsilon_k}}
\Bigg[
-
\Big(\sum_{j=1}^{N}K_j^{\varepsilon_j}\Big)
\Big(
\sum_{k=1,k\neq i}^{N}
K_k^{\varepsilon_k}
f_i^0(x^{(k)})
\Big) 
-
F_i^{\bm\varepsilon}
\Big(
\sum_{j=1}^{N}
K_j^{\varepsilon_j}
\Big)
+
\sum_{k=1}^{N}
K_k^{\varepsilon_k}
F_k^{\bm\varepsilon}
\Bigg]
\\
&=
-K_i^{\varepsilon_i}
\sum_{k=1,k\neq i}^{N}
K_k^{\varepsilon_k}
f_i^0(x^{(k)})
-
K_i^{\varepsilon_i}
F_i^{\bm\varepsilon} + 
\frac{K_i^{\varepsilon_i}}
{\sum_{k=1}^{N}K_k^{\varepsilon_k}}
\sum_{k=1}^{N}
K_k^{\varepsilon_k}
F_k^{\bm\varepsilon}.
\end{split}
\ees
Hence:
\[
\begin{aligned}
A_i^{\bm\varepsilon}
\big(
B_i^{\bm\varepsilon}
-
C^{\bm\varepsilon}
\big)
&=
-
\frac{K_i^{\varepsilon_i}}
{\sum_{k=1}^{N}K_k^{\varepsilon_k}}
\sum_{k=1,k\neq i}^{N}
K_k^{\varepsilon_k}
f_i^0(x^{(k)}) 
-
\frac{K_i^{\varepsilon_i}}
{\sum_{k=1}^{N}K_k^{\varepsilon_k}}
F_i^{\bm\varepsilon}
+
\frac{K_i^{\varepsilon_i}}
{\left(
\sum_{k=1}^{N}
K_k^{\varepsilon_k}
\right)^2}
\sum_{k=1}^{N}
K_k^{\varepsilon_k}
F_k^{\bm\varepsilon}.
\end{aligned}
\]
At this step, we then deduce that:
\be\label{eq:pieceThomas3}
\begin{split}
    \mathscr Q_i(\mathbf K^{\bm \varepsilon}) &= \textcolor{blue}{-} \frac{K_i^{\varepsilon_i}}{(\sum_{k=1}^NK_k^{\varepsilon_k})^2}\Big(\sum_{k=1}^NK_k^{\varepsilon_k} F_k^{\bm \varepsilon} - \sum_{j=1}^N\sum_{k=1,k\neq i}^NK_j^{\varepsilon_j} K_k^{\varepsilon_k} f_i^0(x^{(k)})\Big).
\end{split}
\ee
Recalling that for all $k\in\lb1,N\rb$ and $x\in\overline\Omega\setminus\{x^{(k)}\}$, $f_k^0(x) = \ln|x - x^{(k)}| + R_k^0(x)$ and $R_k^0(x^{(k)}) = 0$, we have
\be\label{eq:pieceThomas4}
\begin{split}
    &\sum_{k=1}^NK_k^{\varepsilon_k} F_k^{\bm \varepsilon} - \sum_{j=1}^N\sum_{k=1,k\neq i}^NK_j^{\varepsilon_j} K_k^{\varepsilon_k} f_i^0(x^{(k)})\\
    &\quad= \sum_{k=1}^N\sum_{j=1,j\neq k}^NK_k^{\varepsilon_k} K_j^{\varepsilon_j} f_j^0(x^{(k)}) - \sum_{j=1}^N\sum_{k=1,k\neq i}^NK_j^{\varepsilon_j} K_k^{\varepsilon_k} f_i^0(x^{(k)})\\
    &\quad= \sum_{k=1}^N\sum_{j=1,j\neq k}^NK_k^{\varepsilon_k} K_j^{\varepsilon_j}\ln|x^{(j)}-x^{(k)}| - \sum_{j=1}^N\sum_{k=1,k\neq i}^NK_j^{\varepsilon_j} K_k^{\varepsilon_k}\ln|x^{(i)}-x^{(k)}|\\
    &\quad+ \sum_{j=1}^N\sum_{k=1}^NK_k^{\varepsilon_k} K_j^{\varepsilon_j}[R_j^0(x^{(k)}) - R_i^0(x^{(k)})].
\end{split}
\ee
Observe now that
\be\label{eq:pieceThomas5}
\begin{split}
    &\sum_{k=1}^N\sum_{j=1,j\neq k}^NK_k^{\varepsilon_k} K_j^{\varepsilon_j}\ln|x^{(j)}-x^{(k)}| - \sum_{j=1}^N\sum_{k=1,k\neq i}^NK_j^{\varepsilon_j} K_k^{\varepsilon_k}\ln|x^{(i)}-x^{(k)}|\\
    &\quad= \sum_{k=1}^N\sum_{j=1,j\neq k}^NK_k^{\varepsilon_k} K_j^{\varepsilon_j}\ln|x^{(j)}-x^{(k)}| - \sum_{k=1}^N\sum_{j=1,j\neq i}^NK_j^{\varepsilon_j} K_k^{\varepsilon_k}\ln|x^{(i)}-x^{(j)}|\\
    &\quad= \sum_{k=1}^N\sum_{j=1,j\notin \{i,k\}}^NK_k^{\varepsilon_k} K_j^{\varepsilon_j}\ln|x^{(j)}-x^{(k)}| + \sum_{k=1}^NK_k^{\varepsilon_k} K_i^{\varepsilon_i}\ln|x^{(i)}-x^{(k)}|\\
    &\quad- \sum_{k=1}^N\sum_{j=1,j\notin \{i,k\}}^NK_j^{\varepsilon_j} K_k^{\varepsilon_k}\ln|x^{(i)}-x^{(j)}| - \sum_{k=1}^N (K_k^{\varepsilon_k})^2\ln|x^{(i)}-x^{(k)}|\\
    &\quad= \sum_{k=1}^N\sum_{j=1,j\notin \{i,k\}}^NK_k^{\varepsilon_k} K_j^{\varepsilon_j}\ln\frac{|x^{(j)}-x^{(k)}|}{|x^{(i)}-x^{(j)}|} +  \sum_{k=1}^N K_k^{\varepsilon_k}(K_i^{\varepsilon_i}-K_k^{\varepsilon_k})\ln|x^{(i)}-x^{(k)}|.
\end{split}
\ee
In conclusion, using \eqref{eq:pieceThomas1}, \eqref{eq:pieceThomas2}, \eqref{eq:pieceThomas3}, \eqref{eq:pieceThomas4}, and \eqref{eq:pieceThomas5}, it holds for all $x\in\overline\Omega_{\bm \varepsilon}$:
\bes
\begin{split}
 \psi_i(x) &=\frac{K_i^{\varepsilon_i}}{\sum_{k=1}^NK_k^{\varepsilon_k}} + \frac{K_i^{\varepsilon_i}}{\big(\sum_{k=1}^NK_k^{\varepsilon_k}\big)^2}\sum_{j,k=1}^NK_k^{\varepsilon_k} K_j^{\varepsilon_k} [R_j^0(x) - R_i^0(x)] \\
    &\quad 
+ \frac{K_i^{\varepsilon_i}}{\sum_{k=1}^NK_k^{\varepsilon_k}}\sum_{k=1,k\neq i}^NK_k^{\varepsilon_k}  \ln \frac{|x-x^{(k)}|}{|x-x^{(i)}|}\\
&\quad - \frac{K_i^{\varepsilon_i}}{(\sum_{k=1}^NK_k^{\varepsilon_k})^2}\Bigg(   \sum_{j,k=1}^N K_k^{\varepsilon_k} K_j^{\varepsilon_j}[R_j^0(x^{(k)}) - R_i^0(x^{(k)})] \\
&\quad + \sum_{j=1}^N\sum_{k\notin\{i,j\}}K_k^{\varepsilon_k} K_j^{\varepsilon_j}\ln\frac{|x^{(j)} - x^{(k)}|}{|x^{(i)} - x^{(j)}|} +  \sum_{k=1}^N K_k^{\varepsilon_k}(K_i^{\varepsilon_i}-K_k^{\varepsilon_k})\ln|x^{(i)}-x^{(k)}|\Bigg).
\end{split}
\ees
Note that
$$- \sum_{j=1}^N\sum_{k\notin\{i,j\}}K_k^{\varepsilon_k} K_j^{\varepsilon_j}\ln\frac{|x^{(j)} - x^{(k)}|}{|x^{(i)} - x^{(j)}|} = \sum_{k=1}^N \sum_{j \notin \{i, k\}} K_k^{\varepsilon_k} K_j^{\varepsilon_j} \ln \frac{|x^{(i)} - x^{(k)}|}{|x^{(j)} - x^{(k)}|}.$$
Using Proposition \ref{prop:psi}, we finally  obtain \eqref{eq:thmP} and \eqref{eq:probaThomas}. This achieves the proof of Theorem \ref{thm:2}.

\section{Proofs of Theorems \ref{thm:1} and \ref{thm:2} via the Green method}\label{sec:layer}

This section establishes Theorems \ref{thm:1} and \ref{thm:2} through a different  methodology. Unlike the approximation-based analysis of the preceding section, the current method relies on the structural properties of the Neumann Green functions introduced in  Section \ref{sec.RkR}.

\subsection{Strategy of the proof}
\label{sec.star2}
The cornerstone of this method lies in viewing Equations \eqref{eq:E1} and \eqref{eq:E2} as perturbations of the Neumann problem as $|\bm \varepsilon |\to 0$. This perspective naturally leads to the introduction of the Neumann Green functions defined in \eqref{eq:Greens1}. The central idea, pioneered in \cite{HoSc04} (see also \cite{HoSc14}), involves performing an integration by parts between $\mathbf e^{\bm \varepsilon}(x)$   and the Green functions  to represent the mean escape time (more) explicitly in terms of the Green functions.  We proceed the same way for the escape location probabilities $\mathbf p_i^{\bm \varepsilon}$. 
This strategy also serves as the foundation for the analysis in \cite{AmKaKaLe12}. In our case, this identity is rigorously justified in Proposition \ref{prop:IPPhard} by relying on  refined regularity results for mixed boundary value problems \cite{OtBr13,TaOtBr13}. Exploiting the boundary conditions satisfied by $\mathbf e^{\bm \varepsilon}$ and $\mathbf p_i^{\bm \varepsilon}$, these integration by parts generate a system of linear equations involving the singular boundary operator $\mathscr L_p$ introduced in Section \ref{sub:key} and traces of the normal derivatives of $\mathbf e^{\bm \varepsilon}$ and $\mathbf p_i^{\bm \varepsilon}$ over $\partial \Omega$. The proof then hinges on inverting this operator. Crucially, this requires the integration by parts formula to remain valid not only in the interior of $\Omega$ but also up to its boundary (in order to exploit the boundary conditions), as detailed in Proposition \ref{prop:IPPhard}. For that reason, one should therefore care about the existence and regularity of the traces of  $\mathbf e^{\bm \varepsilon}$ and $ \mathbf p_i^{\bm \varepsilon}$ over $\partial \Omega$. 

%


\subsection{Integration by parts formulas for smooth functions} 
\label{sec.IPPs}

We start with the following  lemma which shows that  the Green functions are symmetric in their arguments.
\begin{lemma}\label{le.sym}
   For all $x,y\in\Omega$, $x\neq y$, and all $z,\zeta\in\partial\Omega$, $z\neq\zeta$,
    \bes
    G_\Omega(x,y) = G_\Omega(y,x),\ \ \ G_{\partial\Omega}(z,\zeta) = G_{\partial\Omega}(\zeta,z).
    \ees
\end{lemma}

Note that Lemma \ref{le.sym} and \eqref{eq:Greens2} imply that   $R_{\Omega}$ (resp. $R_{\partial\Omega}$) is symmetric over $\Omega \times \Omega$  (resp. over $\partial\Omega \times \partial\Omega$).

\begin{proof}
The proof is short and rather standard. 
Let $x, y \in \Omega$ with $x \neq y$.  We apply Green's second identity to the functions $u_x(\cdot) = G_\Omega(\cdot, x)$ and $v_y(\cdot) = G_\Omega(\cdot, y)$ over the domain $\Omega$. To avoid the singularities, we consider $\Omega_{r}(x,y) = \Omega \setminus (B(x, \delta) \cup B(y, \delta))$ and let $r\to 0$.
In the domain $\Omega_\delta$, both functions are harmonic, so Green’s second identity gives:
$$0=     \int_{\partial \Omega_{r}(x,y)} \big [u_x(\zeta) \partial_{n_{\Omega_{r}(x,y)}(\zeta)} v_y(\zeta) - v_y(\zeta) \partial_{n_{\Omega_{r}(x,y)}(\zeta)} u_x(\zeta) \big] d\sigma(\zeta). $$
The boundary $\partial \Omega_{r}(x,y)$ consists of three disjoint parts: the outer boundary $\partial \Omega$, and the two small circles $S_r(x)= \partial B(x, r)$ and $S_r(y) = \partial B(y, r)$.
By \eqref{eq:Greens1}, for all  $\zeta \in \partial \Omega$,  $\partial_{n_\Omega(\zeta)} G_\Omega(\zeta, x) =\partial_{n_\Omega(\zeta)} G_\Omega(\zeta, y) = {1}/{|\partial \Omega|}$. The integral over $\partial \Omega$ is thus using the normalization in \eqref{eq:Greens1}:
$$
\frac{1}{|\partial \Omega|} \left( \int_{\partial \Omega} G_\Omega(\zeta, x) \, d\sigma - \int_{\partial \Omega} G_\Omega(\zeta, y) \, d\sigma \right)=0.$$
Moreover,  we claim that :
$$\lim_{r\to 0} I_r(y)= -u(y), \text{ with }  I_r(y):=\int_{S_r(y)}   \big [u_x(\zeta) \partial_{n_{\Omega_{r}(x,y)}(\zeta)} v_y(\zeta) - v_y(\zeta) \partial_{n_{\Omega_{r}(x,y)}(\zeta)} u_x(\zeta) \big] d\sigma(\zeta)  .$$
To prove it we write  $v(\zeta) = \frac{1}{2\pi} \ln|\zeta - y| + R_\Omega(\zeta, y)$.  We have also  for $\zeta\in  S_r(y)$, 
$$n_{\Omega_{r}(x,y)}(\zeta) = -\frac{\zeta - y}{|\zeta - y|} \text{ and }  \partial_{n_{\Omega_{r}(x,y)}(\zeta)} \ln|\zeta - y|  = -\frac{1}{r}.$$
We thus have $I_r(y)= J_r(y)+ K_r(y) + L_r(y)$, where, since both $u_x$ and $R_\Omega(\cdot,y)$ are smooth in a neighborhood of $y$, 
$$L_r(y):=\int_{S_r(y)}   \big [u_x(\zeta) \partial_{n_{\Omega_{r}(x,y)}(\zeta)} R_\Omega(\zeta,y) - R_\Omega(\zeta,y) \partial_{n_{\Omega_{r}(x,y)}(\zeta)} u(\zeta) \big] d\sigma(\zeta)  \to 0,$$
and 
$$J_r(y):= -\frac{1}{2\pi r} \int_{S_r(y)} u_x(\zeta)   d\sigma(\zeta) \text{ and } K_r(y) := \int_{S_r(y)} \frac{\ln r}{2\pi}  \partial_{n_{\Omega_{r}(x,y)}(\zeta)} u_x(\zeta) d\sigma(\zeta).$$
Since $u_x$ is smooth at  $y$,  $J_r(y)\to u(y)$, and $|\nabla u_x|$ is bounded by some $M$ in a neighborhood of $y$. Thus, $|K_r(y)| \leq   M r |\ln r|$ and therefore  $K_r(y)\to 0$. 
Similarly, we have, $\lim_{r\to 0} I_r(x) =v_y(x)$. 
Hence $u_x(y)=v_y(x)$, the desired result. This proves that $G_{\Omega}$ is symmetric over $\Omega \times \Omega$. The proof of the symmetry of $G_{\partial\Omega}$ over $\partial\Omega \times \partial\Omega$  is made similarly.   The proof of the lemma is complete. 
\end{proof}

The Green functions are particularly advantageous because they allow $u(x)$ to be expressed via integrals of $\Delta u$ and $\partial_n u$. This integral formulation is the key pillar of the approach presented here.

\begin{lemma}\label{lem:IPPeasy}
    Let $u\in \ccc^\infty(\overline\Omega)$. Then, for all $x\in\Omega$ and all $z\in\partial\Omega$,
    \bes
    u(x) = \frac{1}{|\partial\Omega|}\int_{\partial\Omega}u(\zeta)d\sigma(\zeta) + \int_\Omega G_\Omega(y,x)\Delta u(y)dy - \int_{\partial\Omega}G_\Omega(\zeta,x)\partial_{n(\zeta)}u(\zeta)d\sigma(\zeta).
    \ees
    and
    \bes
    u(z) = \frac{1}{|\partial\Omega|}\int_{\partial\Omega}u(\zeta)d\sigma(\zeta) + \int_\Omega G_{\partial\Omega}(y,z)\Delta u(y)dy - \int_{\partial\Omega}G_{\partial\Omega}(\zeta,z)\partial_{n(\zeta)}u(\zeta)d\sigma(\zeta).
    \ees
\end{lemma}

The proof is standard since $u$ is smooth,  we include   it for sake of completeness.

\begin{proof}
Let  $x\in \Omega$. Recall that $\Omega_r(x)=\Omega\setminus\overline B(x,r)$ (see Section \ref{sec.RkR}), which is   simply denoted in what follows by $\Omega_r$.   For all sufficiently small $r>0$ we have $\overline B(x,r)\subset\Omega$ and therefore $\Omega\cap\partial B(x,r) = \partial B(x,r)$. Because $\Delta_y G(y,x) = 0$ in $\Omega_r$ and since  $y\mapsto G(y,x)$ is smooth over $\Omega_r$, an integration by parts yields, using also \eqref{eq:Greens2}, 
\bes 
 \begin{split}
     -\int_{\Omega_r}G_\Omega(y,x)\Delta u(y)dy  
     &= \int_{\partial\Omega}\Big[u(z)\partial_{n_{\Omega}(z)}G_\Omega(z,x) - G_\Omega(z,x)\partial_{n_{\Omega}(z)}u(z)\Big]d\sigma(z)\\
     &\quad- \int_{\partial B(x,r)}\Big [u(z)\partial_{n_{B(x,r)}(z)}G_\Omega(z,x) - G_\Omega(z,x)\partial_{n_{B(x,r)}(z)}u(z)\Big]d\sigma(z)\\
    &= \frac{1}{|\partial\Omega|}\int_{\partial\Omega}u(z)d\sigma(z) - \int_{\partial\Omega}G_\Omega(z,x)\partial_{n_{\Omega}(z)}u(z)d\sigma(z)\\
   &\quad- \frac{1}{2\pi}\int_{\partial B(x,r)}\frac{u(z)}{|x-z|}d\sigma(z) - \int_{\partial B(x,r)}u(z)\partial_{n_{B(x,r)}(z)}R_\Omega(z,x)d\sigma(z)\\
 &\quad+ \int_{\partial B(x,r)}G_\Omega(z,x)\partial_{n_{B(x,r)}(z)}u(z)d\sigma(z). 
 \end{split}
\ees
We have 
$$ \int_{\partial B(x,r)}\ln |x-z|d\sigma(z) = 2\pi r\ln r\xrightarrow[r\to 0]{} 0.$$
Therefore, since  $y\in \overline \Omega \mapsto R_\Omega(y,x)$  is smooth  (see Section \ref{sec.RkR}) and by \eqref{eq:Greens2}, it holds that:
 \bes
 \lim_{r\to0}\int_{\partial B(x,r)}G_\Omega(z,x)\partial_{n_{B(x,r)}(z)}u(z)d\sigma(z)=0.
 \ees
In addition,  one has:
$$ \lim_{r\to0}\int_{\partial B(x,r)}u(z)\partial_{n_{B(x,r)}(z)}R_\Omega(z,x)d\sigma(z)=0 \text{ and } \lim_{r\to0}\frac{1}{2\pi}\int_{\partial B(x,r)}\frac{u(z)}{|x-z|}d\sigma(z)=u(x).$$
This achieves the proof of the first formula of the lemma when $x\in \Omega$.  
 The proof of the second formula when  $z\in\partial\Omega$
is made with the same arguments and using that  
 \bes
 \lim_{r\to 0} \frac 1\pi \int_{\Omega\cap\partial B(z,r)}\frac{u(\zeta)}{|z-\zeta|}d\sigma(\zeta)= u(z) \text{ and }  \int_{\partial\Omega}|\ln|z-\zeta||d\sigma(\zeta)<+\infty.
 \ees
 The proof of Lemma \ref{lem:IPPeasy} is complete. 
\end{proof}

The following lemma will be needed to unify the two formulas given in Lemma \ref{lem:IPPeasy}. 

\begin{lemma}\label{lem:symmetry}
    For all $x\in\Omega$ and $z\in\partial\Omega$, 
    $G_\Omega(z,x) = G_{\partial\Omega}(x,z)$. 
\end{lemma}

\begin{proof}
For  $f\in L^1(\Omega)$ and $g\in L^1(\partial \Omega)$, 
we set 
\be\label{eq:mean}
m_\Omega(f) := \frac{1}{|\Omega|}\int_{\Omega}f(y)dy \text{ and } m_{\partial \Omega}(g) := \frac{1}{|\partial \Omega|}\int_{\Omega}g(z)d\sigma (z). 
\ee 
  Let $u\in H^1(\Omega)$ be a solution to
\be\label{eq:NeumannConstant1}
\left\{
    \begin{aligned}
        \Delta u &= 1 &&\text{in }\Omega,\\ 
        \partial_{n_\Omega}u &= \frac{|\Omega|}{|\partial\Omega|} &&\text{on } \partial\Omega,
    \end{aligned}
\right.
\ee
For $f\in \ccc^\infty(\overline\Omega)$, we also consider a solution    $u_f\in H^1(\Omega)$  to
\be\label{eq:NeumannConstant2}
\left\{
    \begin{aligned}
        \Delta u_f &= f - m_\Omega(f) &&\text{in }\Omega,\\ 
        \partial_{n_\Omega} u_f &= 0 &&\text{on } \partial\Omega.
    \end{aligned}
\right.
\ee
We emphasize that solutions to the two aforementioned Neumann problems exist  since the following compatibility condition are satisfied:
\bes
\int_\Omega u = \int_{\partial\Omega}\partial_{n_\Omega}u\, d\sigma  \quad  \text{ and }  \quad \int_\Omega u_f = \int_{\partial\Omega}\partial_{n_\Omega}u_f\, d\sigma,
\ees
Furthermore, the solutions $u$ and $u_f$ are uniquely defined up to an additive constant. We leave these constants unspecified, as they do not affect the subsequent analysis. 
Note also that by  the elliptic regularity theory for Neumann problems,  $u,u_f\in \ccc^\infty(\overline\Omega)$.
In what follows,  $(x_n)_{n\geq0}\subset\Omega$ is a sequence such that $x_n\to x^*\in\partial\Omega$ as $n\to+\infty$.
 Applying Lemma \ref{lem:IPPeasy} and using \eqref{eq:Greens1} yields:
 \bes
\forall x \in \Omega, u(x) = m_{\partial\Omega}(u) + \int_\Omega G_\Omega(y,x)dy \, \text{ and } \, \forall x \in \partial \Omega,  \,  u(x) = m_{\partial\Omega}(u) + \int_\Omega G_{\partial \Omega}(y,x)dy
\ees
By continuity of $u$ at $x^*$, it then follows that:
\be\label{eq.Cv-b}
\int_{\Omega} G_\Omega(y,x_n)dy \xrightarrow[n\to +\infty]{} \int_{\Omega}G_{\partial\Omega}(y,x^*)dy.
\ee 
Applying the same argument to $u_f$, we obtain
\bes
\int_{\Omega} G_\Omega(y,x_n)[f(y)-m_\Omega(f)]dy \xrightarrow[n\to +\infty]{} \int_{\Omega}G_{\partial\Omega}(y,x^*)[f(y)-m_\Omega(f)]dy.
\ees
Combining this with \eqref{eq.Cv-b}, we deduce that for all $f \in \ccc^\infty(\overline\Omega)$,
\be\label{eq.Cv-F}
\int_{\Omega} G_\Omega(y,x_n)f(y)dy \xrightarrow[n\to +\infty]{} \int_{\Omega}G_{\partial\Omega}(y,x^*)f(y)dy.
\ee 
Let us now take $f\in\ccc^\infty_c(\Omega)$. Using \eqref{eq:Greens2}, for all $n\geq 0$, we write
\bes
\int_{\Omega} G_\Omega(y,x_n)f(y)dy = \int_{\text{supp}(f)}\Big[\frac{1}{2\pi}\ln|y-x_n|+ R_\Omega(y,x_n)\Big] f(y)dy.
\ees
We want to pass to the limit $n\to +\infty$ in the previous equality. 
Since $x_n\to x^*\in \partial\Omega$, there exist $n_0\geq 0$ and constants $0<c_0<C_0$ such that 
\bes
c_0\leq |y-x_n|\leq C_0 \text{ for all }n\geq n_0\text{ and }y\in \text{supp}(f).
\ees
Hence, by the dominated convergence theorem,
\bes
\int_{\Omega}\ln|y-x_n|f(y)dy \xrightarrow[n\to +\infty]{} \int_{\Omega}\ln|y-x^*|f(y)dy.
\ees 
Moreover, by \eqref{eq:RGreens} and elliptic estimates for Neumann problems (see e.g. \cite[Theorem 4]{Mi78} or \cite[Theorem 2.2.2.5]{Gr11}), there exists $C>0$ such that for all $y\in \text{supp}(f)$:
\bes
\Vert R_\Omega(y,\cdot)\Vert^2_{H^2(\Omega)}\leq C \Vert w_{\Omega}(y,\cdot)\Vert^2_{L^2(\partial\Omega)}\leq C\Big[1 + \int_{\partial\Omega}\frac{d\sigma(z) }{|z-y|}\Big]\leq C.
\ees
By Sobolev embedding, this implies that $R_\Omega$ is uniformly bounded on $\overline\Omega \times \text{supp}(f)$. Using again the dominated convergence theorem, together with the fact that  $R_\Omega(y,x_n) = R_\Omega(x_n,y)$ (see Lemma \ref{le.sym}), we obtain:
\bes
\int_{\Omega} R_\Omega(x_n,y)f(y)dy \xrightarrow[n\to +\infty]{} \int_{\Omega}R_\Omega(x_*,y)f(y)dy.
\ees
Passing to the limit $n\to +\infty$ in  \eqref{eq.Cv-F} yields  for all $f \in \ccc^\infty_c(\Omega)$ and all $x^* \in \partial\Omega$:
\bes
\int_{\Omega} G_\Omega(x^*,y)f(y)dy = \int_{\Omega}G_{\partial\Omega}(y,x^*)f(y)dy.
\ees
This implies that for all $ x_*\in\partial\Omega$ and $y\in\Omega$, $G_\Omega(x^*,y) = G_{\partial\Omega}(y,x^*)$. 
The proof of the lemma is complete. 
\end{proof}
%


Lemma \ref{lem:symmetry} implies the following  unified formula for smooth functions.

\begin{proposition}\label{prop:IPPeasy}
    Let $u\in\ccc^\infty(\overline\Omega)$. Consider $x\in\overline\Omega$. Then, one has
    \bes
    u(x) = \frac{1}{|\partial\Omega|}\int_{\partial\Omega}u(z)d\sigma(z) + \int_\Omega G_\Omega(x,y)\Delta u(y)dy - \int_{\partial\Omega}G_{\partial\Omega}(x,z)\partial_{n_\Omega(z)}u(z)d\sigma(z).
    \ees
\end{proposition}

\begin{proof}
This follows immediately from Lemmas \ref{lem:IPPeasy} and \ref{lem:symmetry}. 
\end{proof}

The formula provided in Proposition \ref{prop:IPPeasy} unifies the \textit{a priori} two distinct identities of Lemma \ref{lem:IPPeasy} into a single, comprehensive expression that remains valid for any $x \in \overline{\Omega}$, regardless of whether $x \in \Omega$ or $x \in \partial\Omega$. In the remainder of this work, to establish Theorems \ref{thm:1} and \ref{thm:2}, we intend to apply this integration-by-parts formula to the non-smooth functions $\mathbf e^{\bm \varepsilon}$ and $\mathbf p_i^{\bm \varepsilon}$, respectively. Consequently, the regularity assumptions must be appropriately relaxed. Special care is required in this analysis, as these functions exhibit   a lack of smoothness at the transition junctions between the Dirichlet and Neumann boundary conditions (see Section \ref{sub:regu}).

\subsection{Integration by parts formula for non-smooth functions}
\begin{proposition}\label{prop:IPPhard}
    Let $u\in H^1(\Omega)$ such that $\Delta u\in\ccc^\infty(\Omega)$ and $\partial_{n_\Omega}u\in L^q(\partial\Omega)$ for some $q>1$. Then $u\in\ccc^{0,\alpha}(\overline\Omega)$ for some $\alpha>0$, and for all $x\in\overline\Omega$,
    \bes
    u(x) = \frac{1}{|\partial\Omega|}\int_{\partial\Omega}u(z)d\sigma(z) + \int_\Omega G_\Omega(x,y)\Delta u(y)dy - \int_{\partial\Omega}G_{\partial\Omega}(x,z)\partial_{n_\Omega(z)}u(z)d\sigma(z).
    \ees
\end{proposition}
\noindent
\textbf{Note}.  The assumption that $\Delta u\in\ccc^\infty(\Omega)$ is smooth in  Proposition \ref{prop:IPPhard} can be relaxed. For instance, the assertion of Proposition \ref{prop:IPPhard}  is still valid   if  $\Delta u\in L^p(\Omega)$ for some $p\in (1,+\infty)$. This follows by a simple adaption of the proof below and more precisely by approximating  $f=\Delta u$  in $L^p(\Omega)$ by smooth functions $f_k$ over $\overline \Omega$.

\begin{proof}
Let $q>1$ and set $
\mathscr V_q = \big \{u\in H^1(\Omega)\ |\ \Delta u\in \ccc^\infty(\Omega),\partial_{n_\Omega}u\in L^q(\partial\Omega) \big  \}$. 
Let $u \in \mathscr V_q$. Set $f := \Delta u \in \ccc^\infty(\Omega)$ and $g := \partial_n u \in L^q(\partial\Omega)$. Define (see \eqref{eq:mean})  
$$\hat u := u - m_\Omega(u).$$
 Then $\hat u$ is the weak solution in $H^1(\Omega)$ to the Neumann problem with $L^q$- boundary data:
\bes
\left\{
    \begin{aligned}
        \Delta \mathbf w &= f &&\text{in }\Omega,\\
        \partial_{n_\Omega} \mathbf w &= g &&\text{on }\partial\Omega,\\
        m_\Omega({\mathbf w}) &=0.
    \end{aligned}
\right.
\ees
This problem admits a unique solution (see e.g. \cite[Section 5]{TaKiBr13}) thanks to the compatibility condition
\be\label{eq:compat}
\int_\Omega f(x)\, dx = \int_{\partial\Omega}g(z)\, d\sigma(z).
\ee 
Note indeed that $\partial_{n_\Omega} u$ is weakly defined as
\bes
\<\partial_{n_\Omega} u,v\>_{H^{-1/2}(\partial\Omega),H^{1/2}(\partial\Omega)}:=\int_\Omega\Delta u(x) \, v(x)dx + \int_\Omega\nabla u(x)\cdot\nabla v(x)dx, \ v\in H^{1/2}(\partial\Omega) .
\ees 
Taking $v \equiv 1$ yields \eqref{eq:compat}.  
Let us now consider a sequence  $(g_k )_{k \in \N} \subset \ccc^\infty(\partial\Omega)$ such that $g_k\to g$ in $L^q(\partial\Omega)$.  
Define
\bes
h_k = g_k  - m_{\partial\Omega}(g_k ) + m_{\partial\Omega}(g).
\ees
Let $\hat u_k$ denote the unique solution to the following Neumann problem in $H^1(\Omega)$ with smooth boundary data:
\bes
\left\{
    \begin{aligned}
        \Delta \mathbf w &= f &&\text{in }\Omega,\\
        \partial_{n_\Omega} \mathbf w &= h_k &&\text{on }\partial\Omega,\\
        m_\Omega({\mathbf w}) &= 0.
    \end{aligned}
\right.
\ees
Note there is indeed a unique solution   $\hat u_k$ in $H^1(\Omega)$ to the previous Neumann problem. In addition,  by the elliptic regularity  theory for Neumann problems, $\hat u_k\in \ccc^\infty(\overline\Omega)$.
We now claim  that 
\be\label{eq.kappan}
\sup_{x\in\overline\Omega}|\hat u_k(x) - \hat u(x)| \xrightarrow[k\to +\infty]{} 0.
\ee
Let us prove it. Using \cite[Proposition 2.3]{Ar18}, for all $1<p<+\infty$, there exists $C = C(p,\Omega)>0$ such that
\bes
\Vert\hat u_k - \hat u\Vert_{W^{1,p}(\Omega)} \leq C \Vert h_k - g\Vert_{W^{-1/p,p}(\partial\Omega)}.
\ees
Now, the Sobolev embeddings yields $L^q(\partial\Omega)\hookrightarrow W^{-1/p,p}(\partial\Omega)$ as soon as $p\leq2q$. Moreover, we also have $W^{1,p}(\Omega)\hookrightarrow \mathcal C^{0,\alpha}(\overline\Omega)$ as soon as $\alpha \leq 1-2/p$. Hence, considering $p\in(2,2q]$ yields
\bes
\sup_{x\in\overline\Omega}|\hat u_k(x) - \hat u(x)| \leq C \vvert{h_k - g}_{L^q(\partial\Omega)}.
\ees

Finally, using H\"older's inequality, it holds that:
\bes
\begin{split}
    \Vert g - h_k \Vert_{L^q(\partial\Omega)} &\leq \Vert g - g_k \Vert_{L^q(\partial\Omega)} + \Vert m_{\partial\Omega}(g) - m_{\partial\Omega}(g_k) \Vert_{L^q(\partial\Omega)}\\
    &= \Vert g - g_k \Vert_{L^q(\partial\Omega)} + |m_{\partial\Omega}(g) - m_{\partial\Omega}(g_k)| |\partial\Omega|^{1/q}\\
    &\leq \Vert g - g_k \Vert_{L^q(\partial\Omega)} + |\partial\Omega|^{1/q'-1}\Vert g - g_k \Vert_{L^q(\partial\Omega)} |\partial\Omega|^{1/q},
\end{split}
\ees
the claim then arises from the convergence $g_k\to g$ in $L^q(\partial\Omega)$.

We are now in  position to conclude the proof of Proposition \ref{prop:IPPhard}. Let us consider the case when $x\in \partial  \Omega$ (the case when $x\in \Omega$ is treated similarly). 
By Proposition \ref{prop:IPPeasy}, for all $x \in \partial\Omega$, we have:
\be\label{eq.LGg}
\hat u_k(x) = \frac{1}{|\partial\Omega|}\int_{\partial\Omega}\hat u_k(z)d\sigma(z) + \int_\Omega G_\Omega(x,y)\Delta\hat u_k(y)dy - \int_{\partial\Omega}G_{\partial\Omega}(x,z)h_k(z)d\sigma(z).
\ee
We want to  pass to the limit as $k \to +\infty$ in the previous equality. We first notice that  $\Delta \hat u_k = f$. It thus only remains to deal with the   boundary term. We  have by \eqref{eq:Greens1}:
\bes
\int_{\partial\Omega}G_{\partial\Omega}(x,z)h_k(z)d\sigma(z) = \int_{\partial\Omega}G_{\partial\Omega}(x,z)g_k (z)d\sigma(z).
\ees
Using Hölder's inequality, it holds that:
\begin{equation}\label{eq.int1}
\Big|\int_{\partial\Omega}\ln|x-z|(g_k (z)-g(z))d\sigma(z)\Big| \leq \Vert \ln|x-\cdot| \Vert_{L^{q'}(\partial\Omega)} \Vert g_k -g \Vert_{L^q(\partial\Omega)},
\end{equation}
which tends to $0$ as $k\to +\infty$. A similar argument applies to the regular part  $z\in\partial\Omega\mapsto R_{\partial\Omega}(x,z)$, which, we recall,  is a smooth function. Therefore,
\be\label{eq.Cvvo}
\int_{\partial\Omega}G_{\partial\Omega}(x,z)g_k^{\bm \varepsilon}(z)d\sigma(z) \xrightarrow[k\to+\infty]{} \int_{\partial\Omega}G_{\partial\Omega}(x,z)g(z)d\sigma(z)
\ee 
Finally, the convergences $\hat u_k(x) \to \hat u(x)$ and $m_{\partial\Omega}(\hat u_k) \to m_{\partial\Omega}(\hat u)$ follow from \eqref{eq.kappan}. We can thus pass to the limit $k\to +\infty$ in  \eqref{eq.LGg}  to deduce the desired result.
The proof of Proposition \ref{prop:IPPhard} is complete. 
\end{proof}

Using Proposition \ref{prop:IPPhard} together with  the regularity results \eqref{eq:regularity}, \eqref{eq:W1p}, \eqref{eq:holder}, and \eqref{eq:tangential}, we obtain the following result. 

\begin{corollary}\label{cor:IPP}
    Let $i\in\lb1,N\rb$. For every $x\in\overline\Omega$,
    \bes
    \mathbf e^{\bm \varepsilon}(x) = \frac{1}{|\partial\Omega|}\int_{\partial\Omega}\mathbf e^{\bm \varepsilon}(z)d\sigma(z) - V(x) - \int_{\partial\Omega}G_{\partial\Omega}(x,z)\partial_{n_\Omega(z)}\mathbf e^{\bm \varepsilon}(z)d\sigma(z),
    \ees
    and
    \bes
    \mathbf p_i^{\bm \varepsilon}(x) = \frac{1}{|\partial\Omega|}\int_{\partial\Omega}\mathbf p_i^{\bm \varepsilon}(z)d\sigma(z) - \int_{\partial\Omega}G_{\partial\Omega}(x,z)\partial_{n_\Omega(z)}\mathbf p_i^{\bm \varepsilon}(z)d\sigma(z).
    \ees
\end{corollary}

 These integration-by-parts formulas relating the quantities of interest $\mathbf{e}^{\bm{\varepsilon}}$ and $\mathbf{p}_i^{\bm{\varepsilon}}$ outline our remaining tasks for proving Theorem \ref{thm:1} (the strategy for Theorem \ref{thm:2} being identical). To establish Theorem \ref{thm:1}, we fully exploit the integration-by-parts formula from Corollary \ref{cor:IPP} between $\mathbf{e}^{\bm{\varepsilon}}$ and the Neumann Green functions. Specifically, this formula reduces our goal to deriving asymptotic equivalents for $\partial_{n_\Omega} \mathbf{e}^{\bm{\varepsilon}}$ and $\int_{\partial\Omega}\mathbf{e}^{\bm{\varepsilon}}(z)d\sigma(z)$. By using the boundary conditions satisfied by $\mathbf{e}^{\bm{\varepsilon}}$, we derive a linear system coupling these two quantities. Inverting this system then requires a detailed study of certain singular integral operators acting on the boundary $\partial\Omega$, that we introduce and study in the next section.

%
%
%

\subsection{The finite Hilbert and logarithmic integral operators}\label{sub:key}
Let $p\in (1,+\infty)$, and denote $A^p((-1,1))$ the set of all functions $g:[-1,1]\to\R$ that are absolutely continuous and satisfy $g'\in L^p((-1,1))$. This space is complete when endowed with the norm
\be\label{eq.Ap}
\Vert g\Vert_{A^p((-1,1))}= |g(0)| + \Vert g'\Vert_{L^p((-1,1))}.
\ee
We also consider the finite Hilbert transform $\Tr_p$, defined by
\bes
\Tr_p\Phi(s) = \frac1i\int_{-1}^1\frac{\Phi(t)}{t-s}dt,\quad \Phi\in L^p((-1,1)),\ s\in(-1,1),
\ees
and the singular integral operator $\Lr_p$, defined by
\bes
\Lr_p\Phi(s) = \int_{-1}^1\ln|s-t|\Phi(t)dt,\quad \Phi\in L^p((-1,1)),\ s\in(-1,1).
\ees
We now recall some basic properties of these operators (see e.g.  \cite[Sections 13.5 and 13.6]{Jo82}).

\begin{proposition}\label{prop:Lr}
Let $p\in (1,+\infty)$. 
The operator $\Tr_p:L^p((-1,1))\to L^p((-1,1))$ (resp. $\Lr_p$) is bounded (resp. compact). Moreover, for all $\Phi\in L^p((-1,1))$,
\be\label{eq.=abs}
\forall s\in [-1,1],\ \Lr_p\Phi(s) - \Lr_p\Phi(-1) = \frac1i\int_{-1}^s\Tr_p\Phi(t)dt.
\ee
In particular, $\Ran\Lr_p\subset A^p((-1,1))$, and $\Lr_p:L^p((-1,1))\to A^p((-1,1))$ is bounded. If, in addition, $p\in(1,2)$, then $\Lr_p:L^p((-1,1))\to A^p((-1,1))$  {is one-to-one}.  
\end{proposition}

We also refer to \cite[Section  13.6]{Jo82} and \cite[Chapter 7.3]{EsKa12} for an explicit expression of $\Lr_p^{-1}$. In particular, for $s\in(-1,1)$,
\begin{equation}\label{eq.inv1}
\Lr_p^{-1}(1)(s) = \frac{1}{\pi\ln(1/2)\sqrt{1-s^2}}.
\end{equation}
We also introduce two additional integral operators that are more directly related to our setting. First, recall that the windows $(\Sigma_\ddd^{k, \varepsilon_k} )_{k\in\lb1,N\rb}$ are defined by (see \eqref{eq.fene}):
\bes
\forall k\in\lb1,N\rb,\ \Sigma_\ddd^{k, \varepsilon_k}  = \partial\Omega\cap B(x^{(k)},\varepsilon_k).
\ees
 Since $\Omega$ is a smooth subdomain of $\mathbb R^2$, for every $k\in\lb1,N\rb$,  there exist $t_k \in \R$, $\delta_k>0$, $\varepsilon_k'\in (0,\delta_k)$, and a smooth bijection map $v_k:[t_k- \delta_k,t_k+ \delta_k]\to \mathbb R^2$ such that: 
 $$
 |v_k'| = 1, \ v_k(t_k) = x^{(k)},\  v_k\big ((t_k-\varepsilon_k',t_k+\varepsilon_k')\big )=\Sigma_\ddd^{k, \varepsilon_k}, \ v_k\big ([t_k-\varepsilon_k',t_k+\varepsilon_k']\big )= \overline \Sigma_\ddd^{k, \varepsilon_k}, $$
and as $\varepsilon_k\to 0$, 
$$ \varepsilon_k' = \varepsilon_k(1 + O(\varepsilon_k)).$$ 
The   condition $|v_k'| = 1$ above indicates that $v_k$ is   an arc-length parametrization of $\overline \Sigma_\ddd^{k, \varepsilon_k}$. 
Moreover, for  $j \neq k \in\lb1,N\rb$ we assume that $t_j\neq t_k$, and thus  every $|\bm \varepsilon|>0$ small enough,   $[t_k - \varepsilon_k', t_k + \varepsilon_k']$ and $[t_{j} - \varepsilon_{j}', t_{j} + \varepsilon_{j}']$ are disjoint.

We will also  use the rescaled parametrization of  $\overline \Sigma_\ddd^{k, \varepsilon_k} $ which is defined by:
\be\label{eq.ukc}
u_k(t) = v_k(t_k+t\varepsilon_k'), \ t \in [-1,1].
\ee
 We emphasize that   $u_k$ depends on $\varepsilon_k>0$ but not $v_k$, $\delta_k$, and $t_k$. Furthermore, since the intervals $[t_k - \varepsilon_k', t_k + \varepsilon_k']$ are pairwise disjoint, we can patch the local parametrizations $v_k|_{[t_k - \varepsilon_k', t_k + \varepsilon_k']}$ together to construct a single global smooth function $v$ on $\mathbb{R}$ such that $v = v_k$ on each $[t_k - \varepsilon_k', t_k + \varepsilon_k']$. 

\begin{remark}
Recall the definition of $\psi_k$ in Section \ref{sec.psik}. The parametrizations $\psi_k$ and $u_k$ describe the same portion of the boundary (namely $\Sigma_\ddd^{k, \varepsilon_k}$), but express it in different local coordinate systems. Since both $\psi_k$ and $v$ are smooth parametrizations of $\partial\Omega$ in a neighborhood of $x^{(k)}$, there exists a local smooth change of variables $\Upsilon_k$ such that
\bes
\psi_k(t) = v(\Upsilon_k(t)),
\ees
for $t$ in a neighborhood of $0$. By construction, we have $\Upsilon_k(0) = t_k$. Furthermore, because $d_0\Psi_k$ is unitary and $v$ is parametrized by arc length, the local coordinate induced by $\Psi_k$ agrees with the arc-length parameter at first order, yielding $\Upsilon_k'(0) = 1$. By Taylor expansion around $0$, it follows that:
\bes
\Upsilon_k(t) = t_k + t + O(t^2).
\ees
Inverting this relation gives the expansion for the inverse diffeomorphism: $\Upsilon_k^{-1}(\tau) = \tau - t_k + O((\tau - t_k)^2)$. Recalling the definition of the rescaled parametrization $u_k(s) = v(t_k + \varepsilon_k' s)$ for $s \in (-1,1)$, we can substitute $\tau = t_k + \varepsilon_k' s$ into the identity $v(\tau) = \psi_k(\Upsilon_k^{-1}(\tau))$ to obtain:
\bes
u_k(s) = \psi_k\big(\Upsilon_k^{-1}(t_k + \varepsilon_k' s)\big) = \psi_k\big(\varepsilon_k' s + O(\varepsilon_k^2)\big).
\ees
Finally, applying a standard Taylor expansion to the smooth function $\psi_k$ around the point $\varepsilon_k' s$ yields:
\bes
u_k(s) = \psi_k(\varepsilon_k' s) + O(\varepsilon_k^2),
\ees
uniformly for $s \in (-1,1)$. Thus, $u_k$ and $\psi_k$ differ only by a smooth reparametrization that is asymptotically a linear scaling.
\end{remark}

Let $j,k\in\lb1,N\rb$ with $j\neq k$. We define the linear operators
\be\label{eq.Ope1}
\Gr ^{j}(\Phi)(s) = \frac1{\varepsilon_j}\int_{-1}^1\ln\frac{|u_j(s)-u_j(t)|}{\varepsilon_j|s-t|}\Phi(t)\, dt,\quad \Phi\in L^p((-1,1)),\ s\in(-1,1),
\ee
and
\be\label{eq.Ope2}
\Hr ^{j,k}(\Phi)(s) = \frac1{|\bm \varepsilon|}\int_{-1}^1\ln\frac{|u_j(s)-u_k(t)|}{|x^{(j)} - x^{(k)}|}\Phi(t)\, dt,\quad \Phi\in L^p((-1,1)),\ s\in(-1,1).
\ee
Note that $\Gr^{k}(\Phi)(s)$ depends on $\varepsilon_k$, while $\Hr^{j,k}$ depends on the vector $\bm \varepsilon$ (see \eqref{eq.epsil}). For simplicity of notation, we omit this dependence from their respective symbols.  
We establish the following result concerning these operators.

\begin{lemma}\label{le.LeG}
Let $j,k\in\lb1,N\rb$ with $j\neq k$. For sufficiently small $|\bm \varepsilon|>0$, the operators $\Gr^{j}$ and $\Hr^{j,k}$ map $L^p((-1,1))$ into $A^p((-1,1))$. Moreover, as $|\bm \varepsilon|\to0$,
\be\label{eq.OpeE}
  \begin{split}
  \Vert \Gr^{j}\Vert_{L^p((-1,1))\to A^p((-1,1))} =  O(1) \,  \text{ and } \, \Vert\Hr^{j,k}\Vert_{L^p((-1,1))\to A^p((-1,1))} = O(1).
  \end{split}
\ee
\end{lemma}

\begin{proof}
We first consider the operator $\Gr^{j}$.  We denote its kernel by 
 \bes
 \forall s,t\in [-1,1]^2, \  g^{j}(s,t) =\frac1{2\varepsilon_j}\ln J^j(s,t),
 \ees
where   $$J^j(s,t)=\frac{|u_j(s) - u_j(t)|^2}{\varepsilon_j^2|s-t|^2} = \frac{1}{\varepsilon_j^2} \int_0^1 \int_0^1 u_j'\big(t + r(s-t)\big) \cdot u_j'\big(t + w(s-t)\big) \, dr \, dw.$$
The function $J^j(s,t)$ is thus  smooth  on $[-1,1]^2$. 
 In addition, for every $\varepsilon_j>0$, there is $c>0$ such that 
  \bes
 \min_{s,t\in[-1,1]}J^j(s,t)\geq c .
 \ees
 As a result,  for every $\varepsilon_j>0$, $g^j\in \mathcal C^\infty( [-1,1]^2)$   and $\Gr^{j}(\Phi)\in \ccc^\infty([-1,1])$. Let us now prove the first estimate in \eqref{eq.OpeE}. Recall  \eqref{eq.Ap}. 
 We have for every $s,t\in [-1,1]$:
 $$\partial_s g^j(s,t) = \frac{1}{2\varepsilon_j} \frac{\partial_s J(s,t)}{J(s,t)}.$$
Recall \eqref{eq.ukc}. Set for $r,w\in [0,1]$ and $s,t\in [-1,1]$:
$$\tau_r = t_j + \varepsilon_j'\big(t + r(s-t)\big) \quad \text{and} \quad \tau_w = t_j + \varepsilon_j'\big(t + w(s-t)\big).$$
Note that  $\partial_s \tau_r = \varepsilon_j' r$ and $\partial_s \tau_w = \varepsilon_j' w$. Hence, since $J^j(s,t) = (\varepsilon_j'/\varepsilon_j)^2 \int_0^1 \int_0^1 v'(\tau_r) \cdot v'(\tau_w) \, dr \, dw$, we have:
 $$\partial_s J(s,t) = \varepsilon_j' (\varepsilon_j'/\varepsilon_j)^2  \int_0^1 \int_0^1 \Big[ r \, v''(\tau_r) \cdot v'(\tau_w) + w \, v'(\tau_r) \cdot v''(\tau_w) \Big] \, dr \, dw.$$
Therefore, we get for all $s,t\in [-1,1]$:
\bes
 \begin{split}
 \partial_s g^j(s,t) &= \frac{1}{2\varepsilon_j} \cdot \frac{\varepsilon_j' (\varepsilon_j'/\varepsilon_j)^2\int_0^1 \int_0^1 \Big[ r \, v''(\tau_r) \cdot v'(\tau_w) + w \, v'(\tau_r) \cdot v''(\tau_w) \Big] \, dr \, dw}{(\varepsilon_j'/\varepsilon_j)^2 \int_0^1 \int_0^1 v'(\tau_r) \cdot v'(\tau_w) \, dr \, dw}\\
 &= \left(\frac{\varepsilon_j'}{2\varepsilon_j}\right) \cdot \frac{\int_0^1 \int_0^1 \Big[ r \, v''(\tau_r) \cdot v'(\tau_w) + w \, v'(\tau_r) \cdot v''(\tau_w) \Big] \, dr \, dw}{\int_0^1 \int_0^1 v'(\tau_r) \cdot v'(\tau_w) \, dr \, dw}.
 \end{split}
 \ees
 Recall that $\varepsilon_j' = \varepsilon_j(1 + O(\varepsilon_j))$. 
Note that as $\varepsilon_j \to 0$,   $\tau_r \to t_j$ and $\tau_w \to t_j$ (uniformy in $s,t\in [-1,1]$).  We claim that  as $\varepsilon_j \to 0$, uniformly in $s,t\in [-1,1]$:
 $$ \int_0^1 \int_0^1 v'(\tau_r) \cdot v'(\tau_w) \, dr \, dw =  1 + O(\varepsilon'_j).$$
Let us prove this claim. For every $s, t \in [-1,1]$ and $r, w \in [0,1]$:
\bes 
\begin{split}
 | v'(\tau_r) \cdot v'(\tau_w) - 1  | &= | v'(\tau_r) \cdot v'(\tau_w) - v'(t_j) \cdot v'(t_j) |\\
 & \leq  | v'(\tau_r) \cdot \big(v'(\tau_w) - v'(t_j)\big)  | +  | v'(t_j) \cdot \big(v'(\tau_r) - v'(t_j)\big)  | \\
 &\leq |v'(\tau_w) - v'(t_j)| + |v'(\tau_r) - v'(t_j)| \leq M_j |\tau_w - t_j| + M_j |\tau_r - t_j|,
 \end{split}
 \ees 
  where $M_j=\sup |v''|$. 
   Because $\tau_r - t_j = \varepsilon_j'\, [t + r(s-t) ]$ and $\tau_w - t_j = \varepsilon_j'[t + w(s-t)]$, 
we have for every $s, t \in [-1,1]$ and $r, w \in [0,1]$:
 $$\big| t + r(s-t) \big| \leq 3 \quad \text{and} \quad \big| t + w(s-t) \big| \leq 3$$
 Integrating over $[0,1]^2$ yields:
 $$\left| \int_0^1 \int_0^1 v'(\tau_r) \cdot v'(\tau_w) \, dr \, dw - 1 \right| \leq \int_0^1 \int_0^1 6M_j\varepsilon_j' \, dr \, dw = 6M_j\varepsilon_j'.$$
 This proves the claim. In addition, performing an asymptotic  expansion of $v''(\tau_r) \cdot v'(\tau_w)$  and $v'(\tau_r) \cdot v''(\tau_w) $ as $\varepsilon_j\to 0$ and using that  $v''(t_j) \cdot v'(t_j) =0$, we deduce that  as $\varepsilon_j \to 0$, uniformly in $s,t\in [-1,1]$:
 $$\int_0^1 \int_0^1 \Big[ r \, v''(\tau_r) \cdot v'(\tau_w) + w \, v'(\tau_r) \cdot v''(\tau_w) \Big] \, dr \, dw = O(\varepsilon_j').$$
In conclusion,  when $\varepsilon_j \to 0$, uniformly in $s,t\in [-1,1]$, we have:
 $$\partial_s g^j(s,t) = O(\varepsilon_j'),\  \text{ and thus } \Gr^{j}(\Phi)'= O(\varepsilon_j').$$
 Let us now look at the term $\Gr^{j}(\Phi)(0)$. We have using again that $v''(t_j) \cdot v'(t_j) =0$,  as $\varepsilon_j \to 0$ and  uniformly in $t\in [-1,1]$: 
 \bes 
\begin{split}
 J^j(0,t)=  \frac{|u_j(0) - u_j(t)|^2}{\varepsilon_j^2\, t^2} &=  (\varepsilon_j'/\varepsilon_j)^2   \int_0^1 \int_0^1 v'\big(t_j + \varepsilon_j' t(1-r)\big) \cdot v'\big(t_j + \varepsilon_j' t(1-w)\big) \, dr \, dw\\
 &= (\varepsilon_j'/\varepsilon_j)^2\big[1+ O(\varepsilon_j')\big].
 \end{split}
 \ees 
 This proves that uniformly in $t\in [-1,1]$, $g^{j}(0)=O(1)$. Hence, $\Gr^{j}(\Phi)(0)=O(1)$.   This completes the proof of the uniform bounds for   operator $\Gr^{j}$ asserted in \eqref{eq.OpeE}. 
  We now consider the the operator $\Hr^{j,k}$.
 Since $|x^{(j)}-x^{(k)}|>0$, for $|\bm \varepsilon|>0$ sufficiently small we have:
 \be\label{eq.dml}
 \min_{s,t\in[-1,1]}|u_j(s)-u_k(t)|\geq c.
 \ee
 It follows that $\Hr^{j,k}(\Phi)\in \ccc^\infty([-1,1])$.
Let us denote its kernel by:
 $$
\forall s,t\in [-1,1]^2, \ h^{j,k}(s,t) = \frac{1}{|\bm \varepsilon|} \ln \frac{|u_j(s) - u_k(t)|}{|x^{(j)} - x^{(k)}|}.
 $$
When   $|\bm \varepsilon|\to0$ we have uniformly in $s,t\in [-1,1]$:
\bes
\begin{split}
 \Hr^{j,k}(\Phi)(0)  &= \frac1{ |\bm \varepsilon|} \int_{-1}^1\Big( \ln \frac{ |  x^{(k)}-u_j(0)+  u_k(t)-x^{(k)} |}{|x^{(j)} - x^{(k)}|}\Big)\Phi(t)\, dt \\
    &=\frac1{ |\bm \varepsilon|}\int_{-1}^1 \ln(1+ O(\bm \varepsilon) ) \,   \Phi(t)\, dt.
\end{split}
\ees
Hence, $ \Hr^{j,k}(\Phi)(0) =O(1)$. 
Moreover, in the limit $|\bm \varepsilon| \to 0$ and uniformly in $s,t\in [-1,1]$, it holds that:
 $$\partial_s h^{j,k}(s,t) = \frac{1}{|\bm \varepsilon|} \frac{u_j'(s) \cdot (u_j(s) - u_k(t))}{|u_j(s) - u_k(t)|^2}=O(1).$$
 This completes the proof of the uniform bounds for   operator $\mathscr H^{j,k} $ asserted in \eqref{eq.OpeE}.  
 The proof of the lemma is complete.
 \end{proof}

In what follows, when viewing the operators \eqref{eq.Ope1} and \eqref{eq.Ope2} on $L^p((-1,1))$, we denote them respectively by  $\Gr^{j}_p$ and
 $ \Hr_p ^{j,k}$.  In addition,  with these two operators, we can write for every   $j\neq k\in\lb1,N\rb$, $\Phi\in L^p((-1,1))$, and $s\in(-1,1)$,
\be\label{eq:key1}
\int_{-1}^1\ln|u_j(s)-u_j(t)|\Phi(t)dt = \ln\varepsilon_j\int_{-1}^1\Phi(t)dt + \Lr_p(\Phi)(s) + \varepsilon_j\Gr_p^{j}(\Phi)(s).
\ee
 and  
\be\label{eq:key2}
\int_{-1}^1\ln |u_j(s)-u_k(t)|\Phi(t)dt = \ln|x^{(j)} - x^{(k)}|\int_{-1}^1\Phi(t)dt +  |\bm \varepsilon|\Hr_p^{j,k}(\Phi)(s).
\ee

\subsection{End of the proofs  of Theorems \ref{thm:1} and \ref{thm:2} via the Green method}\label{sec:thmBoris}
In this section, we prove Theorems \ref{thm:1} and \ref{thm:2} using the integration-by-parts formulas connecting $\mathbf{e}^{\bm{\varepsilon}}$ and $\mathbf{p}_i^{\bm{\varepsilon}}$. , we can now outline the proof strategy for Theorem \ref{thm:1}, which carries over identically to Theorem \ref{thm:2}. The proof of Theorem \ref{thm:1} relies on fully exploiting the integration-by-parts relation in Corollary \ref{cor:IPP} between $\mathbf{e}^{\bm{\varepsilon}}$ and the Green functions. In view of this identity, our primary objective is to establish asymptotic expansions for $\partial_{n_\Omega} \mathbf{e}^{\bm{\varepsilon}}$ and $\int_{\partial\Omega}\mathbf{e}^{\bm{\varepsilon}}(z)d\sigma(z)$. Through the boundary conditions of $\mathbf{e}^{\bm{\varepsilon}}$, we obtain a linear system involving these functions. To invert this system and complete the proof, we must analyze the properties of specific singular integral operators acting on the boundary of $\Omega$

\subsubsection{Study of the mean escape time}
\label{eq.Sec:thm1-}
Recall \eqref{eq:E2}.  With the notation introduced in \eqref{eq:mean}, we have:
\bes
m_{\partial\Omega}(\mathbf e^{\bm \varepsilon}) = \frac{1}{|\partial\Omega|}\int_{\partial\Omega}\mathbf e^{\bm \varepsilon}(z)d\sigma(z).
\ees
Thanks to \eqref{eq:tangential}, we can also perfectly define 
\bes
\mathbf T^{\bm \varepsilon}_j = \frac1\pi\int_{\Sigma_\ddd^{j, \varepsilon_j}}\partial_{n_\Omega(z)} \mathbf e^{\bm \varepsilon}(z)d\sigma(z).
\ees
The compatibility equation $
\int_\Omega\Delta\mathbf e^{\bm \varepsilon}(x)dx = \int_{\partial\Omega}\partial_{n_\Omega(z)}\mathbf e^{\bm \varepsilon}(z)d\sigma(z)$ 
implies that
\be\label{eq:sumtracee}
\sum_{j=1}^N\mathbf T^{\bm \varepsilon}_j = -\frac{|\Omega|}{\pi}.
\ee
We now introduce the following functions:
\bes
\forall t\in [-1,1], \ \Phi_j^{\varepsilon}(t) =  \varepsilon_j'\, \partial_{n_\Omega(u_j (t))}\mathbf e^{\bm \varepsilon}(u_j(t)).
\ees
With this notation, it holds that 
 $ \mathbf T^{\bm \varepsilon}_j = \frac 1\pi \int_{-1}^1\Phi_j^{\varepsilon}(t)\, dt$. 
Recall \eqref{eq:Greens2}. 
Let $s\in (-1,1)$. Applying Corollary \ref{cor:IPP} with $x\in\{u_1(s),\ldots,u_N(s)\}$, and the Dirichlet boundary condition satisfied by $\mathbf e^{\bm \varepsilon}$ (see \eqref{eq:E2}), we obtain the system, $\forall j\in\lb1,N\rb$:
\bes
\begin{split}
      0 &= m_{\partial\Omega}(\mathbf e^{\bm \varepsilon}) - \int_\Omega G_\Omega(u_j(s),y)dy\\
    &\quad- \frac1\pi\sum_{k=1}^N\int_{\Sigma_\ddd^{k, \varepsilon_k}}\ln|u_j(s)-z|\partial_{n_\Omega(z)}\mathbf e^{\bm \varepsilon}(z)d\sigma(z) - \sum_{k=1}^N\int_{\Sigma_\ddd^{k, \varepsilon_k}}R_{\partial\Omega}(u_j(s),z)\partial_{n_\Omega(z)}\mathbf e^{\bm \varepsilon}(z)d\sigma(z).
\end{split}
\ees
This yields, for all $s\in (-1,1)$:
\bes
\begin{split}
    \forall j\in\lb1,N\rb,\ \ 0 &= m_{\partial\Omega}(\mathbf e^{\bm \varepsilon}) - \int_\Omega G_\Omega(u_j(s),y)dy  \\
    &\quad- \frac1\pi\sum_{k=1}^N\int_{-1}^1\ln|u_j(s)-u_k(t)|\Phi_k^{\varepsilon}(t)dt - \sum_{k=1}^N\int_{-1}^1R_{\partial\Omega}(u_j(s),u_k(t))\Phi_k^{\varepsilon}(t)dt.
\end{split}
\ees
 On the other hand, the function 
$V(\cdot) = \int_\Omega G_\Omega(\cdot ,y)\,dy \in H^1(\Omega)$ is the unique weak solution to:
\bes 
        \Delta V(x) = 1 \text{ in }\Omega, \ 
        \partial_{n_\Omega(\zeta)} V(\zeta)  = \frac{|\Omega|}{|\partial\Omega|}  \text{ on }\partial\Omega, \ 
        \int_{\partial\Omega} V(\zeta) \, d\sigma(\zeta) = 0.
\ees
As a result, by the elliptic regularity  theory for Neumann problems,  $V$ is smooth over $\overline \Omega$. We then have 
$V(u_j(s))=V(x^{(j)})+ O(\varepsilon_j)$ uniformly in $s\in (-1,1)$. 
Hence, for all $s\in (-1,1)$:
\be\label{eq:Ko}
\begin{split}
    \forall j\in\lb1,N\rb,\ \ 0 &= m_{\partial\Omega}(\mathbf e^{\bm \varepsilon}) - V(x^{(j)}) + O(\varepsilon_j)\\
    &\quad- \frac1\pi\sum_{k=1}^N\int_{-1}^1\ln|u_j(s)-u_k(t)|\Phi_k^{\varepsilon}(t)dt - \sum_{k=1}^N\int_{-1}^1R_{\partial\Omega}(u_j(s),u_k(t))\Phi_k^{\varepsilon}(t)dt.
\end{split}
\ee
Our objective is now to solve this system for the unknowns $m_{\partial\Omega}(\mathbf e^{\bm \varepsilon})$ and $\Phi_1^{\bm \varepsilon},\ldots,\Phi_N^{\bm \varepsilon}$. Let $q>1$ be as in \eqref{eq:tangential}.  From \eqref{eq:tangential}, let $q\in (1,2)$ such that $\Phi_j^{\varepsilon}\in L^q((-1,1))$ for all $j\in\lb1,N\rb$. Recall that 
 that $\mathscr L_q$ is invertible (see Proposition \ref{prop:Lr}). 
 Therefore, using \eqref{eq:key1} and \eqref{eq:key2}, the system of equations \eqref{eq:Ko} writes, for all $j\in\lb1,N\rb$ and $s\in (-1,1)$:
\bes
\begin{split}
    0 &= m_{\partial\Omega}(\mathbf e^{\bm \varepsilon}) - V(x^{(j)}) + O(\varepsilon_j) - \frac1\pi\big [ \Lr_q + \varepsilon_j\Gr_q^{j}\big ](\Phi_j^{\varepsilon})(s) - \mathbf T^{\bm \varepsilon}_j\ln\varepsilon_j\\
    &\quad- \sum_{k=1,k\neq j}^N\Big[\, \frac{ |\bm \varepsilon|}{\pi}\Hr_q^{j,k}(\Phi_k^{\varepsilon})(s) + \mathbf T^{\bm \varepsilon}_k\ln|x^{(j)} - x^{(k)}|\, \Big]- \sum_{k=1}^N\int_{-1}^1R_{\partial\Omega}(u_j(s),u_k(t))\Phi_k^{\varepsilon}(t)dt.
\end{split}
\ees
Since the functions $R_{\partial\Omega}$, $u_j$, and $u_k$ are smooth, we have uniformly in $s\in (-1,1)$:
\bes
\begin{split}
    \int_{-1}^1R_{\partial\Omega}(u_j(s),u_k(t))\Phi_k^{\bm \varepsilon}(t)dt &= R_{\partial\Omega}(x^{(j)},x^{(k)})\int_{-1}^1 \Phi_k^{\bm \varepsilon}(t)dt + \mathscr C_{k,j}^{\bm \varepsilon}( \Phi_k)(s),
\end{split} 
\ees
where $\mathscr C_{k,j}^{\bm \varepsilon}: L^q((-1,1))\to A^q((-1,1))$ is a bounded linear operator satisfying, in the limit $|\bm \varepsilon|\to 0$,     $\Vert\mathscr C_{k,j}^{\bm \varepsilon} \Vert_{L^p((-1,1))\to A^p((-1,1))}   = O(|\bm \varepsilon|)$. Hence, using also Lemma \ref{le.LeG}, it follows that for all $j\in\lb1,N\rb$ and $s\in (-1,1)$:
\bes
\begin{split}
    &\frac1\pi\big [ \mathscr L_q + \varepsilon_j\mathscr G_q^{j}\big ](\Phi_j^{\bm \varepsilon})(s)  + \sum_{k=1,k\neq j}^N  \frac{ |\bm \varepsilon|}{\pi}\mathscr H_q^{j,k}(\Phi_k^{\bm \varepsilon})(s)  \\
    &\quad\quad= m_{\partial\Omega}(\mathbf e^{\bm \varepsilon}) - V(x^{(j)}) + O(\varepsilon_j) - \mathbf T^{\bm \varepsilon}_j\ln\varepsilon_j\\
    &\quad\quad\quad- \pi\, \mathbf T^{\bm \varepsilon}_j R_{\partial\Omega}(x^{(j)},x^{(j)}) - \pi\sum_{k=1,k\neq j}^N\mathbf T^{\bm \varepsilon}_kG_{\partial\Omega}(x^{(j)},x^{(k)})-  \sum_{k=1}^N \mathscr C_{k,j}^{\bm \varepsilon}(\Phi_k)(s).
\end{split}
\ees
In matrix-operator form, this writes:
\be\label{eq:system1}
\frac1\pi \Big[\mathbf L_q + \diag (\varepsilon_1\mathscr G_q^{1},\ldots,\varepsilon_j\mathscr G_q^{j}) + |\bm \varepsilon|\mathbf H^{\bm \varepsilon}_q + \pi\mathbf C^{\bm \varepsilon}_q \Big] \Phi^{\bm \varepsilon} = m_{\partial\Omega}(\mathbf e^{\bm \varepsilon})\mathbf U - \mathbf V^{\bm \varepsilon} - \mathbf Q_1^{\bm \varepsilon}\mathbf T^{\bm \varepsilon}
\ee
where 
\bes
\mathbf L_q =  \diag(\mathscr L_q ,\ldots, \mathscr L_q),\ (\mathbf H^{\bm \varepsilon}_q)_{j,k} = \indic_{j\neq k} \, \mathscr H_q^{j,k},\ (\mathbf C^{\bm \varepsilon}_q)_{j,k} = \mathscr C^{\bm \varepsilon}_{j,k},
\ees
\bes
\Phi^{\bm \varepsilon} = \begin{pmatrix}\Phi_1^{\bm \varepsilon}\\\vdots\\\Phi_N^{\bm \varepsilon}\end{pmatrix},\ \mathbf T^{\bm \varepsilon} = \begin{pmatrix}\mathbf T^{\bm \varepsilon}_1\\\vdots\\\mathbf T^{\bm \varepsilon}_N\end{pmatrix},\ \mathbf U = \begin{pmatrix}1\\\vdots\\1\end{pmatrix},\ \text{and}\  \mathbf V^{\bm \varepsilon} = \begin{pmatrix}V(x^{(1)}) + O(\varepsilon_1)\\\vdots\\V(x^{(N)}) + O(\varepsilon_N)\end{pmatrix},
\ees
and where $\mathbf Q_1^{\bm \varepsilon} = \mathbf D_{\bm \varepsilon} + \mathbf A - \ln(1/2)\, \mathbf I$, with
\bes
\mathbf D_{\bm \varepsilon}= \diag\big (\ln\frac{\varepsilon_1}{2} ,\ldots, \ln\frac{\varepsilon_N}{2}\big ),
\ees
and $\mathbf A = (A_{j,k})_{1\leq j,k\leq N}$ is the symmetric matrix (see Lemma \eqref{le.sym}) defined by 
\be\label{eq.ajk}
A_{j,k} = \pi G_{\partial\Omega}(x^{(j)},x^{(k)}) \quad \text{if } j \neq k, \quad \text{and} \quad A_{k,k} = \pi R_{\partial\Omega}(x^{(k)},x^{(k)}).
\ee
Denoting $\mathbf J^{\bm \varepsilon}_q = \diag (\varepsilon_1\mathscr G_q^{1},\ldots,\varepsilon_j\mathscr G_q^{j}) + |\bm \varepsilon|\mathbf H^{\bm \varepsilon}_q + \pi\mathbf C^{\bm \varepsilon}_q$, we write system \eqref{eq:system1} as
\bes
\frac1\pi \mathbf L_q\big  [\mathbf I +  \mathbf L_q^{-1}\mathbf J^{\bm \varepsilon}_q\big ] \Phi^{\bm \varepsilon} =m_{\partial\Omega}(\mathbf e^{\bm \varepsilon})\mathbf U - \mathbf V^{\bm \varepsilon} - \mathbf Q_1^{\bm \varepsilon}\mathbf T^{\bm \varepsilon}.
\ees
Here, $\mathbf L_q^{-1} = \diag(\mathscr L_q^{-1},\ldots, \mathscr L_q^{-1})$ is well defined thanks to Proposition \ref{prop:Lr}. Recall that in the limit $|\bm{\varepsilon}|\to 0$:
\bes
\vvert{\mathbf J^{\bm \varepsilon}_q}_{L^q((-1,1))^N\to A^q((-1,1))^N} = O(|\bm\varepsilon|),
\ees
and therefore $\mathbf I +  \mathbf L_q^{-1}\mathbf J^{\bm \varepsilon}_q$ is invertible (using a Neumann series) for $|\bm\varepsilon|\to0$. Let us denote $\mathbf R^{1,\bm \varepsilon} : L^q((-1,1))^N\to L^q((-1,1))^N$ such that $[\mathbf I + \mathbf L_q^{-1}\mathbf J^{\bm \varepsilon}_q]^{-1} = \mathbf I + \mathbf R^{1,\bm \varepsilon}$. Note that
\bes
\vvert{\mathbf R^{1,\bm \varepsilon}}_{L^q((-1,1))^N\to L^q((-1,1))^N} = O(|\bm\varepsilon|),
\ees
Set
\be\label{eq:aVect2}
\mathbf B_j^{\bm \varepsilon}= \frac{1}{\ln(1/2)} \big(m_{\partial\Omega}(\mathbf e^{\bm \varepsilon})\mathbf U - \mathbf V^{\bm \varepsilon} - \mathbf Q_1^{\bm \varepsilon}\mathbf T^{\bm \varepsilon}\big)_j.
\ee
Therefore, applying Proposition \ref{prop:Lr} and \eqref{eq.inv1} yields, for $s\in (-1,1)$:
\be\label{eq.phi1}
\Phi_j^{\bm \varepsilon}(s) = \frac{\mathbf B_j^{\bm \varepsilon}}{\sqrt{1-s^2}} + \sum_{k=1}^N(\mathbf R^{1,\bm \varepsilon})_{j,k}\Big(\frac{1}{\sqrt{1-\cdot ^2}}\Big)(s)\, \mathbf B_k^{\bm \varepsilon}.
\ee
where (recall $1<q<2$ and so $q/2<1$):
\bes
\big\Vert(\mathbf R^{1,\bm \varepsilon})_{j,k}\Big(\frac{1}{\sqrt{1-\cdot ^2}}\Big )\big \Vert_{L^q((-1,1))^N}\le C |\bm \varepsilon| \, \big \Vert \frac{1}{\sqrt{1-\cdot ^2}}\big \Vert _{L^q((-1,1))} \le C|\bm \varepsilon|.
\ees
The goal is thus to study the asymptotic behavior of $\mathbf B_j^{\bm \varepsilon}$ (and thus of the ones of $m_{\partial\Omega}(\mathbf e^{\bm \varepsilon})$ and $\mathbf T^{\bm \varepsilon}$, see \eqref{eq:aVect2}) when $|\bm \varepsilon|\to 0$.

We first focus on $m_{\partial\Omega}(\mathbf e^{\bm \varepsilon})$. 
Integrating \eqref{eq.phi1} over $(-1,1)$ yields the relation $
\mathbf{T}^{\bm{\varepsilon}} = \left( \mathbf{I} + \mathbf{R}^{2,\bm{\varepsilon}} \right) \mathbf{B}^{\bm{\varepsilon}}$, 
where, for $j,k\in\lb1,N\rb$,
\bes
\mathbf{R}_{j,k}^{2,\bm{\varepsilon}} = \frac1\pi\int_{-1}^1 \mathbf{R}^{1,\bm{\varepsilon}}_{j,k} \Big(\frac{1}{\sqrt{1-\cdot^2}}\Big)(s) \, ds  = O(|\bm{\varepsilon}|).
\ees
Here again, we use a Neumann series to invert $\mathbf{I} + \mathbf{R}^{2,\bm{\varepsilon}}$ over $\R^N$, leading to
\be\label{eq:recoverB1}
\mathbf{B}^{\bm{\varepsilon}} = \left( \mathbf{I} + \mathbf{R}^{3,\bm{\varepsilon}} \right) \mathbf{T}^{\bm{\varepsilon}},
\ee
for some $\mathbf{R}^{3,\bm{\varepsilon}}:\R^N\to\R^N$ with $\mathbf{R}^{3,\bm{\varepsilon}} = O(|\bm{\varepsilon}|)$. By  definition of $\mathbf{B}^{\bm{\varepsilon}}$ (see \eqref{eq:aVect2}), we then have:
\bes
\frac{1}{\ln(1/2)} \big(  m_{\partial\Omega}(\mathbf{e}^{\bm{\varepsilon}})\mathbf{U} - \mathbf{V}^{\bm{\varepsilon}} - \mathbf{Q}_1^{\bm{\varepsilon}}\mathbf{T}^{\bm{\varepsilon}} \big) = (\mathbf{I} + \mathbf R^{3,\bm \varepsilon}) \mathbf{T}^{\bm{\varepsilon}}.
\ees
Hence, recalling that $\mathbf{Q}_1^{\bm{\varepsilon}} = \mathbf{D}_{\bm{\varepsilon}} + \mathbf{A} - \ln(1/2)\mathbf{I}$, this leads to
\bes
\begin{split}
    m_{\partial\Omega}(\mathbf{e}^{\bm{\varepsilon}})\mathbf{U} - \mathbf{V}^{\bm{\varepsilon}} &= \Big ( \big[ \mathbf{D}_{\bm{\varepsilon}} + \mathbf{A} - \ln(1/2)\mathbf{I} \big] + \ln(1/2)\mathbf{I} + \ln(1/2) \mathbf R^{3,\bm \varepsilon}\Big ) \mathbf{T}^{\bm{\varepsilon}}\\
    &= \Big ( \mathbf{D}_{\bm{\varepsilon}} + \mathbf{A}   + \ln(1/2) \mathbf R^{3,\bm \varepsilon}\Big ) \mathbf{T}^{\bm{\varepsilon}}.
\end{split}
\ees
Recall that $\mathbf{D}_{\bm{\varepsilon}} = \diag\left(\ln \frac{\varepsilon_1}{2}, \dots, \ln \frac{\varepsilon_N}{2}\right)$. Since $\mathbf{D}_{\bm{\varepsilon}}$ is a diagonal matrix of large negative logs, it is invertible for $|\bm{\varepsilon}|$ small enough. For every small $|\bm{\varepsilon}|>0$, the matrix $\mathbf{I} + \mathbf{D}_{\bm{\varepsilon}}^{-1}\mathbf{A} + \ln(1/2) \mathbf{D}_{\bm{\varepsilon}}^{-1}  \mathbf R^{3,\bm \varepsilon}$ can be inverted via a Neumann series again, since $\mathbf{D}_{\bm{\varepsilon}}^{-1} = -\diag(\mathbf{K}^{\bm{\varepsilon}}) \to 0$. This yields:
\be\label{eq:recoverB2}
\mathbf{T}^{\bm{\varepsilon}} = \left( \mathbf{I} - \mathbf{D}_{\bm{\varepsilon}}^{-1}\mathbf{A} +  \mathbf R^{4,\bm \varepsilon} \right) \mathbf{D}_{\bm{\varepsilon}}^{-1} \left( m_{\partial\Omega}(\mathbf{e}^{\bm{\varepsilon}})\mathbf{U} - \mathbf{V}^{\bm{\varepsilon}} \right),
\ee
where $
\mathbf R^{4,\bm \varepsilon} = O(1/ |\ln|\bm{\varepsilon}||^2) = O(|\mathbf{K}^{\bm{\varepsilon}}|^2)$. 
Since
\begin{equation}\label{eq:matrix_components}
\begin{split}
(\mathbf{D}_{\bm{\varepsilon}}^{-1}\mathbf{U})_j &= -K_j^{\varepsilon_j}, 
\ 
( \mathbf{D}_{\bm{\varepsilon}}^{-1}\mathbf{V}^{\bm{\varepsilon}} )_j = -K_j^{\varepsilon_j} V(x^{(j)}) + O(\varepsilon_j K_j^{\varepsilon_j}), \\
(\mathbf{D}_{\bm{\varepsilon}}^{-1}\mathbf{A} \mathbf{D}_{\bm{\varepsilon}}^{-1} \mathbf{U} )_j &= K_j^{\varepsilon_j} \sum_{k=1}^N A_{j,k} K_k^{\varepsilon_k}, 
\ 
(\mathbf{D}_{\bm{\varepsilon}}^{-1}\mathbf{A} \mathbf{D}_{\bm{\varepsilon}}^{-1} \mathbf{V}^{\bm{\varepsilon}} )_j = K_j^{\varepsilon_j} \sum_{k=1}^N A_{j,k} K_k^{\varepsilon_k} V(x^{(k)}) + O(|\mathbf{K}^{\bm{\varepsilon}}|^3),
\end{split}
\end{equation}
we have:
\bes
\mathbf{T}_j^{\bm{\varepsilon}} = m_{\partial\Omega}(\mathbf{e}^{\bm{\varepsilon}}) \Big[ -K_j^{\varepsilon_j} - K_j^{\varepsilon_j}\sum_{k=1}^N A_{j,k}K_k^{\varepsilon_k} + O(|\mathbf{K}^{\bm{\varepsilon}}|^3) \Big] + V(x^{(j)})K_j^{\varepsilon_j} + O(|\mathbf{K}^{\bm{\varepsilon}}|^2).
\ees
Next, we invoke the flux condition $\mathbf{T}^{\bm{\varepsilon}}\cdot \mathbf{U} = -\frac{|\Omega|}{\pi}$ (see \eqref{eq:sumtracee}) to get in the limit $|\bm{\varepsilon}|\to 0$:
\bes
-\frac{|\Omega|}{\pi} = -m_{\partial\Omega}(\mathbf{e}^{\bm{\varepsilon}}) \Big[\sum_{j=1}^N(K_j^{\varepsilon_j} + K_j^{\varepsilon_j}\sum_{k=1}^N A_{j,k}K_k^{\varepsilon_k}) + O(|\mathbf{K}^{\bm{\varepsilon}}|^3) \Big] + \sum_{j=1}^NV(x^{(j)})K_j^{\varepsilon_j} + O(|\mathbf{K}^{\bm{\varepsilon}}|^2),
\ees
i.e.
\bes
m_{\partial\Omega}(\mathbf{e}^{\bm{\varepsilon}}) = \frac{\frac{|\Omega|}{\pi} + \sum_{j=1}^NV(x^{(j)})K_j^{\varepsilon_j} + O(|\mathbf{K}^{\bm{\varepsilon}}|^2)}{\sum_{j=1}^N(K_j^{\varepsilon_j} + K_j^{\varepsilon_j}\sum_{k=1}^N A_{j,k}K_k^{\varepsilon_k}) + O(|\mathbf{K}^{\bm{\varepsilon}}|^3)}.
\ees
Note that this implies that $m_{\partial\Omega}(\mathbf{e}^{\bm{\varepsilon}}) = O(|\mathbf{K}^{\bm{\varepsilon}}|^{-1})$. Setting $S^{\bm{\varepsilon}} = \sum_{j=1}^NK_j^{\varepsilon_j}$, it holds that:
\bes
\begin{split}
    m_{\partial\Omega}(\mathbf{e}^{\bm{\varepsilon}}) &= \frac{1}{S_{\bm{\varepsilon}}} \left[ \frac{|\Omega|}{\pi} + \sum_{j=1}^NV(x^{(j)})K_j^{\varepsilon_j} + O(|\mathbf{K}^{\bm{\varepsilon}}|^2) \right] \cdot \left[ 1 + \frac{1}{S_{\bm{\varepsilon}}} \sum_{j=1}^N K_j^{\varepsilon_j}\sum_{k=1}^N A_{j,k}K_k^{\varepsilon_k} + O(|\mathbf{K}^{\bm{\varepsilon}}|^3) \right]^{-1}.
\end{split}
\ees
Expanding the fraction we have when $|\bm{\varepsilon}|\to 0$:
\be\label{eq:piece}
m_{\partial\Omega}(\mathbf{e}^{\bm{\varepsilon}}) = \frac{|\Omega|}{\pi \sum_{j=1}^NK_j^{\varepsilon_j}} + \frac{\sum_{j=1}^N V(x^{(j)})K_j^{\varepsilon_j}}{\sum_{j=1}^NK_j^{\varepsilon_j}} - \frac{\frac{|\Omega|}{\pi} \sum_{j=1}^N K_j^{\varepsilon_j}\sum_{k=1}^N A_{j,k}K_k^{\varepsilon_k}}{(\sum_{j=1}^NK_j^{\varepsilon_j})^2} + O(|\mathbf{K}^{\bm{\varepsilon}}|).
\ee

We can now come back to the study of
the asymptotic behavior of $\mathbf B_j^{\bm \varepsilon}$ (see \eqref{eq:aVect2}) when $|\bm \varepsilon|\to 0$.
  We  plug  back \eqref{eq:recoverB2} into \eqref{eq:recoverB1} (recall that $\mathbf{R}^{3,\bm{\varepsilon}} = O(|\bm{\varepsilon}|)$, $\mathbf R^{4,\bm \varepsilon} =   O(|\mathbf{K}^{\bm{\varepsilon}}|^2)$, and  $\mathbf{D}_{\bm{\varepsilon}}^{-1} = O(|\mathbf K^{\bm \varepsilon}|^{-1})$), and we deduce that:
  \bes
\mathbf{B}^{\bm{\varepsilon}} =  \mathbf{T}^{\bm{\varepsilon}}+ \mathbf{R}^{3,\bm{\varepsilon}} \mathbf{T}^{\bm{\varepsilon}} =  \left( \mathbf{I} - \mathbf{D}_{\bm{\varepsilon}}^{-1}\mathbf{A} +  \mathbf R^{4,\bm \varepsilon} \right) \mathbf{D}_{\bm{\varepsilon}}^{-1} \left( m_{\partial\Omega}(\mathbf{e}^{\bm{\varepsilon}})\mathbf{U} - \mathbf{V}^{\bm{\varepsilon}} \right) + O(|\bm{\varepsilon}|^{1-\delta})
\ees
for any $\delta>0$. This yields in the limit $|\bm{\varepsilon}|\to 0$:
\be\label{eq:aVect3}
\mathbf{B}^{\bm{\varepsilon}} = - \frac{|\Omega|}{\pi\sum_{k=1}^NK_k^{\varepsilon_k}}\Big(\mathbf K^{\bm \varepsilon} + O(|\mathbf K^{\bm \varepsilon}|^2)\Big).
\ee
Note in articular  that $\mathbf B^{\varepsilon}$ is bounded (i.e. $\mathbf B^{\varepsilon}=O(1)$). Therefore, by \eqref{eq.phi1}, we have: 
\be\label{eq:PhiVect2}
\forall s\in(-1,1),\ \Phi^{\bm \varepsilon}(s) = \frac{\mathbf B^{\varepsilon}}{\sqrt{1-s^2}} + O_{L^q((-1,1))}( |\bm \varepsilon|).
\ee
We are now in position to derive an asymptotic expansion of  $\mathbf e^{\bm \varepsilon}$. Let $x\in\overline\Omega\setminus\{x^{(1)},\ldots,x^{(N)}\}$, using Corollary \ref{cor:IPP}, and the expression for $G_{\partial\Omega}$ (see \eqref{eq:Greens1}), we have
\bes
\begin{split}
    \mathbf e^{\bm \varepsilon}(x) &= \frac{1}{|\partial\Omega|}\int_{\partial\Omega}\mathbf e^{\bm \varepsilon}(z)d\sigma(z) - V(x) - \int_{\partial\Omega}G_{\partial\Omega}(x,z)\partial_{n_\Omega(z)}\mathbf e^{\bm \varepsilon}(z)d\sigma(z)\\
    &= m_{\partial\Omega}(\mathbf e^{\bm \varepsilon}) - V(x) - \frac1\pi\sum_{j=1}^N\int_{\Sigma_\ddd^{j, \varepsilon_j}}\ln|x-z|\partial_{n_\Omega(z)}\mathbf e^{\bm \varepsilon}(z)d\sigma(z) - \sum_{j=1}^N\int_{\Sigma_\ddd^{j, \varepsilon_j}}R_{\partial\Omega}(x,z)\partial_{n_\Omega(z)}\mathbf e^{\bm \varepsilon}(z)d\sigma(z).
\end{split}
\ees
Then, a change of variable and \eqref{eq:PhiVect2} imply that when $|\bm{\varepsilon}|\to 0$: 
\bes
\begin{split}
    \mathbf e^{\bm \varepsilon}(x) &= m_{\partial\Omega}(\mathbf e^{\bm \varepsilon}) - V(x) - \frac1\pi\sum_{j=1}^N\int_{-1}^1\ln|x-u_j(t)|\Phi_j^{\bm \varepsilon}(t)dt - \sum_{j=1}^N\int_{-1}^1R_{\partial\Omega}(x,u_j(t))\Phi_j^{\bm \varepsilon}(t)dt\\
    &= m_{\partial\Omega}(\mathbf e^{\bm \varepsilon}) - V(x) - \frac1\pi\sum_{j=1}^N\ln|x-x^{(j)}|\int_{-1}^1\frac{\mathbf B^{\bm \varepsilon}_j}{\sqrt{1-t^2}}dt\\
    &\quad- \sum_{j=1}^NR_{\partial\Omega}(x,x^{(j)})\int_{-1}^1\frac{\mathbf B^{\bm \varepsilon}_j}{\sqrt{1-t^2}}dt + r^{\bm \varepsilon}(x)  + O( |\bm \varepsilon|)\\
    &= m_{\partial\Omega}(\mathbf e^{\bm \varepsilon}) - V(x) - \pi\sum_{j=1}^N\mathbf B^{\bm \varepsilon}_j\, G_{\partial\Omega}(x,x^{(j)}) + r^{\bm \varepsilon}(x) + O( |\bm \varepsilon|),
\end{split}
\ees
where
$$r^{\bm{\varepsilon}}(x) = -\frac{1}{\pi} \sum_{j=1}^N \int_{-1}^1 \Big[\ln\left( \frac{|x - u_j(t)|}{|x - x^{(j)}|} \right) \Phi_j^{\bm{\varepsilon}}(t) \, dt -  \int_{-1}^1 \left[ R_{\partial\Omega}(x, u_j(t)) - R_{\partial\Omega}(x, x^{(j)}) \right] \Phi_j^{\bm{\varepsilon}}(t) \,  dt\Big].$$
Since
 $\Phi_j^{\bm{\varepsilon}}(t) = \frac{\mathbf B_j^{\bm{\varepsilon}}}{\sqrt{1-t^2}} + \mathcal{R}_j^{\bm{\varepsilon}}(t)$ where  $\mathcal{R}_j^{\bm{\varepsilon}}(t) = O_{L^q((-1,1))}(|\bm{\varepsilon}|)$, we have $r^{\bm{\varepsilon}}(x) = r_0^{\bm{\varepsilon}}(x) + \rho^{\bm{\varepsilon}}(x)$
 where
 $$r_0^{\bm{\varepsilon}}(x) = -\frac{1}{\pi} \sum_{j=1}^N \mathbf B_j^{\bm{\varepsilon}} \Big[\int_{-1}^1 \ln\left( \frac{|x - u_j(t)|}{|x - x^{(j)}|} \right) \frac{1}{\sqrt{1-t^2}} \, dt -  \int_{-1}^1 \frac{R_{\partial\Omega}(x, u_j(t)) - R_{\partial\Omega}(x, x^{(j)})}{\sqrt{1-t^2}} \, dt \Big].$$
and 
$$\rho^{\bm{\varepsilon}}(x) = -\frac{1}{\pi} \sum_{j=1}^N\Big[ \int_{-1}^1 \ln\left( \frac{|x - u_j(t)|}{|x - x^{(j)}|} \right) \mathcal{R}_j^{\bm{\varepsilon}}(t) \, dt -  \int_{-1}^1 \left[ R_{\partial\Omega}(x, u_j(t)) - R_{\partial\Omega}(x, x^{(j)}) \right] \mathcal{R}_j^{\bm{\varepsilon}}(t) \, dt\Big].$$
Since $\mathbf B^{\varepsilon}=O(1)$ and $u_j(t) - x^{(j)} = O(\varepsilon_j)$  (uniformly in $t\in (-1,1)$),  we have for any $\delta>0$ (recall \eqref{eq.Oe}), 
$$r_0^{\bm{\varepsilon}}(x) = \sum_{j=1}^N \frac{O(\varepsilon_j)}{|x - x^{(j)}|} = O_{L^\infty(\Omega_{\bm{\varepsilon}})}(|\bm{\varepsilon}|^{1-\delta})\text{ and } \rho^{\bm{\varepsilon}}(x) = O_{L^\infty(\Omega_{\bm{\varepsilon}})}(|\bm{\varepsilon}|^{2-\delta}).$$ 
This leads to
\be\label{eq:eapproxinteressant}
\mathbf e^{\bm \varepsilon}(x) = m_{\partial\Omega}(\mathbf e^{\bm \varepsilon}) - V(x) - \pi\sum_{j=1}^N\mathbf B^{\bm \varepsilon}_j\, G_{\partial\Omega}(x,x^{(j)}) + O_{L^\infty(\Omega_{\bm \varepsilon})}( |\bm \varepsilon|^{1-\delta}).
\ee
Substituting back \eqref{eq:piece} and \eqref{eq:aVect3} into \eqref{eq:eapproxinteressant}, we then obtain for all $\delta>0$,
\bes
\mathbf e^{\bm \varepsilon}(x) = \frac{|\Omega|}{\pi\sum_{k=1}^NK_k^{\varepsilon_k}}\Big[1 + \sum_{k=1}^NK_k^{\varepsilon_k}\big (\ln|x-x^{(k)}| + h_k(x)\big )\Big] + O_{L^\infty(\Omega_{\bm \varepsilon})}(|\mathbf K^{\bm \varepsilon}|^{1-\delta}),
\ees
where for all $k\in\lb1,N\rb$, $h_k$ is given by (recall that $V(x^{k})-V(x)=\int_\Omega \big (G_\Omega(x^{(k)},y)-G_\Omega(x,y)\big )dy$):
\bes
h_k(x) = \frac{\pi}{|\Omega|}\int_\Omega \big (G_\Omega(x^{(k)},y)-G_\Omega(x,y)\big )dy + \pi R_{\partial\Omega}(x,x^{(k)}) - \frac{1}{\sum_{j=1}^NK_j^{\varepsilon_j}}\sum_{j=1}^NK_j^{\varepsilon_j} A_{j,k}.
\ees
This exactly the formula given in Theorems \ref{thm:1}  if we show that  for all $x\in\Omega_{\bm \varepsilon}$, 
\be\label{eq.cp}
\sum_{k=1}^NK_k^{\varepsilon_k}h_k(x) = \sum_{k=1}^NK_k^{\varepsilon_k}g_k^{\bm \varepsilon}(x),
\ee
 where $g_k^{\bm \varepsilon}$ is defined in \eqref{eq:espThomas}. First,   thanks to \eqref{eq.RpaOmega} and \eqref{eq:Greens1}, $h_k$ solves the system
\bes
\left\{
\begin{aligned}
    \Delta h_k^{\bm \varepsilon} &=  -\frac{\pi}{|\Omega|}
    &&\text{in }\Omega,\\
    \partial_{n(\zeta)} h_k^{\bm \varepsilon}(\zeta) &= -\partial_{n(\zeta)}\, \ln|\zeta-x^{(k)}| &&\text{on }\partial\Omega.
\end{aligned}
\right.
\ees
Consequently, by \eqref{eq:Rok}, $h_k^{\bm \varepsilon}  = R_k^0 + h_k^{\bm \varepsilon}(x^{(k)})$. We then deduce that for every $x\in\overline\Omega$:
\bes
\begin{split}
    h_k^{\bm \varepsilon}(x) &= R_k^0(x) + \pi R_{\partial\Omega}(x^{(k)},x^{(k)}) - \frac{1}{\sum_{j=1}^NK_j^{\varepsilon_j}}\sum_{j=1}^NK_j^{\varepsilon_j} A_{j,k}\\
    &=   R_k^0(x) + \pi R_{\partial\Omega}(x^{(k)},x^{(k)}) - \frac{K_k^{\varepsilon_k} \pi R_{\partial\Omega}(x^{(k)},x^{(k)}) + \sum_{j=1,j \neq k}^N K_j^{\varepsilon_j} \pi G_{\partial\Omega}(x^{(j)},x^{(k)})}{\sum_{j=1}^NK_j^{\varepsilon_j}} \\
    &=  R_k^0(x) + \frac{(\sum_{j=1,j \neq k}^N K_j^{\varepsilon_j}) \pi R_{\partial\Omega}(x^{(k)},x^{(k)}) - \sum_{j=1,j \neq k}^N K_j^{\varepsilon_j} \pi G_{\partial\Omega}(x^{(j)},x^{(k)})}{\sum_{j=1}^NK_j^{\varepsilon_j}}\\
    &=   R_k^0(x) - \frac{1}{\sum_{j=1}^N K_j^{\varepsilon_j}} \sum_{j=1,j \neq k}^N K_j^{\varepsilon_j} \left[ \pi G_{\partial\Omega}(x^{(j)},x^{(k)}) - \pi R_{\partial\Omega}(x^{(k)},x^{(k)}) \right].
\end{split}
\ees  
Recall \eqref{eq:espThomas} writes (see \eqref{eq:fk0} and  \eqref{eq:Rok}),
$g_k^{\bm \varepsilon}(x) = R_k^0(x) - \big(\sum_{j=1,j\neq k}^N K_j^{\varepsilon_j}  f_j^0(x^{(k)}) \big) / \sum_{j=1}^NK_j^{\varepsilon_j}$. 
Thus,   the function $g_k^{\bm \varepsilon} - h_k^{\bm \varepsilon}$ (which is constant over $\overline \Omega$) is equal to 
\bes\frac{1}{\sum_{j=1}^N K_j^{\varepsilon_j}} \sum_{j=1,j \neq k}^N K_j^{\varepsilon_j} \Big[\underbrace{\pi G_{\partial\Omega}(x^{(j)},x^{(k)}) - \pi R_{\partial\Omega}(x^{(k)},x^{(k)}) - f_j^0(x^{(k)})}_{  \pi R_{\partial\Omega}(x^{(j)},x^{(k)}) - \pi R_{\partial\Omega}(x^{(k)},x^{(k)}) - R_j^0(x^{(k)})} \Big].
\ees
Therefore, the equality \eqref{eq.cp} is equivalent to
\bes
\sum_{j,k=1, j\neq k}^NK_k^{\varepsilon_k} K_j^{\varepsilon_j} \Big[R_j^0(x^{(k)})  - \pi R_{\partial\Omega}(x^{(j)},x^{(k)}) + \pi R_{\partial\Omega}(x^{(k)},x^{(k)})\Big]=0,
\ees
that we rewrite  
\be\label{eq.eg-c0}
0=\sum_{j,k=1, j\neq k}^N K_k^{\varepsilon_k} K_j^{\varepsilon_j} m_{j,k} = \sum_{j<k}^N K_k^{\varepsilon_k} K_j^{\varepsilon_j}  ( m_{j,k} + m_{k,j}  ),
\ee
 with $m_{j,k} = R_j^0(x^{(k)}) - \pi R_{\partial\Omega}(x^{(j)},x^{(k)}) + \pi R_{\partial\Omega}(x^{(k)},x^{(k)})$. 
The identity \eqref{eq.eg-c0} is satisfied if we prove that for all $1\le j<k\le N$, $m_{j,k} + m_{k,j}=0$ i.e.   (recall $R_{\partial\Omega}$ is symmetric over $\partial \Omega \times \partial \Omega$):
\be\label{eq.eg-c}
  R_j^0(x^{(k)}) + R_k^0(x^{(j)})  -2 \pi R_{\partial\Omega}(x^{(j)},x^{(k)}) + \pi R_{\partial\Omega}(x^{(k)},x^{(k)}) + \pi R_{\partial\Omega}(x^{(j)},x^{(j)})=0. 
\ee
To prove the claim \eqref{eq.eg-c}, we introduce   for 
$i,j,k\in\lb1,N\rb$,  the following  function $h_{i,j,k}$ defined over  $\overline\Omega$ by:
\bes
\begin{split}
    h_{i,j,k}(x) &= R_i^0(x^{(k)}) - R_j^0(x^{(k)}) + R_j^0(x) - R_i^0(x)\\
    &\quad- \pi R_{\partial\Omega}(x^{(i)},x^{(k)}) + \pi R_{\partial\Omega}(x^{(j)},x^{(k)}) - \pi R_{\partial\Omega}(x,x^{(j)}) + \pi R_{\partial\Omega}(x,x^{(i)}).
\end{split}
\ees  
We then observe that  $h_{k,j,k}(x) + h_{j,k,j}(x) = - R_j^0(x^{(k)}) - R_k^0(x^{(j)})  + 2\pi R_{\partial\Omega}(x^{(j)},x^{(k)}) - \pi R_{\partial\Omega}(x^{(k)},x^{(k)}) - \pi R_{\partial\Omega}(x^{(j)},x^{(j)})$, and thus that the identity \eqref{eq.eg-c} holds if  
\bes
h_{k,j,k}(x) + h_{j,k,j}(x) = 0\quad \text{ for some   }x\in \overline \Omega.
\ees
We now prove that for all $i,j,k\in\lb1,N\rb$, the function $h_{i,j,k}$ vanishes over $\overline \Omega$. 
Using \eqref{eq:Rok} and \eqref{eq:RGreens}, we have
\bes
\forall x\in \Omega, \ \Delta h_{i,j,k}(x) = 0,\ \text{ and } \forall z\in \partial \Omega, \  \partial_{n(z)}h_{i,j,k}(z) = 0.
\ees
Thus, the function $h_{i,j,k}$ is constant over $\overline\Omega$. Since $h_{i,j,k}(x^{(k)}) = 0$ it follows that  $h_{i,j,k}\equiv0$ on $\overline\Omega$.   
This achieves the proofs of \eqref{eq.eg-c} and then of \eqref{eq.cp}. 
 The proof of Theorem \ref{thm:1} is complete.


\subsubsection{Study of the escape probability distribution} Let $i\in\lb1,N\rb$ be fixed throughout this section (the index of the escape window through which the escape probability is evaluated). Recall that $\mathbf p_i^{\bm \varepsilon}$ is the solution to \eqref{eq:E1}. 
Then, by Corollary \ref{cor:IPP}, we have for all $x\in \overline \Omega$, 
$$\mathbf p_i^{\bm \varepsilon}(x) = m_{\partial\Omega}(\mathbf p^{\bm \varepsilon}_i)   - \sum_{j=1}^N\int_{\Sigma_\ddd^{j, \varepsilon_j}}  G_{\partial\Omega}(x,z)\partial_{n_\Omega(z)}\mathbf p_i^{\bm \varepsilon}(z)d\sigma(z),$$
where we recall that $ m_{\partial\Omega}(\mathbf p^{\bm \varepsilon}_i) =\frac{1}{|\partial\Omega|}\int_{\partial\Omega}\mathbf p_i^{\bm \varepsilon} $.  Recall also that 
$$\int_{\partial\Omega}\partial_{n_\Omega(z)}\mathbf{p}_i^{\bm{\varepsilon}}(z)d\sigma(z) = \int_\Omega\Delta\mathbf{p}_i^{\bm{\varepsilon}}(x)dx = 0.$$
We introduce for $t\in[-1,1]$, $
\hat\Phi^{\bm \varepsilon}_j(t) = \varepsilon_j'\partial_{n_\Omega(u_j(t))}\mathbf p_i^{\bm \varepsilon}(u_j(t))$ and $\hat{\mathbf T}^{\bm{\varepsilon}}_j =  \frac1\pi\int_{-1}^1\hat \Phi_j^{\bm \varepsilon}$. 
We also define the vectors 
$ \hat \Phi^{\bm \varepsilon} = (\hat \Phi_j^{\bm \varepsilon})_{j=1,\ldots,N}$ and $\hat{\mathbf T} = (\hat{\mathbf T}^{\bm \varepsilon}_j)_{j=1,\ldots,N}$.   
With theses notations $\sum_{j=1}^N \hat{\mathbf{T}}^{\bm{\varepsilon}}_j = 0$. 
In addition using \eqref{eq:Greens2}, for all $x\in \overline \Omega$,
\be\label{eq.pii}
\begin{split}
\mathbf{p}_i^{\bm{\varepsilon}}(x) &=  m_{\partial\Omega}(\mathbf p^{\bm \varepsilon}_i) - \frac1\pi\sum_{j=1}^N\int_{-1}^1\ln|x-u_j(t)|\hat\Phi_j^{\bm \varepsilon}(t)dt - \sum_{j=1}^N\int_{-1}^1R_{\partial\Omega}(x,u_j(t))\hat\Phi_j^{\bm \varepsilon}(t)dt.
\end{split}
\ee
We now use the boundary conditions in \eqref{eq:E1} to derive a system to be inverted.  Let $s \in (-1,1)$.   We have
 $\mathbf{p}_i^{\bm{\varepsilon}}(u_j(s)) = \delta_{i,j}$ and thus, 
 $$ \delta_{i,j} = m_{\partial\Omega}(\mathbf{p}^{\bm{\varepsilon}}_i) - \frac1\pi\sum_{k=1}^N\int_{-1}^1\ln|u_j(s)-u_k(t)|\hat{\Phi}_k^{\bm{\varepsilon}}(t)dt - \sum_{k=1}^N\int_{-1}^1R_{\partial\Omega}(u_j(s),u_k(t))\hat{\Phi}_k^{\bm{\varepsilon}}(t)dt.$$ 
Let $q\in (1,2)$ such that (see \eqref{eq:tangential})  that $\hat \Phi_j^{\varepsilon}\in L^q((-1,1))$ for all $j\in\lb1,N\rb$. Using \eqref{eq:key1} and \eqref{eq:key2}, the previous  system  of equations  writes, for all $j\in\lb1,N\rb$ and $s\in (-1,1)$:
$$\begin{aligned}
 0&=- \frac1\pi\big [ \mathscr L_q + \varepsilon_j\mathscr G_q^{j}\big ](\hat{\Phi}_j^{\bm{\varepsilon}})(s) - \sum_{k=1,k\neq j}^N \frac{ |\bm{\varepsilon}|}{\pi}\mathscr H_q^{j,k}(\hat{\Phi}_k^{\bm{\varepsilon}})(s)  
+  m_{\partial\Omega}(\mathbf{p}^{\bm{\varepsilon}}_i) - \delta_{i,j} - \hat{\mathbf{T}}^{\bm{\varepsilon}}_j\ln\varepsilon_j  \\
&\quad - \pi\, \hat{\mathbf{T}}^{\bm{\varepsilon}}_j R_{\partial\Omega}(x^{(j)},x^{(j)})   - \pi\sum_{k=1,k\neq j}^N\hat{\mathbf{T}}^{\bm{\varepsilon}}_kG_{\partial\Omega}(x^{(j)},x^{(k)})   + \sum_{k=1}^N \mathscr C_{k,j}^{\bm{\varepsilon}}(\hat{\Phi}_k)(s).
\end{aligned}$$
In matrix-operator formulation, this writes:
 $$\frac1\pi \Big[\mathbf L_q + \diag (\varepsilon_1\mathscr G_q^{1},\dots,\varepsilon_N\mathscr G_q^{N}) + |\bm{\varepsilon}|\mathbf H^{\bm{\varepsilon}}_q + \pi\mathbf C^{\bm{\varepsilon}}_q \Big] \hat{\Phi}^{\bm{\varepsilon}} = m_{\partial\Omega}(\mathbf{p}^{\bm{\varepsilon}}_i)\mathbf U - \mathbf{E}^{(i)} - \mathbf Q_1^{\bm{\varepsilon}}\hat{\mathbf{T}}^{\bm{\varepsilon}},$$
  where $\mathbf L_q$, $\mathbf{H}^{\bm{\varepsilon}}_q$, $\mathbf{C}^{\bm{\varepsilon}}_q$, $\mathbf{Q}_1^{\bm{\varepsilon}}$, and $\mathbf{U}$ maintain their  definitions, and $\mathbf{E}^{(i)}_j= \delta_{i,j}$.  
  
  Our next goal is to invert the previous system to deduce the asymptotic behavior of $m_{\partial\Omega}(\mathbf{p}^{\bm{\varepsilon}}_i)$ and $\hat{\mathbf{T}}^{\bm{\varepsilon}}$, which is required to establish the asymptotic expansion of $\mathbf{p}_i^{\bm{\varepsilon}}$ using the extra condition that:
  $$\<\hat{\mathbf{T}}^{\bm \varepsilon},\mathbf U\>   = 0.$$
  Note that this problem is indeed very similar to the system \eqref{eq:system1} we solved in  Section \ref{eq.Sec:thm1-} if, roughly speaking, one replaces the vector $\mathbf{V}^{\bm{\varepsilon}}$ by the standard basis vector $\mathbf{E}^{(i)}$, and accounts for the fact that the  constraint   is now $0$ and not $- {|\Omega|}/{\pi}$. Therefore, reasoning as we did in the next section when analyzing $\Phi^{\bm{\varepsilon}}$, we  obtain:
\be\label{eq:PsiVect}
\hat\Phi^{\bm \varepsilon}_j(s) = \frac{\hat{\mathbf B}^{\bm \varepsilon}_j}{\sqrt{1-s^2}} + \sum_{k=1}^N( {\mathbf R}^{1,\bm \varepsilon})_{j,k}\Big(\frac{1}{\sqrt{1-\cdot ^2}}\Big)(s)\hat{\mathbf B}^{\bm \varepsilon}_k,
\ee
where  
\be\label{eq:bhatj}
\hat{\mathbf B}^{\bm \varepsilon}_j = \frac{1}{\ln(1/2)} \big(m_{\partial\Omega}(\mathbf p^{\bm \varepsilon}_i)\mathbf U - \mathbf E^{(i)}- \mathbf Q_1^{\bm \varepsilon}\hat{\mathbf T}^{\bm \varepsilon}\big)_j.
\ee
Following the computations leading to  \eqref{eq:recoverB2}, we deduce that
$$\hat{\mathbf{T}}^{\bm{\varepsilon}} = \left( \mathbf{I} - \mathbf{D}_{\bm{\varepsilon}}^{-1}\mathbf{A} + \mathbf{R}^{4,\bm{\varepsilon}} \right)\mathbf{D}_{\bm{\varepsilon}}^{-1}\left( m_{\partial\Omega}(\mathbf{p}^{\bm{\varepsilon}}_i)\mathbf{U} - \mathbf{E}^{(i)} \right).$$
from which we get  using also \eqref{eq:matrix_components} (changing there ${\mathbf{V}}^{\bm{\varepsilon}}$   into $\mathbf{E}^{(i)}$), as $|\bm{\varepsilon}|\to 0$:
\be\label{eq:TT}
\hat{\mathbf{T}}_j^{\bm{\varepsilon}} = -m_{\partial\Omega}(\mathbf{p}^{\bm{\varepsilon}}_i)K_j^{\varepsilon_j} + \delta_{i,j}K_i^{\varepsilon_i} - K_j^{\varepsilon_j}\sum_{k=1}^N A_{j,k}\left( -m_{\partial\Omega}(\mathbf{p}^{\bm{\varepsilon}}_i)K_k^{\varepsilon_k} + \delta_{i,k}K_i^{\varepsilon_i} \right) + O(|\mathbf{K}^{\bm{\varepsilon}}|^3).
\ee
Using $\<\hat{\mathbf{T}}^{\bm \varepsilon},\mathbf U\>   = 0$, it holds that:
\be\label{eq:piece2}
m_{\partial\Omega}(\mathbf{p}^{\bm{\varepsilon}}_i) = \frac{K_i^{\varepsilon_i}\big[1 + \sum_{k=1}^N A_{i,k}K_k^{\varepsilon_k} + O(|\mathbf{K}^{\bm{\varepsilon}}|^2)\big]}{\sum_{j=1}^NK_j^{\varepsilon_j} + \sum_{j,k=1}^N K_j^{\varepsilon_j}K_k^{\varepsilon_k}A_{j,k} + O(|\mathbf{K}^{\bm{\varepsilon}}|^3)}.
\ee
Hence (recall that $S^{\bm{\varepsilon}} = \sum_{j=1}^NK_j^{\varepsilon_j}$), we have as $|\bm{\varepsilon}|\to 0$:
\be\label{eq:piece-o}
\begin{split}
m_{\partial\Omega}(\mathbf{p}^{\bm{\varepsilon}}_i) &= \frac{K_i^{\varepsilon_i}}{S_{\bm{\varepsilon}}} \left[ 1 + \sum_{k=1}^N A_{i,k}K_k^{\varepsilon_k} - \frac{1}{S_{\bm{\varepsilon}}}\sum_{j,k=1}^N K_j^{\varepsilon_j}K_k^{\varepsilon_k}A_{j,k} + O(|\mathbf{K}^{\bm{\varepsilon}}|^2) \right]\\
&= \frac{K_i^{\varepsilon_i}}{S_{\bm{\varepsilon}}} + \frac{K_i^{\varepsilon_i}\sum_{k=1}^N A_{i,k}K_k^{\varepsilon_k}}{S_{\bm{\varepsilon}}} - \frac{K_i^{\varepsilon_i}\sum_{j,k=1}^N K_j^{\varepsilon_j}K_k^{\varepsilon_k}A_{j,k}}{S_{\bm{\varepsilon}}^2} + O(|\mathbf{K}^{\bm{\varepsilon}}|^2).
\end{split}
\ee
Note that $m_{\partial\Omega}(\mathbf{p}^{\bm{\varepsilon}}_i) =  {K_i^{\varepsilon_i}}/{\sum_{m=1}^N K_m^{\varepsilon_m}} + O(|\mathbf{K}^{\bm{\varepsilon}}|)$, and thus $m_{\partial\Omega}(\mathbf{p}^{\bm{\varepsilon}}_i)=O(1)$ as $|\bm{\varepsilon}|\to 0$. 
We now come back to  \eqref{eq:TT} to get  that 
$$\begin{aligned}
\hat{\mathbf{T}}_j^{\bm{\varepsilon}} &= \delta_{i,j}K_i^{\varepsilon_i} - \frac{K_i^{\varepsilon_i}K_j^{\varepsilon_j}}{S_{\bm{\varepsilon}}} - K_j^{\varepsilon_j} A_{j,i} K_i^{\varepsilon_i}   - \frac{K_i^{\varepsilon_i}K_j^{\varepsilon_j}\sum_{k=1}^N A_{i,k}K_k^{\varepsilon_k}}{S_{\bm{\varepsilon}}} + \frac{K_i^{\varepsilon_i}K_j^{\varepsilon_j}\sum_{k=1}^N A_{j,k}K_k^{\varepsilon_k}}{S_{\bm{\varepsilon}}} \\
&\quad + \frac{K_i^{\varepsilon_i}K_j^{\varepsilon_j}\sum_{m,k=1}^N K_m^{\varepsilon_m}K_k^{\varepsilon_k}A_{m,k}}{S_{\bm{\varepsilon}}^2} + O(|\mathbf{K}^{\bm{\varepsilon}}|^3).
\end{aligned}$$ 
Since  $\hat{\mathbf{B}}^{\bm{\varepsilon}} = \left( \mathbf{I} + \mathbf{R}^{3,\bm{\varepsilon}} \right) \hat{\mathbf{T}}^{\bm{\varepsilon}}=   \hat{\mathbf{T}}_j^{\bm{\varepsilon}} + O(|\mathbf{K}^{\bm{\varepsilon}}|^\infty)$,   it follows that:
\be\label{eq:piece5}
\hat{\mathbf B}^{\bm \varepsilon}_j = K_i^{\varepsilon_i}\delta_{i,j} - \frac{K_i^{\varepsilon_i}K_j^{\varepsilon_j}}{\sum_{k=1}^NK_k^{\varepsilon_k}} {+ O(|\mathbf{K}^{\bm{\varepsilon}}|^2)}.
\ee 
Using \eqref{eq.pii} and the analysis  leading to \eqref{eq:eapproxinteressant}, we have in the limit $|\bm{\varepsilon}|\to 0$, for    every $\delta>0$ and for all $x\in\overline\Omega\setminus\{x^{(1)},\ldots,x^{(N)}\}$:
\bes
\mathbf p_i^{\bm \varepsilon}(x) = m_{\partial\Omega}(\mathbf{p}^{\bm{\varepsilon}}_i) -  \sum_{j=1}^N\hat{\mathbf B}^{\varepsilon}_j\, \ln|x-x^{(j)}| - \pi\sum_{j=1}^N\hat{\mathbf B}^{\varepsilon}_j\, R_{\partial\Omega}(x,x^{(j)}) + O_{L^\infty(\Omega_{\bm \varepsilon})}( |\bm \varepsilon|^{1-\delta}),
\ees
Using the definition of $A$ (see \eqref{eq.ajk}) and substituting back \eqref{eq:piece-o} and \eqref{eq:piece5}   in the previous equality yields \eqref{eq:thmP} and \eqref{eq:probaBoris}. Indeed, {we have  for all $x\in  
\Omega_{\bm \varepsilon} $ (see \eqref{eq.Oe})}, 
\bes
\begin{split}
    \sum_{j=1}^N\hat{\mathbf B}^{\varepsilon}_j\ln|x-x^{(j)}| &= K_i^{\varepsilon_i}\ln|x-x^{(i)}| - \frac{K_i^{\varepsilon_i}}{\sum_{k=1}^NK_k^{\varepsilon_k}}\sum_{j=1,j\neq i}^NK_j^{\varepsilon_j}\ln|x-x^{(j)}|  {+ O(|\mathbf{K}^{\bm{\varepsilon}}|^{2-\delta})}\\
    &= \frac{K_i^{\varepsilon_i}}{\sum_{k=1}^NK_k^{\varepsilon_k}}\Big[\sum_{j=1}^NK_j^{\varepsilon_j}\ln|x-x^{(i)}| - \sum_{j=1,j\neq i}^NK_j^{\varepsilon_j}\ln|x-x^{(j)}|\Big] {+ O(|\mathbf{K}^{\bm{\varepsilon}}|^{2-\delta})}\\
    &= \frac{K_i^{\varepsilon_i}}{\sum_{k=1}^NK_k^{\varepsilon_k}}\sum_{j=1}^NK_j^{\varepsilon_j}\ln\frac{|x-x^{(i)}|}{|x-x^{(j)}|} {+ O(|\mathbf{K}^{\bm{\varepsilon}}|^{2-\delta})}.
\end{split}
\ees
{Moreover, since $R_{\partial\Omega}$ is bounded over $\overline \Omega$,  we have  for all $x\in  
\Omega_{\bm \varepsilon} $}:
\bes
\pi\sum_{j=1}^N\hat{\mathbf B}^{\varepsilon}_jR_{\partial\Omega}(x,x^{(j)}) = \frac{K_i^{\varepsilon_i}}{\sum_{k=1}^NK_k^{\varepsilon_k}}\sum_{j=1}^NK_j^{\varepsilon_j}\big[\pi R_{\partial\Omega}(x,x^{(i)}) - \pi R_{\partial\Omega}(x,x^{(j)})\big] {+ O(|\mathbf{K}^{\bm{\varepsilon}}|^{2})}.
\ees
 The proof of  Theorem \ref{thm:2} is complete.

\section{Proofs of Corollaries \ref{cor:3}, \ref{cor:1}, and \ref{cor:2}}\label{sec:compat}
\label{sec:cor}

In this section, we  proceed to prove Corollaries \ref{cor:3}, \ref{cor:1}, and \ref{cor:2}.

\subsection{Proof of Corollary \ref{cor:3}}
 \label{sec.Cor3}
 
By \eqref{eq:thmP},  for $x\in \overline{\Omega}_{\bm{\varepsilon}}$, 
 $\P_x\big[X_{\tau_{\overline\Sigma_\ddd^{\bm \varepsilon}}}\in\overline\Sigma_\ddd^{i, \varepsilon_i}\big]= \mathbf a^{\bm{\varepsilon}}_i + b^{\bm{\varepsilon}}_i(x) + c^{\bm{\varepsilon}}_i(x) + O(|\mathbf{K}^{\bm{\varepsilon}}|^{2-\delta})$ 
where:  
\bes
    \begin{split} 
       \mathbf a^{\bm{\varepsilon}}_i&= \frac{K^{\varepsilon_i}_i}{\sum_{k=1}^{N}K_k^{\varepsilon_k}} + \frac{K^{\varepsilon_i}_i}{(\sum_{k=1}^{N}K_k^{\varepsilon_k})^2}\sum_{k=1}^N\sum_{j\notin\{i,k\}}K_k^{\varepsilon_k}K_j^{\varepsilon_j}\ln\frac{|x^{(k)}-x^{(i)}|}{|x^{(k)}-x^{(j)}|}\\
        &\quad+\frac{K^{\varepsilon_i}_i}{(\sum_{k=1}^{N}K_k^{\varepsilon_k})^2}\sum_{k=1}^NK_k^{\varepsilon_k}(K_k^{\varepsilon_k}-K_i^{\varepsilon_i})\ln|x^{(i)}-x^{(k)}| \\
&= \frac{K^{\varepsilon_i}_i}{\sum_{k=1}^{N}K_k^{\varepsilon_k}} \Bigg[ 1 + \frac{1}{\sum_{l=1}^{N}K_l^{\varepsilon_l}}\sum_{k=1}^N\sum_{j\notin\{i,k\}}K_k^{\varepsilon_k}K_j^{\varepsilon_j}\ln\frac{|x^{(k)}-x^{(i)}|}{|x^{(k)}-x^{(j)}|} \\
&\quad + \frac{1}{\sum_{l=1}^{N}K_l^{\varepsilon_l}}\sum_{k=1}^NK_k^{\varepsilon_k}(K_k^{\varepsilon_k}-K_i^{\varepsilon_i})\ln|x^{(i)}-x^{(k)}| \Bigg] =   \frac{K_i^{\varepsilon_i}}{\sum_{k=1}^{N}K_k^{\varepsilon_k}} [1+ O(|\mathbf{K}^{\bm{\varepsilon}}|)],
                \end{split}  
            \ees
            and 
          $$b^{\bm{\varepsilon}}_i(x) = \frac{K^{\varepsilon_i}_i}{\sum_{k=1}^{N}K_k^{\varepsilon_k}}\sum_{k=1}^NK_k^{\varepsilon_k}\ln\frac{|x-x^{(k)}|}{|x-x^{(i)}|}, \   c^{\bm{\varepsilon}}_i(x) = \frac{K^{\varepsilon_i}_i}{(\sum_{k=1}^{N} K_k^{\varepsilon_k})^2}\sum_{j,k=1}^N K_j^{\varepsilon_j}K_k^{\varepsilon_k}g_{i,j,k}(x).$$
          Note that from   \eqref{eq:assH}, it holds that, as $|\bm{\varepsilon}| \to 0$,  
          $$\sum_{k=1}^{N}K_k^{\varepsilon_k} \ln\frac{|x-x^{(k)}|}{|x-x^{(i)}|}= \sum_{k=1}^{N}O\big (\mu_{k,i}^{\bm{\varepsilon}}\big )\to 0, \text{  uniformly on $\overline{\Omega}_{\bm{\varepsilon}}$.} $$ 
  We also  have  when $|\bm{\varepsilon}| \to 0$:
          $$\frac{1}{ \mathbf a^{\bm{\varepsilon}}_i }= \frac{\sum_{k=1}^N K_k^{\varepsilon_k}}{K_i^{\varepsilon_i}}  [1+ O(|\mathbf{K}^{\bm{\varepsilon}}|)].$$
               Then, uniformly in $x\in   \overline{\Omega}_{\bm{\varepsilon}}$, we have:
 \bes
  \begin{split}
          \frac{b^{\bm{\varepsilon}}_i(x)}{ \mathbf a^{\bm{\varepsilon}}_i } = \sum_{k=1}^N O\big (\mu_{k,i}^{\bm{\varepsilon}}\big )[1+ O(|\mathbf{K}^{\bm{\varepsilon}}|)]= \sum_{k=1}^N O\big (\mu_{k,i}^{\bm{\varepsilon}}\big )  \ \text{ and } \     \frac{c^{\bm{\varepsilon}}_i(x)}{ \mathbf a^{\bm{\varepsilon}}_i } =O(|\mathbf{K}^{\bm{\varepsilon}}|).
 \end{split}
 \ees 
As a result, uniformly in $x,y\in   \overline{\Omega}_{\bm{\varepsilon}}$, it holds that:
  $$\frac{\P_x\big[X_{\tau_{\overline\Sigma_\ddd^{\bm \varepsilon}}}\in\overline\Sigma_\ddd^{i, \varepsilon_i}\big]}{\P_y\big[X_{\tau_{\overline\Sigma_\ddd^{\bm \varepsilon}}}\in\overline\Sigma_\ddd^{i, \varepsilon_i}\big]} = \frac{1 + \sum_{k=1}^N O\big (\mu_{k,i}^{\bm{\varepsilon}}\big ) + O(|\mathbf{K}^{\bm{\varepsilon}}|) + O(\gamma_i^{\bm \varepsilon}(\delta))  }{1 + \sum_{k=1}^N O\big (\mu_{k,i}^{\bm{\varepsilon}}\big ) + O(|\mathbf{K}^{\bm{\varepsilon}}|)  + O(\gamma_i^{\bm \varepsilon}(\delta)) },$$
  and 
  $$\P_x\big[X_{\tau_{\overline\Sigma_\ddd^{\bm \varepsilon}}}\in\overline\Sigma_\ddd^{i, \varepsilon_i}\big]= \frac{K_i^{\varepsilon_i} }{ \sum_{k=1}^N K_k^{\varepsilon_k}}  [ 1 +  O(|\mathbf{K}^{\bm{\varepsilon}}|)  ]\cdot \Big[1 + \sum_{k=1}^N O\big (\mu_{k,i}^{\bm{\varepsilon}}\big ) + O(|\mathbf{K}^{\bm{\varepsilon}}|)  + O(\gamma_i^{\bm \varepsilon}(\delta)) \Big].$$
 Note that $\sum_{k=1}^N O\big(\mu_{k,i}^{\bm{\varepsilon}}\big)$ dominates $O(|\mathbf{K}^{\bm{\varepsilon}}|)$  since as   $|\bm{\varepsilon}| \to 0$:
$$\sum_{k=1}^N \mu_{k,i}^{\bm{\varepsilon}}/|\mathbf{K}^{\bm{\varepsilon}}|\ge  - \ln (K_i^{\varepsilon_k}) \sum_{k=1}^N \frac{K_k^{\varepsilon_i} }{|\mathbf{K}^{\bm{\varepsilon}}| } \simeq - \ln (K_i^{\varepsilon_k}) \to +\infty.$$
  The proof of the corollary is complete.

\subsection{Proofs of Corollaries  \ref{cor:1} and \ref{cor:2}}

 Recall   \eqref{eq:thmP}.  
Assume that for all $k\in\lb1,N\rb$, $K_k^{\varepsilon_k}$ admits the following expansion
\bes
K_k^{\varepsilon_k} = \frac{a_k}{\ln\varepsilon_1} + \frac{b_k}{|\ln\varepsilon_1|^2} + \frac{O(1)}{|\ln\varepsilon_1|^3}
\ees
for some fixed (namely independent of $\bm \varepsilon$) constants $a_k,b_k\in\R^*$.  Assume moreover that the sum of the $a_k$ is non-zero. Therefore, we have
\be\label{eq:average}
\begin{split}
    \frac{\Ke_i}{\sum_{k=1}^N\Ke_k} &= \frac{\frac{a_i}{\ln\varepsilon_1} + \frac{b_i}{|\ln\varepsilon_1|^2} + \frac{O(1)}{|\ln\varepsilon_1|^3}}{\frac{\sum_{k=1}^Na_k}{\ln\varepsilon_1}\big(1 + \frac{1}{\ln\varepsilon_1}\frac{\sum_{k=1}^Nb_k}{\sum_{k=1}^Na_k} + \frac{O(1)}{|\ln\varepsilon_1|^2}\big)}\\
    &= \frac{a_i}{\sum_{k=1}^Na_k}\Big[1 + \frac{b_i}{a_i\ln\varepsilon_1} + \frac{O(1)}{|\ln\varepsilon_1|^2}\Big]\cdot \Big[1 - \frac{1}{\ln\varepsilon_1}\frac{\sum_{k=1}^Nb_k}{\sum_{k=1}^Na_k} + \frac{O(1)}{|\ln\varepsilon_1|^2}\Big]\\
    &= \frac{a_i}{\sum_{k=1}^Na_k}\Big[1 + \frac{b_i}{a_i\ln\varepsilon_1} - \frac{1}{\ln\varepsilon_1}\frac{\sum_{k=1}^Nb_k}{\sum_{k=1}^Na_k}\Big] + \frac{O(1)}{|\ln\varepsilon_1|^2}
\end{split}
\ee
We are now in position to prove Corollaries \ref{cor:1} and \ref{cor:2}.

  \subsubsection{Proof of Corollary \ref{cor:1}}
  In the regime of Corollary \ref{cor:1} we have  for all $k\in\lb1,N\rb$:
\bes
K_k^{\varepsilon_k} = \frac{-1}{\ln\varepsilon_1 + \ln(\alpha_k/2)} = \frac{-1}{\ln\varepsilon_1} + \frac{\ln(\alpha_k/2)}{|\ln\varepsilon_1|^2} + \frac{O(1)}{|\ln\varepsilon_1|^3}.
\ees
With the notation above, $a_k = -1$ and $ b_k = \ln(\alpha_k/2)$. 
Let $i\in\lb1,N\rb$. By \eqref{eq:average}, we get:
\bes
\frac{\Ke_i}{\sum_{k=1}^N\Ke_k} = \frac1N\Big[1 - \frac{\ln(\alpha_i/2)}{\ln\varepsilon_1} + \frac{1}{N\ln\varepsilon_1}\sum_{k=1}^N\ln(\alpha_k/2)\Big] + \frac{O(1)}{|\ln\varepsilon_1|^2}.
\ees
Note also that 
\bes
K_k^{\varepsilon_k} - K_i^{\varepsilon_i} = \frac{\ln(\alpha_k/\alpha_i)}{|\ln\varepsilon_1|^2} + \frac{O(1)}{|\ln\varepsilon_1|^3},
\ees
 Thus, using  \eqref{eq:thmP}, we deduce that     for all  $\delta\in (0,2)$,  as  $|\bm \varepsilon|\to 0$,  it holds uniformly in  $x\in\overline\Omega_{\bm \varepsilon}$:
\bes
\begin{split}
    \mathbf p_i^{\bm \varepsilon}(x) &= \frac1N - \frac{\ln\alpha_i}{N\ln\varepsilon_1} + \frac{\sum_{k=1}^N\ln \alpha_k}{N^2\ln\varepsilon_1} - \frac1{N^2\ln\varepsilon_1}\sum_{k=1}^N\sum_{j\notin\{i,k\}}\ln\frac{|x^{(j)}-x^{(i)}|}{|x^{(j)}-x^{(k)}|}\\
    &\quad- \frac1{N\ln\varepsilon_1}\sum_{k=1}^N\ln\frac{|x-x^{(k)}|}{|x-x^{(i)}|} - \frac1{N^2\ln\varepsilon_1}\sum_{j,k=1}^Ng_{i,j,k}(x) + \frac{O_{L^\infty(\Omega_{\bm \varepsilon})}(1)}{|\ln\varepsilon_1|^{2-\delta}}\\
    &= \frac1N + \frac{1}{N^2\ln\varepsilon_1}\sum_{k=1}^N\ln\frac{\alpha_k\prod_{j\neq k}|x^{(j)}-x^{(k)}|}{\alpha_i\prod_{j\neq i}|x^{(j)}-x^{(i)}|}\\
    &\quad- \frac{1}{N\ln\varepsilon_1}\sum_{k=1}^N\ln\frac{|x-x^{(k)}|}{|x-x^{(i)}|} - \frac1{N^2\ln\varepsilon_1}\sum_{j,k=1}^Ng_{i,j,k}(x) + \frac{O_{L^\infty(\Omega_{\bm \varepsilon})}(1)}{|\ln\varepsilon_1|^{2-\delta}}.
\end{split}
\ees
The proof of Corollary   \ref{cor:1} is complete.


  \subsubsection{Proof of Corollary \ref{cor:2}}

In the regime of Corollary \ref{cor:2} we have   for all $k\in\lb1,N\rb$,
\bes
K_k^{\varepsilon_k} = \frac{-\beta_k}{\ln\varepsilon_1 - \beta_k\ln2} = \frac{-\beta_k}{\ln\varepsilon_1} - \frac{\beta_k^2\ln2}{|\ln\varepsilon_1|^2} + \frac{O(1)}{|\ln\varepsilon_1|^3}.
\ees
In this case, $a_k = -\beta_k$ and $b_k = -\beta_k^2\ln2$. 
Let $i\in\lb1,N\rb$. Therefore, \eqref{eq:average} yields
\bes
\begin{split}
    \frac{\Ke_i}{\sum_{k=1}^N\Ke_k} &= \frac{\beta_i}{\sum_{k=1}^N\beta_k}\Big[1 + \frac{\beta_i^2\ln2}{\beta_i\ln\varepsilon_1} - \frac{1}{\ln\varepsilon_1}\frac{\sum_{k=1}^N\beta_k^2\ln2}{\sum_{k=1}^N\beta_k}\Big] + \frac{O(1)}{|\ln\varepsilon_1|^2}\\
    &= \frac{\beta_i}{\sum_{k=1}^N\beta_k} + \frac{\ln2}{\ln\varepsilon_1}\frac{\beta_i}{(\sum_{k=1}^N\beta_k)^2}\Big(\sum_{k=1}^N\beta_k{\beta_i} - {\sum_{k=1}^N\beta_k^2}\Big) + \frac{O(1)}{|\ln\varepsilon_1|^2}\\
    &= \frac{\beta_i}{\sum_{k=1}^N\beta_k} + \frac{\ln2}{\ln\varepsilon_1}\frac{\beta_i}{(\sum_{k=1}^N\beta_k)^2}\sum_{k=1}^N\beta_k(\beta_i - \beta_k) + \frac{O(1)}{|\ln\varepsilon_1|^2}.
\end{split}
\ees
Moreover, $K_k^{\varepsilon_k}-K_i^{\varepsilon_i} =  \frac{-\beta_k}{\ln\varepsilon_1}   -  \frac{-\beta_i}{\ln\varepsilon_1}   + O(|\ln\varepsilon_1|^{-2}) = \frac{\beta_i - \beta_k}{\ln\varepsilon_1} + O(|\ln\varepsilon_1|^{-2}) $ and therefore: 
\bes
\frac{K^{\varepsilon_i}_i}{\sum_{j=1}^{N}K_j^{\varepsilon_j}}\frac{K^{\varepsilon_k}_k}{\sum_{j=1}^{N}K_j^{\varepsilon_j}}(K_k^{\varepsilon_k}-K_i^{\varepsilon_i}) = \frac{1}{\ln\varepsilon_1}\frac{\beta_i}{\sum_{j=1}^N\beta_j}\frac{\beta_k}{\sum_{j=1}^N\beta_j}(\beta_i-\beta_k) {+O(|\ln\varepsilon_1|^{-2})}.
\ees
Thus,  using  \eqref{eq:thmP}, we obtain that    for all  $\delta\in (0,2)$,  as  $|\bm \varepsilon|\to 0$,  it holds uniformly in  $x\in\overline\Omega_{\bm \varepsilon}$:
\bes
\begin{split}
    \mathbf p_i^{\bm \varepsilon}(x) &= \frac{\beta_i}{\sum_{k=1}^N\beta_k} + \frac{1}{\ln\varepsilon_1}\frac{\beta_i}{(\sum_{k=1}^N\beta_k)^2}\sum_{k=1}^N\beta_k(\beta_i - \beta_k)\ln(2|x^{(k)}-x^{(i)}|)\\
    &\quad- \frac{1}{\ln\varepsilon_1}\frac{\beta_i}{(\sum_{k=1}^N\beta_k)^2}\sum_{k=1}^N\sum_{j\notin\{i,k\}}\beta_k\beta_j\ln\frac{|x^{(j)}-x^{(i)}|}{|x^{(j)}-x^{(k)}|} - \frac{1}{\ln\varepsilon_1}\frac{\beta_i}{\sum_{k=1}^N\beta_k}\sum_{k=1}^N\beta_k\ln\frac{|x-x^{(k)}|}{|x-x^{(i)}|}\\
    &\quad- \frac{1}{\ln\varepsilon_1}\frac{\beta_i}{(\sum_{k=1}^N\beta_k)^2}\sum_{k=1}^N\sum_{j\notin\{i,k\}}\beta_k\beta_jg_{i,j,k}(x) + \frac{O_{L^\infty(\Omega_{\bm \varepsilon})}(1)}{|\ln\varepsilon_1|^{2-\delta}}.
\end{split}
\ees
The proof of Corollary   \ref{cor:2} is complete. 
 \medskip

\noindent
\textbf{Acknowledgments}. 
This research was  partially funded by the French National Research Agency (ANR) via Project ANR-25-CE40-6875-01 (ANR DySLoS). B. Nectoux  wishes to thank D. Le Peutrec and L. Michel for their insightful discussions regarding the regularity of solutions to mixed boundary value problems.

\

\end{document}